\documentclass[11pt,english]{book} 
\usepackage{ae,aecompl}
\usepackage{titlesec}
\usepackage{titletoc}
\usepackage{multicol}
\usepackage{lettrine} 
\usepackage{babel}
\usepackage{minitoc}
\usepackage{color}
\usepackage{amsmath}
\usepackage{amsfonts}
\usepackage{amsthm}
\usepackage{amssymb}
\usepackage{graphicx}
\usepackage{epsfig}
\usepackage{cite}
\usepackage{latexsym}
\usepackage{esint}
\usepackage{geometry}
\usepackage[unicode=true,
 bookmarks=true,bookmarksnumbered=false,bookmarksopen=false,
 breaklinks=true,pdfborder={0 0 0},backref=false,colorlinks=true]
 {hyperref}
 \addto\shorthandsspanish{\spanishdeactivate{~<>.}}

\hypersetup{pdftitle={Elliptic Equations in High-Contrast Media and Applications},
 pdfauthor={Leonardo Andr\'es Poveda Cuevas},
 pdfsubject={Master Thesis},
 linkcolor=red, urlcolor=magenta, citecolor=blue, pdfstartview={FitH}, hyperfootnotes=false, unicode=true}
\makeatletter
\usepackage{setspace}
\usepackage{fancyhdr}
\usepackage{listings}
\usepackage{lipsum}
\usepackage{courier}

\@ifundefined{showcaptionsetup}{}{%
 \PassOptionsToPackage{caption=false}{subfig}}
\usepackage{subfig}
\makeatother

\newtheorem{theorem}{Theorem}[chapter]

\newtheorem{definition}{Definition}[chapter]
\newtheorem{lemma}{Lemma}[chapter]
\newtheorem{corollary}{Corollary}[chapter]
\newtheorem{remark}{Remark}[chapter]

\global\long\def\dive{\mathrm{div}}
\global\long\def\R{\mathbb{R}}

\global\long\def\geom{\mathrm{geom}}
\global\long\def\Span{\mathrm{Span}}
\global\long\def\supp{\mathrm{supp}}
\global\long\def\loc{\mathrm{loc}}

\global\long\def\diame{\mathrm{diam}}
\global\long\def\const{\mathrm{const}}
\global\long\def\coarse{\mathrm{coarse}}
 
\input Rothdn.fd

\renewcommand{\thechapter}{\arabic{chapter}}
\titleformat{\chapter}[display]
{\bfseries\LARGE}
{\titlerule
\vspace{1pt}
\titlerule[0.9pt] 
\vspace{1pt}%
\titlerule
\vspace{0.5pc}
\filleft \large{\chaptertitlename}  \large \thechapter}
{2pt}
{
\filcenter
}
[\vspace{4pt}%
\titlerule
\vspace{1pt}%
\titlerule
\vspace{1pt}%
]

\newcommand*{\titleGM}{\begingroup 
\hbox{ 
\hspace*{0.01\textwidth} 
\rule{1pt}{\textheight} 
\hspace*{0.01\textwidth} 
\parbox[b]{1.0\textwidth}{ 

{\noindent\Huge\bfseries Elliptic Equations in High-Contrast Media and Applications}\\[2\baselineskip] 
{\LARGE Ecuaciones El\'ipticas en Medios de Alto Contraste y Aplicaciones}\\[2\baselineskip]
{\large \textit{Master Thesis}}\\[6\baselineskip] 
{\Large \textsc{Leonardo Andr\'es Poveda Cuevas}}\\ 

\vspace{0.3\textheight}
\hspace{1cm}\textbf{\includegraphics[scale=0.15]{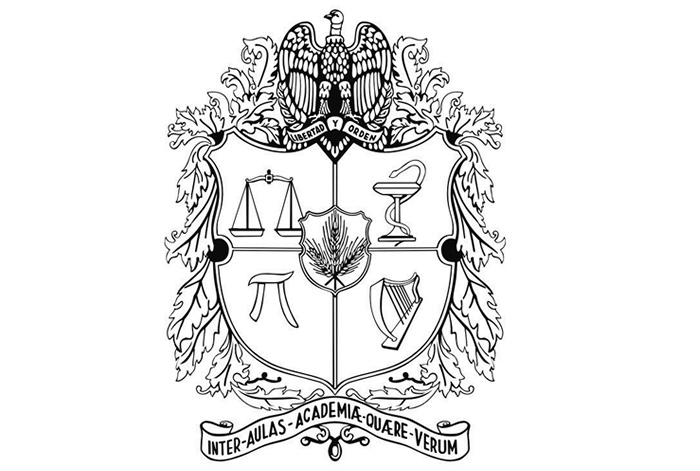}}

{Departamento de Matem\'aticas\\
\bf{Universidad Nacional de Colombia}}\\[\baselineskip] 
}}
\endgroup}

\title{\bfseries Elliptic Equations in High-Contrast Media and Applications}
\author{Leonardo Andr\'es Poveda Cuevas}

\begin{document}

\pagenumbering{roman}\setcounter{page}{1}
\newpage\thispagestyle{empty}  

\titleGM 


\newpage\thispagestyle{empty}\vspace*{1cm}

\begin{center}

\par\end{center}{\large \par}

\newpage\thispagestyle{empty}\vspace*{1cm}

\begin{center}
\textbf{\large{ELLIPTIC EQUATIONS IN HIGH-CONTRAST MEDIA AND APPLICATIONS}}
\par\end{center}{\large \par}
\vspace*{2cm}
\begin{center}
\sc\large{ECUACIONES EL\'IPTICAS EN MEDIOS DE ALTO CONTRASTE Y APLICACIONES}
\par\end{center}{\large \par}

\vfill

\begin{center}
\textbf{LEONARDO ANDR\'ES POVEDA CUEVAS}
\par\end{center}

\begin{center}
\textbf{\vfill }
\par\end{center}

%

\hspace{7cm}%
\begin{minipage}[b]{0.5\columnwidth}%
\begin{singlespace}
Thesis presented to the Graduate Program in Mathematics at the National University of Colombia to obtain the Degree of Master of Science.\\

Adviser: Ph.D. Juan Galvis
\end{singlespace}

\end{minipage}


\vfill
\begin{center}
\textbf{\includegraphics[scale=0.15]{Figures/unal.png}}
\par\end{center}

\begin{center}
\textbf{UNIVERSIDAD NACIONAL DE COLOMBIA}\\
\textbf{FACULTAD DE CIENCIAS}\\
\textbf{DEPARTAMENTO DE MATEM\'ATICAS}\\
\textbf{BOGOT\'A}\\
\textbf{2014}
\par\end{center}

\newpage\thispagestyle{empty}
\vspace{2cm}
\hspace{5cm}%

\begin{flushright}
To my Family.\\
\textsl{``The persistence and dedication are the hope of a new dawn''}
\par\end{flushright}
\hspace{5cm}
\begin{flushright}
To my Girlfriend.\\
\textsl{``Her warm care allow to extend the dream of our live''}
\par\end{flushright}


\newpage\thispagestyle{empty}

\vspace*{4pt}
\begin{center}
\textbf{ACKNOWLEDGMENTS}
\par\end{center}

%

I am thankful to Professor Juan Carlos, my adviser, for his great and patient guidance. I received from him the initial idea, suggestions and scientific conversations for this manuscript. Also I received his example, not only as a scientist but as a person.\\

I am grateful to the most important people in my life: my parents, Aura Rosa and Fredy for their encouragement and financial support; my girlfriend, Paola for her great love and understanding of my new successes, and my brothers Jackson, Alejandro and Laura for their great trust and company at different times of our life.\\

I express my thanks to my uncle Olmer for bearing my presence and for our extraordinary discussions in these last years.\\

I also want to thank to the Departament of Mathematics at National University of Colombia.

\newpage\thispagestyle{empty}

\hspace{9cm}\vspace{12cm}

\hspace{9cm}%
\begin{minipage}[b]{0.35\columnwidth}%
\begin{singlespace}
\lettrine{I}{f nature} were not beautiful, it would not be worth knowing, and if nature were not worth knowing, life would not be worth living.
\end{singlespace}

\begin{flushright}
\texttt{Henri Poincar\'e}
\par\end{flushright}%
\end{minipage}

\clearpage

\pagestyle{fancy}

\pagenumbering{roman}\setcounter{page}{6}

\dominitoc
\tableofcontents

\clearpage\pagenumbering{arabic} \setcounter{page}{10}

\pagestyle{fancy}


\chapter{Introduction}

The mathematical and numerical analysis for partial differential equations in multi-scale and high-contrast media are important  in many practical applications. In fact, in many applications related to fluid flows in porous media the coefficient models the permeability (which represents how easy the porous media let the fluid flow). The values of the  permeability usually, and specially for complicated porous media, vary in several orders of magnitude. We refer to this case as  the high-contrast case  and we say that the corresponding elliptic equation that models (e.g. the pressure) has high-contrast coefficients, see for instance \cite{galvis2014spec, Efendiev2014generalized, efendiev2012coarse}.\\

A fundamental purpose is to understand the effects on the solution related to the variations of high-contrast in the properties of the porous media. In terms of the model, these variations appear in the coefficients of the differential equations. In particular, this interest rises the importance of computations of numerical solutions. For the case of elliptic problems with discontinuous coefficients and bounded contrast, Finite Element Methods and Multiscale methods have been helpful to understand numerically solutions of these problems.\\

In the case  of high-contrast (not necessarily bounded contrast), in order to devise efficient numerical strategies  it is important to understand the behavior of solutions of these equations. Deriving and manipulating asymptotic expansions representing solutions certainly help in this task. The nature of the asymptotic expansion will reveal properties of the solution. For instance, we can compute functionals of solutions and single out its behavior with respect to the contrast or other important parameters.\\

Other important application of elliptic equations correspond to elasticity problems. In particular, linear elasticity equations model the equilibrium and the local deformation of deformable bodies. For instance, an elastic body subdues to local external forces. In this way, we recall these local  forces as strains. The local relationship between strains and deformations is the main problem in the solid mechanics, see  \cite{MR936420,MR1477663,MR0010851,MR0075755}. Under mathematical development these relations are known as Constitutive Laws, which depend on the material, and the process that we want to model.\\

In the case of having composite materials, physical properties such as the Young's modulus can vary in several orders of magnitude. In this setting, having asymptotic series to express solutions it is also useful tool to understand the effect of the high-contrast in the solutions, as well as the interaction among different materials.\\

In this manuscript we study asymptotic expansions of high-contrast elliptic problems. In Chapter \ref{Chapter2} we recall the weak formulation and make the overview of the derivation for high-conductivity inclusions. In addition, we explain in detail the derivation of asymptotic expansions for the one dimensional case with one concentric inclusion, and we develop a particular example. In Chapter \ref{Chapter3} we use the Finite Element Method, which is applied in the numerical computation of terms of the asymptotic expansion. We also present an application to Multiscale Finite Elements, in particular, we numerically design approximation for the term $u_0$ with local harmonic characteristic functions. Chapter \ref{Chapter4} is dedicated to the case of the linear elasticity, where we study the asymptotic problem with one inelastic inclusion and we provide the convergence for this expansion problem. Finally, in Chapter \ref{Chapter5} we state some conclusions and final comments.



\chapter{Asymptotic Expansions for the Pressure Equation}\label{Chapter2}

\minitoc

\section{Introduction}\label{Section2.1}

In this chapter we detail the derivation of asymptotic expansions for high-contrast elliptic problems of the form,
\begin{equation} \label{Strongform}
-\dive(\kappa(x)\nabla u(x))=f(x),\ \mbox{  in }  D,  
\end{equation}
with Dirichlet data defined by $u=g$ on $\partial D$. We assume that $D$ is the disjoint union of a background domain and inclusions, $D=D_0\cup (\bigcup_{m=1}^M \overline{D}_m)$. We assume that $D_0,D_1,\dots,D_M,$ are polygonal domains (or domains with smooth boundaries). We also assume that each $D_m$ is a connected domain, $m=0,1,\dots,M$. Additionally, we assume that $D_m$ is compactly included in the open set $D\setminus 
\bigcup_{\ell=1, \ell \not=m}^M\overline{D}_\ell$, i.e., 
$\overline{D}_m\subset D\setminus \bigcup_{\ell=1, \ell\not=m}^M \overline{D}_\ell$, and we define $D_0:= D\setminus 
\bigcup_{m=1}^M\overline{D}_m$. Let $D_0$ represent the background domain and the sub-domains $\{D_m\}_{m=1}^M$ represent the inclusions. For simplicity of the presentation  we consider only interior inclusions. Other cases can be study similarly.\\

We consider a coefficient with 
multiple high-conductivity inclusions. Let $\kappa$ be defined by 
\begin{equation}\label{eq:coeff1inc-multiple}
\kappa(x)=\left\{\begin{array}{cc} 
\eta,& x\in   D_m, ~~m=1,\dots,M, \\
1,& x\in  D_0=D\setminus \bigcup_{m=1}^M\overline{D}_m.
\end{array}\right.
\end{equation}

We seek to determine $\{u_j\}_{j=0}^\infty\subset 
H^1(D)$ such that
\begin{equation}\label{expansionu_eta}
u_\eta=u_0+\frac{1}{\eta}u_1+\frac{1}{\eta^2}u_2+\dots=
\sum_{j=0}^\infty \eta^{-j} u_{j}, 
\end{equation}
and such that they satisfy the following Dirichlet boundary 
conditions, 
\begin{equation}\label{Boundarycondi}
u_0=g \mbox{ on } \partial D \quad \mbox{ and }
\quad u_{j}=0 \mbox{ on } \partial D \mbox{ for } j\geq 1.
\end{equation}
This work complements current work in the use of similar expansions to design and analyze efficient numerical approximations of high-contrast elliptic equations; see \cite{Efendiev2014generalized,MR2477579}.\\

We have the following weak formulation of problem \eqref{Strongform}: find $u\in H^{1}(D)$ such that 
\begin{equation}\label{eq:problem}
\left\{\begin{array}{ll}
{\cal A}(u,v)={\cal F}(v), & \mbox{ for all } v\in H_0^1(D),\\
\hspace{.3in} u=g, &\mbox{ on } \partial D.
\end{array}\right.
\end{equation}
The bilinear form ${\cal A}$ and the linear functional ${\cal F}$ are defined by
\begin{align}\label{eq:def:a}
{\cal A}(u,v)&=\int_D 
\kappa(x)\nabla u(x)\cdot  \nabla v(x)dx,  
 &&\mbox{ for all }  u,v\in H_0^1(D)\\ \nonumber
{\cal F}(v)&=\int_Df(x)v(x)dx, &&\mbox{ for all } v\in H_0^1(D). 
\end{align}
For more details of weak formulation, see Section \ref{SectionA.3}. We denote by $u_{\eta}$ the solution of the problem  \eqref{eq:problem} with Dirichlet boundary condition  \eqref{Boundarycondi}. From now on, we use the notation $w^{(m)}$, which means that the function $w$ is restricted on domain $D_{m}$, that is $w^{(m)}=w|_{D_m},\,m=0,1,\dots,M.$

\section{Overview of the Iterative Computation of Terms in Expansions}

In this section we summarize the procedure to obtain the terms in the expansion. We present a detailed derivation of the expansion in Section \ref{sec:2.3} for the one dimensional case.

\subsection{Derivation for One High-Conductivity Inclusion}\label{Sec:2.2.1}

Let us denote by $u_\eta$ the solution of~\eqref{eq:problem} with the coefficient $\kappa(x)$ defined in \eqref{eq:coeff1inc-multiple}. We express  the expansion as in \eqref{expansionu_eta} with the functions $u_{j}$, $j=0,1,\dots,$ satisfying the conditions on 
the boundary of $D$ given in \eqref{Boundarycondi}. 
For the case of the one inclusion with $m=0,1$, we consider $D$ as the disjoint union of a background domain $D_0$ and one inclusion $D_1$, such that $D=D_0\cup \overline{D}_1$.
We assume that 
$D_1$ is compactly included in $D$ ($\overline{D}_1\subset D$).\\

We obtain the following development for each term of this asymptotic expansion in the problem \eqref{eq:problem}. First we replace $u(x)$ for the expansion \eqref{expansionu_eta} in the bilinear form  \eqref{eq:def:a}, we have
\begin{eqnarray*}
\int_D\kappa(x)\sum_{j=0}^{\infty}\eta^{-j}\nabla u_{j}\cdot \nabla v & = & \int_{D_0}\sum_{j=0}^{\infty}\eta^{-j}\nabla u_{j}\cdot \nabla v+\eta\int_{D_1}\sum_{j=0}^{\infty}\eta^{-j}\nabla u_{j}\cdot \nabla v \\
& = & \sum_{j=0}^{\infty}\eta^{-j}\int_{D_0}\nabla u_{j}\cdot \nabla v+\eta\sum_{j=0}^{\infty}\eta^{-j}\int_{D_1}\nabla u_{j}\cdot \nabla v\\
& = & \sum_{j=0}^{\infty}\eta^{-j}\int_{D_0}\nabla u_{j}\cdot \nabla v+\sum_{j=0}^{\infty}\eta^{-j+1}\int_{D_1}\nabla u_{j}\cdot \nabla v,\\
\end{eqnarray*}
we change the index in the last sum to obtain,
\[
\sum_{j=0}^{\infty}\eta^{-j}\int_{D_0}\nabla u_{j}\cdot \nabla v+\sum_{j=-1}^{\infty}\eta^{-j}\int_{D_1}\nabla u_{j+1}\cdot \nabla v,
\]
and then, we have
\[
\sum_{j=0}^{\infty}\eta^{-j}\int_{D_0}\nabla u_{j}\cdot \nabla v+\eta\int_{D_1}\nabla u_0\cdot \nabla v+\sum_{j=0}^{\infty}\eta^{-j}\int_{D_1}\nabla u_{j+1}\cdot \nabla v.
\]
We obtain,
\begin{equation}
\eta \int_{D_{1}}\nabla u_{0}\cdot \nabla v+\sum_{j=0}^\infty \eta^{-j}\left(\int_{D_{0}}\nabla u_{j}\cdot \nabla v+\int_{D_{1}}\nabla u_{j+1}\cdot \nabla v\right)=\int_{D}fv.   
\end{equation}
In brief, we obtain the following equations after matching equal powers,
\begin{equation}\label{arreglo1}
\int_{D_{1}}\nabla u_{0}\cdot \nabla v = 0,
\end{equation}
\begin{equation}\label{arreglo2}
\int_{D_{0}}\nabla u_{0}\cdot \nabla v+\int_{D_{1}}\nabla u_{1}\cdot \nabla v=\int_{D}fv,
\end{equation}
and for $j\geq 1$, 
\begin{equation}\label{arreglo3}
\int_{D_{0}}\nabla u_{j}\cdot \nabla v+\int_{D_{1}}\nabla u_{j+1}\cdot \nabla v=0,
\end{equation}
for all $v\in H_{0}^{1}(D)$.\\

The equation \eqref{arreglo1} tells us that the function $u_{0}$ restricted to $D_{1}$ is constant, that is $u_{0}^{(1)}$ is a constant function. We introduce the following subspace,
\[
V_{\const}=\{ v\in H_{0}^{1}(D), \mbox{ such that } v^{(1)}=v|_{D_{1}} \mbox { is constant} \}.
\]
If in the equation \eqref{arreglo2} we choose test function  $z\in V_{\const}$, then, we see that $u_{0}$ satisfies the problem
\begin{equation}\label{zVconst}
\begin{array}{ll}
\int_{D}\nabla u_{0}\cdot \nabla z= \int_{D}fz,& \mbox{for all } z\in V_{\const}\\
\hspace{0.65in} u_{0}=g,& \mbox{on } \partial D.
\end{array}
\end{equation}
The problem \eqref{zVconst} is elliptic and it has a unique solution. This follows from the ellipticity of bilinear form ${\cal A}$, see Section \ref{Weakforsection}.\\

We introduce the \emph{harmonic characteristic function} $\chi_{D_{1}}\in H_{0}^{1}(D)$ with the condition
\[
\chi_{D_{1}}^{(1)}=1, \mbox{ in }D_{1},
\]
and which is equal to the harmonic extension of its boundary data in $D_{0}$ (see \cite{calo2014asymptotic}).
We then have,
\begin{align}\label{Chioneinclusion}
\displaystyle \int_{D_0} \nabla \chi_{D_{1}}^{(0)}\cdot \nabla z &=0, &&\mbox{ for all }  z\in H^1_0(D_0), \\
\chi_{D_{1}}^{(0)}&=1, &&\mbox{ on } \partial D_1,\nonumber\\
\chi_{D_{1}}^{(0)}&=0, &&\mbox{ on } \partial D. \nonumber
\end{align}
To obtain an explicit formula for $u_0$ we use the facts that the problem \eqref{zVconst} is elliptic and has unique solution, and a property of the harmonic characteristic functions described in the next Remark.

\begin{remark}\label{RemarkNew}
Let $w$ be a harmonic extension to $D_0$ of its Neumann data on $\partial D_0$. That is, $w$ satisfy the following problem
\[
\int_{D_0}\nabla w\cdot \nabla z=\int_{\partial D_0}\nabla w \cdot n_0z, \quad \mbox{for all }z\in H^1(D_0).
\]
Since $\chi_{D_1}=0$ on $\partial D$ and $\chi_{D_1}=1$ on $\partial D_1$, we have that
\[
\int_{D_0}\nabla \chi_{D_1}\cdot \nabla w=\int_{\partial D_0}\nabla w\cdot n_0\chi_{D_1}=0\left(\int_{\partial D}\nabla w\cdot n\right)+1\left(\int_{\partial D_1}\nabla w \cdot n_0\right),
\]
and we conclude that for every harmonic function on $D_0$,
\[
\int_{D_0}\nabla \chi_{D_1}\cdot \nabla w=\int_{\partial D_1}\nabla w\cdot n_0.
\]
Note that if $\xi \in H^1(D)$ is such that $\xi^{(1)}=\xi|_{D_1}=c$ is a constant in $D_1$ and $\xi^{(0)}=\xi|_{D_0}$ is harmonic in $D_0$, then $\xi=c\chi_{D_1}$.
\end{remark}
An explicit formula for $u_{0}$ is obtained as
\begin{equation}\label{eq:formulau0}
u_{0}=u_{0,0}+c_{0}\chi _{D_{1}},
\end{equation}
where $u_{0,0}\in H^{1}(D)$ is defined by $u_{0,0}^{(1)}=0$ in $D_{1}$ and $u_{0,0}^{(0)}$ solves the Dirichlet problem
\begin{align}
\displaystyle \int_{D_0} \nabla u_{0,0}^{(0)}\cdot \nabla z &=\int_{D_0}fz, &&\mbox{ for all } z\in H^1_0(D_0). \\
u_{0,0}^{(0)}&=0, &&\mbox{ on } \partial D_1,\nonumber\\
u_{0,0}^{(0)}&=g, &&\mbox{ on } \partial D. \nonumber
\end{align}
From equation \eqref{zVconst}, \eqref{eq:formulau0} and  the facts in the Remark \ref{RemarkNew} we have
\[
\int_{D_0}\nabla u_0\cdot \nabla \chi_{D_1}  =\int_Df\chi_{D_1},
\]
or
\[
\int_{D_0}\nabla(u_{0,0}+c_0\chi_{D_1})\cdot \nabla \chi_{D_1} =\int_Df\chi_{D_1},
\]
from which we can obtain the constant $c_{0}$ given by 
\begin{equation}
c_{0}=\frac{\int_Df\chi_{D_1}-\int_{D_{0}}\nabla u_{0,0}\cdot \nabla \chi_{D_{1}}}{\int_{D_{0}}|\nabla \chi_{D_{1}}|^{2}},
\end{equation}
or, using the Remark above we also have
\begin{equation} \label{c-u00}
c_{0}=\frac{\int_{D_{1}}f-\int_{\partial D_{1}}\nabla u_{0,0}\cdot n_0}{\int_{\partial D_{1}}\nabla \chi_{D_{1}}\cdot n_0}.
\end{equation}
Thus, by \eqref{c-u00} $c_{0}$ balances the fluxes across $\partial D_1$, see \cite{calo2014asymptotic}.\\

\begin{figure}[!h]
\begin{centering}
\includegraphics[scale=0.35]{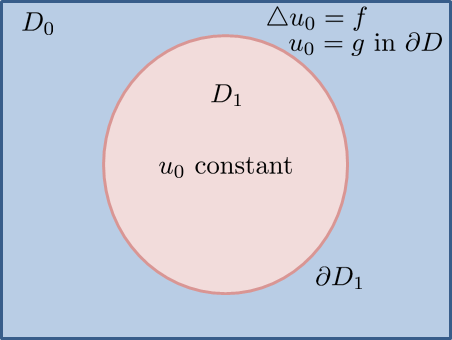}
\caption{Illustration of local problems related to $u_{0}$  in the inclusion and the background.} \label{Figure01}
\end{centering}
\end{figure}

In Figure  \ref{Figure01} we illustrate the properties of $u_{0}$, that is, the local problem $u_{0}$ solves in the inclusions and the background.\\

To get the function $u_{1}$ we proceed as follows. We first write
\[
u_1^{(1)}=\widetilde{u}_1^{(1)}+c_1, \quad \mbox{ where } \int_{D_1}\widetilde{u}_1^{(1)}=0,
\]
and $\widetilde{u}_1^{(1)}$ solves the Neumann problem
\begin{equation}\label{eq:Neumanu1}
\int_{D_1}\nabla \widetilde{u}_1^{(1)}\cdot \nabla z=\int_{D_1}fz -
\int_{\partial D_1}\nabla u_0^{(0)}\cdot n_1  z,
\quad \mbox{ for all } z\in H^1(D_1). 
\end{equation}
The detailed deduction of this Neumann problem will be presented later in Section \ref{sec:2.3}. Problem~\eqref{eq:Neumanu1}
satisfies the compatibility condition so it is solvable. The constant $c_1$ is given by
\begin{equation}\label{eq:def:ci}
c_1=-\frac{\int_{\partial D_1} \nabla \widetilde{u}_{1}^{(0)}\cdot n_0  }{
\int_{\partial D_1} \nabla \chi_{D_1}^{(0)}\cdot n_0}=
-\frac{\int_{D_0}\nabla \widetilde{u}_1 \cdot\nabla \chi_{D_1} }{
\int_{D_0} |\nabla \chi_{D_1}|^2}. 
\end{equation}

\begin{figure}[!h]
\begin{centering}
\includegraphics[scale=0.35]{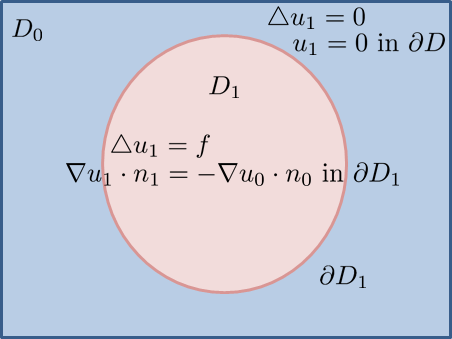}
\caption{Illustration of local problems related to $u_{1}$  in the inclusions and the background. First a problem in $D_1$ is solved and then a problem in $D_0$ is solved using the boundary data in $\partial D_1$.} \label{Figure02}
\end{centering}
\end{figure}

In the Figure \ref{Figure02} we illustrate the properties of $u_{1}$, that is, the local problem $u_{1}$ solves in the inclusion and the background. Now, we complete the construction of $u_1$ and show how to construct $u_j$, $j=2,3,\dots$. For a given $j=1,2,\dots$ and given $u_j^{(1)}$ we show how to construct $u_j^{(0)}$ and $u^{(1)}_{j+1}$.\\

Assume that we already constructed  $u_j^{(1)}$ in $D_1$ as the solution of a Neumann problem in 
$D_1$ and can be written as
\begin{equation}
u_j^{(1)}=\widetilde{u}_j^{(1)}+c_j, \quad \mbox{ where } \int_{D_1}
\widetilde{u}_j=0.
\end{equation}
We find $u_j^{(0)}$ in $D_0$ by solving a Dirichlet problem with known Dirichlet data, that is, 
\begin{equation}\label{eq:DirforuiD0}
\begin{array}{l}
\displaystyle \int_{D_0}\nabla u_{j}^{(0)}\cdot \nabla z=0, \mbox{ for all } 
z\in H^1_0(D_0), \\\\
u_j^{(0)}=u_j^{(1)} \,\left(=\widetilde{u}_j^{(1)}+c_j\right) \mbox{ on } 
\partial D_1
\quad \mbox{ and } \quad u_j=0 \mbox{ on } \partial D. 
\end{array} 
\end{equation}

Since $c_j$, $j=1,2,\dots,$  are constants, their harmonic extensions 
are  given by $c_j\chi_{D_1}$, $j=1,2,\dots$.  Then,  we conclude that 
\begin{equation}
u_j=\widetilde{u}_j+c_j\chi_{D_1},
\end{equation}
where $\widetilde{u}_j^{(0)}$ is defined by~\eqref{eq:DirforuiD0}. 
The balancing constant $c_j$ is given as
\begin{equation}\label{eq:def:ci}
c_j=-\frac{\int_{\partial D_1} \nabla \widetilde{u}_{j}^{(0)}\cdot n_0  }{
\int_{\partial D_1} \nabla \chi_{D_1}^{(0)}\cdot n_0}=
-\frac{\int_{D_0}\nabla \widetilde{u}_j \cdot\nabla \chi_{D_1} }{
\int_{D_0} |\nabla \chi_{D_1}|^2}, 
\end{equation}
so we have $\int_{\partial D_1 }\nabla 
u_{j}^{(0)}\cdot n_0=0$.
This completes the construction of $u_{j}$. 
The compatibility conditions are satisfied, so the problem can be to solve \eqref{eq:def:ci}. This will be detailed in the one dimensional case in Section \ref{sec:2.3}. In the Figure \ref{Figure03} we 
illustrate the properties of $u_{j}$, that is, the local problem $u_{j}$ solves in the inclusions and the 
background.\\

\begin{figure}[!h]
\begin{centering}
\includegraphics[scale=0.35]{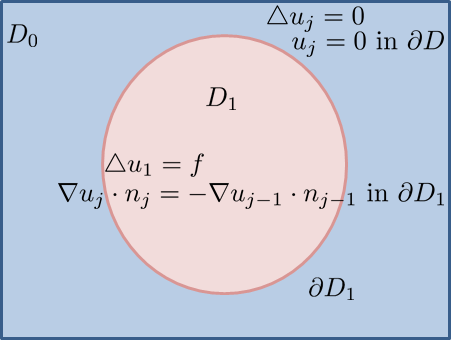}
\caption{Illustration of local problems related to $u_{j}$ , $j\geq 2$, in the inclusion and the background. First a problem in $D_1$ is solved and then a problem in $D_0$ is solved using the boundary data in $\partial D_1$.} \label{Figure03}
\end{centering}
\end{figure}

We recall the convergence result obtained in \cite{calo2014asymptotic}. 
There it is proven that there is a constant $C>0$ such that for every 
$\eta>C$, the expansion~\eqref{expansionu_eta} converges 
(absolutely) in $H^1(D)$. The asymptotic limit 
$u_0$ satisfies~\eqref{eq:formulau0}. Additionally,
 for every 
$\eta>C$, we have 
\[
\left\Vert u-\sum_{j=0}^J \eta^{-j}u_{j}\right\Vert _{H^1(D)}\leq C_1
\left(\|f\|_{H^{-1}(D_1)}
+\|g\|_{H^{1/2}(D)}\right)\sum_{j=J+1}^\infty\left(\frac{C}{\eta}\right)^{j},
\]
for $J\geq 0$.  For the proof we refer to \cite{calo2014asymptotic}. See also Section \ref{Sec:2.2.2} and Chapter \ref{Chapter4} were we extended the analysis for the case of high-contrast elliptic problems.

\subsection{Derivation for Multiple High-Conductivity Inclusions}\label{Sec:2.2.2}

In this section we express the expansion as in \eqref{expansionu_eta} with the coefficient 
\[
\kappa(x)=\left\{\begin{array}{cc} 
\eta,& x\in   D_m, ~~m=1,\dots,M, \\
1,& x\in  D_0=D\setminus \bigcup_{m=1}^M\overline{D}_m.
\end{array}\right.
\]
First, we describe the asymptotic problem.\\

We recall the space of constant functions inside the inclusions
\[
V_{\const}=\left\{v\in H_0^1(D), \mbox{ such that }v|_{D_m} \mbox{ is a constant for all }m=1,\dots,M\right\}.
\]

By analogy with the case of one high-conductivity inclusion, if we choose test function  $z\in V_{\const}$, then, we see 
that $u_{0}$ satisfies the problem
\begin{equation}\label{zVconstmultiply}
\begin{array}{ll}
\int_{D}\nabla u_{0}\cdot \nabla z= \int_{D}fz,& \mbox{for all } z\in V_{\const}\\
\hspace{0.65in} u_{0}=g,& \mbox{on } \partial D.
\end{array}
\end{equation}
The problem \eqref{zVconst} is elliptic and it has a unique solution. For each $m=1,\dots, M$ we introduce the harmonic characteristic function $\chi_{D_{m}}\in H_{0}^{1}(D)$ with the condition
\begin{equation}\label{chi multi}
\chi_{D_{m}}^{(1)}\equiv\delta_{m\ell} \mbox{ in }D_{\ell},\mbox{ for }\ell=1,\dots,M,
\end{equation}
and which is equal to the harmonic extension of its boundary data in $D_{0}$, $\chi_{D_m}$ (see \cite{calo2014asymptotic}).\\

We then have,
\begin{align}\label{Chimultiinclusions}
\int_{D_0} \nabla \chi_{D_{m}}\cdot \nabla z &=0, &&\mbox{ for all }  z\in H_0^1(D_0).\\
\chi_{D_{m}}&=\delta_{m\ell}, &&\mbox{ on } \partial D_{\ell},\mbox{ for }\ell=1,\dots,M,\nonumber\\
\chi_{D_{m}}&=0, &&\mbox{ on } \partial D. \nonumber
\end{align}
Where $\delta_{m\ell}$ represent the Kronecker delta, which is equal to $1$ when $m=\ell$ and $0$ otherwise. To obtain an explicit formula for $u_0$ we use the facts that the problem \eqref{zVconst} is elliptic and has unique solution, and a similar property of the harmonic characteristic functions described in the Remark \ref{RemarkNew}, we replace $\chi_{D_1}$ by $\chi_{D_m}$ for this case.\\

We decompose $u_0$ into the harmonic extension to $D_0$ of a function in $V_{\const}$ plus a function $u_{0,0}$ with value $g$ on the boundary and zero boundary condition on $\partial D_m$, $m=1,\dots,M$. We write
\[
u_0=u_{0,0}+\sum_{m=1}^{M}c_m(u_0)\chi_{D_m},
\]
where $u_{0,0}\in H^1(D)$ with $u_{0,0}=0$ in $D_m$ for $m=1,\dots,M$, and $u_{0,0}$ solves the problem in $D_0$
\begin{equation}\label{Def:u00}
\int_{D_0}\nabla u_{0,0}\cdot \nabla z=\int_{D_0}fz, \quad\mbox{for all }z\in H_0^1(D_0),
\end{equation}
with $u_{0,0}=0$ on $\partial D_m$, $m=1,\dots,M$ and $u_{0,0}=g$ on $\partial D$. We compute the constants $c_m$ using the same procedure as before. We have
\[
\int_{D_0}\nabla\left(u_{0,0}+\sum_{m=1}^{M}c_m(u_0)\chi_{D_m}\right)\cdot \nabla \chi_{D_{\ell}}=\int_Df\chi_{D_{\ell}},\quad\mbox{for }\ell=1,\dots,M,
\]
or
\[
\sum_{m=1}^{M}c_m(u_0)\int_{D_0}\nabla \chi_{D_m}\cdot \nabla\chi_{D_{\ell}}=\int_{D}f\chi_{D_{\ell}}-\int_{D_0}\nabla u_{0,0}\cdot \nabla \chi_{D_{\ell}},
\]
and this last problem is equivalent to the linear system,
\[
\mathbf{A}_{\geom}\mathbf{X}=\mathbf{b},
\]
where $\mathbf{A}_{\geom}=[a_{m\ell}]$ and $\mathbf{b}=(b_1,\dots,b_M)\in \R^M$ are defined by
\begin{align}\label{Def:aml}
a_{m\ell}&=\int_{D}\nabla \chi_{D_m}\cdot \nabla \chi_{D_{\ell}}=\int_{D_0}\nabla \chi_{D_m}\cdot \nabla\chi_{D_{\ell}},&&\\
b_{\ell}&=\int_{D}f\chi_{D_{\ell}}-\int_{D_0}\nabla u_{0,0}\cdot \nabla\chi_{D_{\ell}},&&\label{Def:bl}
\end{align}
and $\mathbf{X}=(c_1(u_0),\dots,c_M(u_0))\in \R^M$. Then we have
\begin{equation}\label{Def:X}
\mathbf{X}=\mathbf{A}_{\geom}^{-1}\mathbf{b}.
\end{equation}
Now using the conditions given for $\chi_{D_m}$ in the Remark \ref{RemarkNew} we have
\begin{equation}\label{Def:aml2}
a_{m\ell}=\int_{D}\nabla \chi_{D_m}\cdot \nabla \chi_{D_{\ell}}=\int_{\partial D_m}\chi_{D_{\ell}}\cdot n_m=\int_{\partial D_{\ell}}\nabla \chi_{D_{\ell}}\cdot n_{\ell}.
\end{equation}
Note that $\sum_{m=1}^{M}c_m(u_0)\chi_{D_m}$ is the solution of a Galerkin projection in the space $\Span\left\{\chi_{D_m}\right\}_{m=1}^{M}$. For more details see \cite{calo2014asymptotic} and references there in.\\

Now, we describe the next individual terms of the asymptotic expansion. As before, we have the restriction of $u_1$ to the sub-domain $D_m$, that is
\[
u_1^{(m)}=\widetilde{u}_1^{(m)}+c_{1,m},\quad\mbox{with }\int_{D_m}\widetilde{u}_1^{(m)}=0,
\]
and $\widetilde{u}_1^{(m)}$ satisfies the Neumann problem
\[
\int_{D_m}\nabla \widetilde{u}_1^{(m)}\cdot \nabla z=\int_{D_m}fz-\int_{\partial D_m}\nabla u_0^{(0)}\cdot n_{m}z,\mbox{ for all }z\in H^1(D_m),
\]
for $m=1,\dots,M$. The constants $c_{1,m}$ will be chosen later.\\

Now, for $j=1,2,\dots$ we have that $u_j^{(m)}$ in $D_m$, $m=1,\dots,M$, then we find $u_j^{(0)}$ in $D_0$ by solving the Dirichlet problem
\begin{align}\label{Def:ujDiric}
\int_{D_0}\nabla u_{j}^{(0)}\cdot \nabla z&=0, &&\mbox{ for all }z\in H_0^1(D_0)\\\nonumber
u_j^{(0)}&=u_j^{(m)}\, \left(=\widetilde{u}_j^{(m)}+c_{j,m}\right), &&\mbox{on }\partial D_m,\, m=1,\dots,M,\\\nonumber
u_j^{(0)}&=0, &&\mbox{on }\partial D.\nonumber
\end{align}
Since $c_{j,m}$ are constants, we define their corresponding harmonic extension by $\sum_{m=1}^{M}c_{j,m}\chi_{D_m}$. So we rewrite
\begin{equation}\label{Def:uj}
u_j=\widetilde{u}_j+\sum_{m=1}^{M}c_{j,m}\chi_{D_m}.
\end{equation}
The $u_{j+1}^{(m)}$ in $D_m$ satisfy the following Neumann problem
\[
\int_{D_m}\nabla u_{j+1}^{(m)}\cdot \nabla z=-\int_{\partial D_m}\nabla u_j^{(0)}\cdot n_0z,\quad\mbox{for all }z\in H^1(D).
\]
For the compatibility condition we need that for $\ell=1,\dots, M$
\begin{eqnarray*}
0  = \int_{\partial D_{\ell}}\nabla u_{j+1}^{(\ell)}\cdot n_{\ell} & = & -\int_{\partial D_{\ell}}\nabla u_j^{(0)}\cdot n_0\\
& = & -\int_{\partial D_{\ell}}\nabla \left(\widetilde{u}_j^{(0)}+\sum_{m=1}^{M}c_{j,m}\chi_{D_m}^{(0)}\right)\cdot n_0\\
& = & -\int_{\partial D_{\ell}}\nabla \widetilde{u}_j^{(0)}\cdot n_0-\sum_{m=1}^{M}c_{j,m}\int_{\partial D_m}\nabla \chi_{D_m}^{(0)}\cdot n_0.
\end{eqnarray*}
From \eqref{Def:aml} and \eqref{Def:aml2} we have that $\mathbf{X}_j=(c_{j,1},\dots,c_{j,M})$ is the solution of the system
\[
\mathbf{A}_{\geom}\mathbf{X}_j=\mathbf{Y}_j,
\]
where
\[
\mathbf{Y}_j=\left(-\int_{\partial D_1}\nabla \widetilde{u}_{j}^{(0)}\cdot n_0,\dots ,-\int_{\partial D_m}\nabla \widetilde{u}_j^{(0)}\cdot n_0\right),
\]
or
\[
\mathbf{Y}_j=\left(-\int_{D_0}\nabla \widetilde{u}_{j}^{(0)}\cdot \nabla \chi_{D_1},\dots ,-\int_{D_0}\nabla \widetilde{u}_j^{(0)}\cdot \nabla \chi_{D_M}\right).
\]
For the convergence we have the result obtained in \cite{calo2014asymptotic}. There it is proven that there are constants $C,C_1>0$ such that $\eta>C$, the expansion \eqref{expansionu_eta} converges absolutely in $H^1(D)$ for $\eta$ sufficiently large. We recall the following result.

\begin{theorem}
Consider the problem \eqref{eq:problem} with coefficient \eqref{eq:coeff1inc-multiple}. The corresponding expansion \eqref{expansionu_eta} with boundary condition  \eqref{Boundarycondi} converges absolutely in $H^1(D)$ for $\eta$ sufficiently large. Moreover, there exist positive constants $C$ and $C_1$ such that for every $\eta>C$, we have
\[
\left\Vert u-\sum_{j=0}^{J}\eta^{-j}u_j\right\Vert_{H^1(D)}\leq C_1\left(\|f\|_{H^{-1}(D)}+\|g\|_{H^{1/2}(\partial D)}\right)\sum_{j=J+1}^{\infty}\left(\frac{C}{\eta}\right)^{j},
\]
for $J\geq 0$.
\end{theorem}

\section{Detailed Derivation in One Dimension}\label{sec:2.3}

In this section we detail the procedure to derive (by using weak formulations) the asymptotic expansion for the solution of the problem \eqref{Strongform} in one dimension. In order to simplify the presentation we have chosen a one dimensional problem with only one high-contrast inclusion. The procedure, without detail on the derivation was summarized in Section \ref{Sec:2.2.1} for one inclusion and Section \ref{Sec:2.2.2} for multiple inclusions.\\

Let us consider the following one dimensional problem (in its strong form)
\begin{equation}
\left\{ \begin{array}{cc}
-\left(\kappa\left(x\right)u'\left(x\right)\right)'=f\left(x\right), & \mbox{for all }x\in\left(-1,1\right),\\
u\left(-1\right)=u\left(1\right)=0.
\end{array}\right.\label{eq1 strong form}
\end{equation}
where the high-contrast coefficient is given by,
\begin{equation}\label{eq:def:kappa1D}
\kappa\left(x\right)=\left\{ \begin{array}{cc}
1, & -1\leq x<-\delta\mbox{ or }\delta<x\leq1,\\
\eta, & -\delta\leq x\leq\delta.
\end{array}\right.
\end{equation}
In this case we assume $\eta>>1$. This differential equation models the stationary temperature of a bar represented by the one dimensional domain $\left[-1,1\right]$.  In this case, the coefficient $\kappa$ models the conductivity of the bar which depends on the material the bar is made of. For this particular coefficient, the part of the domain represented by the interval  $\left(-\delta,\delta\right)$ is highly-conducting when compared with the rest of the domain and we say that this medium (bar) has high-contrast conductivity properties.\\

We first write the weak form of this problem. We follow the usual procedure, that is, we select a space of test functions, multiply both sides of the equations by this test functions, then, we use integration by parts formula (in one dimension) and obtain the weak form. In order to fix ideas and concentrate on the derivation of the asymptotic expansion we use the usual test function and solutions spaces, in this case,
that would be subspaces of $H^{1}\left(-1,1\right)$ for both, solutions and test functions.\\

Let $v\in H_{0}^{1}\left(-1,1\right)$ be a test function and $u\in H^{1}\left(-1,1\right)$ is the sought solution, for more details of the spaces $H^1$ and $H_0^1$ see Section \ref{SectionA.3}. Then after integration by parts we can write
\[
\int_{-1}^{1}\kappa u'v'=\int_{-1}^{1}fv,\quad \mbox{ for all }  v\in H_{0}^{1}\left(-1,1\right).
\]
Here $u$ is the weak solution of problem (\ref{eq1 strong form}). Existence and uniqueness of the weak solution follows from usual arguments (Lax-Milgram theorem in Section \ref{SectionB.3}). To emphasis the dependence of $u$ on the contrast $\eta$, from now on, we write $u_{\eta}$ for the solution of (\ref{eq1 strong form}). Here we note that in order to simplify the notation we have omitted the integration variable $x$ and the integration measure $dx$. We can split this integral in  sub-domains integrals and recalling the definitions of the high-contrast conductivity coefficient $\kappa=\kappa\left(x\right)$ in \eqref{eq:def:kappa1D}, we obtain
\begin{equation}
\int_{-1}^{-\delta}u'v'+\eta\int_{-\delta}^{\delta}u'v'+\int_{\delta}^{1}u'v' =\int_{-1}^{1}fv,\label{eq2 weak form}
\end{equation}
for all $v\in H_{0}^{1}\left(-1,1\right)$. We observe that each integral on the left side is finite (since all the factors are in $L^{2}\left(-1,1\right)$).\\

Our goal is a write a expansion of the form 
\begin{equation}
u_{\eta}\left(x\right)=\sum_{j=0}^{\infty}\eta^{-j}u_{j}\left(x\right),\label{eq3 power series}
\end{equation}
with individual terms in $H^{-1}\left(-1,1\right)$ such that the satisfy the Dirichlet boundary condition $u_{j}\left(-1\right)=u_{j}\left(1\right)=0$ for $j\geq1$.  Other boundary conditions for \eqref{eq1 strong form} can be handled similarly. Each term will solve (weakly)  boundary value problems in the sub-domains $\left(-1,-\delta\right),\,\left(-\delta,\delta\right)$ and $\left(\delta,1\right)$. The different data on the boundary of sub-domains are revealed by the corresponding local weak formulation derived from the power
series above \eqref{eq3 power series}. We discuss this in detailed below.\\

We first assume that \eqref{eq3 power series} is a valid solution of problem \eqref{eq2 weak form}, see \cite{calo2014asymptotic}, so that we can substitute \eqref{eq3 power series} into (\ref{eq2 weak form}). We obtain that for all $v\in H_{0}^{1}\left(-1,1\right)$ the following hold
\[
\int_{1}^{-\delta}\sum_{j=0}^{\infty}\eta^{-j}u'_{j}v'+\eta\int_{-\delta}^{\delta}\sum_{j=0}^{\infty}\eta^{-j}u'_{j}v'+\int_{\delta}^{1}\sum_{j=0}^{\infty}\eta^{-j}u'_{j}v'=\int_{-1}^{1}fv,
\]
or, after formally interchanging integration and summation signs, 
\[
\sum_{j=0}^{\infty}\eta^{-j}\left[\int_{1}^{-\delta}u'_{j}v'+\int_{\delta}^{1}\eta^{-j}u'_{j}v'\right]+\sum_{j=0}^{\infty}\eta^{-j+1}\int_{-\delta}^{\delta}u'_{j}v'=\int_{-1}^{1}fv.
\]
Rearrange these to obtain,
\[
\sum_{j=0}^{\infty}\eta^{-j}\left[\int_{1}^{-\delta}u'_{j}v'+\int_{\delta}^{1}\eta^{-j}u'_{j}v'\right]+\eta\int_{-\delta}^{\delta}u'_{0}v'+\sum_{j=0}^{\infty}\eta^{-j}\int_{-\delta}^{\delta}u'_{j+1}v'=\int_{-1}^{1}fv,
\]
which after collecting terms can be written as
\begin{equation}\label{eq4 wf-ueta-exp}
\sum_{j=0}^{\infty}\eta^{-j}\left[\int_{1}^{-\delta}u'_{j}v'+\int_{-\delta}^{\delta}u'_{j+1}v'+\int_{\delta}^{1}\eta^{-j}u'_{j}v'\right]+\eta\int_{-\delta}^{\delta}u'_{0}v'=\int_{-1}^{1}fv,
\end{equation}
which holds for all test function $v\in H^1_0(-1,1)$. Now we match up the coefficients corresponding to equal powers on the both sides of equation \eqref{eq4 wf-ueta-exp}.

\subsubsection{Terms corresponding to $\eta$}

In the equation \eqref{eq4 wf-ueta-exp} above, there is only one term with $\eta$
so that we obtain
\[
\int_{-\delta}^{\delta}u'_{0}v'=0,\quad\mbox{for all }v\in H_{0}^{1}\left(-1,1\right).
\]
Thus  $u_0^\prime=0$  (which can be readily seen if we take a test function $v$ such that $v=u_{0}$) and therefore $u_{0}$ is a constant in $(-\delta , \delta)$.

\subsubsection{Terms corresponding  to $\eta^{0}=1$}

The next coefficients to match up are those of $\eta^{0}=1$, the coefficients in \eqref{eq4 wf-ueta-exp} with $j=0$.\\
\begin{figure}
\begin{centering}
\includegraphics[scale=0.3]{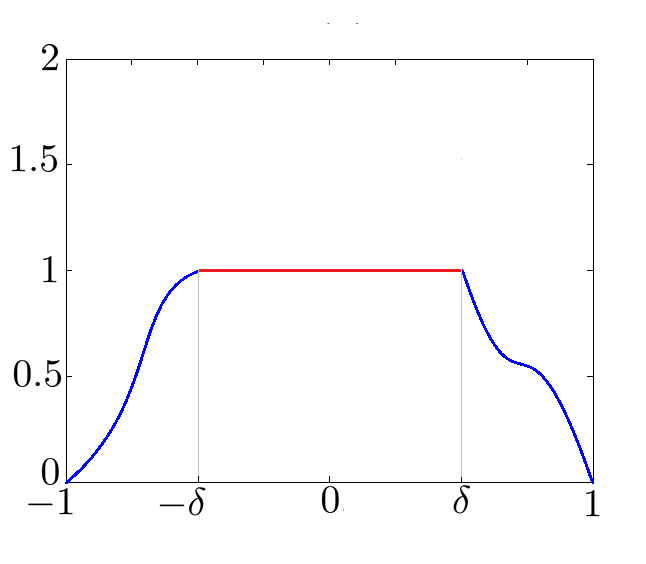}
\caption{$v$ function in $V_{\const}$.} \label{Figurev}
\end{centering}
\end{figure}
Let
\[
V_{\const}=\left\{ v\in H_{0}^{1}(-1,1): \ v^{(1)}=v|_{(-\delta,\delta)} \mbox{ is constant}\right\}.
\]
For an illustration of the function $v$ in $V_{\const}$, see Figure \ref{Figurev}. We have
\begin{equation}\label{eq5 wf u1}
\int_{-1}^{-\delta}u'_{0}v'+\int_{\delta}^{1}u'_{0}v'=\int_{-1}^{1}fv,\quad\mbox{for all }v\in V_{\const}.
\end{equation}
To study further this problem we introduce the following decomposition for functions in $v\in V_{\const}$. For any $v\in  V_{\const}$, we write 
\[
v=c_0\chi+v^{\left(1\right)}+v^{\left(2\right)},
\]
where $v^{\left(1\right)}\in H_{0}^{1}\left(-1,-\delta\right)$, $v^{\left(2\right)}\in H_{0}^{1}\left(\delta,1\right)$ and $\chi$ is a continuous function defined by
\begin{equation}\label{Chifunction}
\chi(x)=\left\{ \begin{array}{cc}
1, & x\in(-\delta,\delta),\\
0, & x=-1,\, x=1,\\
\mbox{harmonic}, & \mbox{otherwise}.
\end{array}\right.
\end{equation}
\begin{figure}[!h]
\begin{centering}
\includegraphics[scale=0.3]{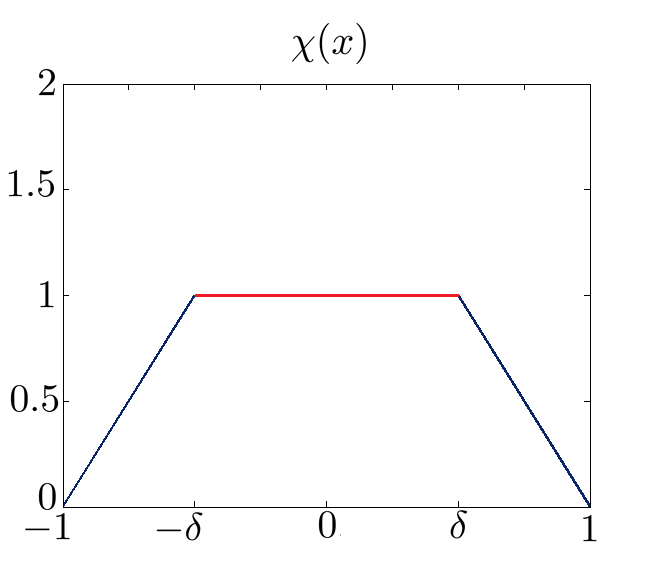}
\caption{$\chi$ function.} \label{FigureChi}
\end{centering}
\end{figure}
See Figure \ref{FigureChi} for an illustration of $\chi$. Note also that
\begin{equation}\label{Chideriv}
\chi'(x)=\left\{ \begin{array}{cc}
0, & x\in(-\delta,\delta),\\
1/(1-\delta), & x\in(-1,-\delta)\\
-1/(1-\delta), & x\in(\delta,1),
\end{array}\right.
\end{equation}
and observe that $\chi\in H_0^1(-1,1)$.\\

The same decomposition holds for the $u_{0}$, that is, $u_{0}=c_0\chi+u^{(1)}+u^{(2)}$, since $u_0\in V_{\const}$. Note that this is an orthogonal decomposition. This can be verified by direct integration. Next we split (\ref{eq5 wf u1}) into three equations using the decomposition introduced above. This is done by testing again subset of test functions determined by the decomposition introduced above.

\begin{enumerate}
\item Test against $\chi$, that is, let $v=\chi\in H_{0}^{1}\left(-1,1\right)$
in (\ref{eq5 wf u1}) and using that $u_{0}=d\chi+u^{\left(1\right)}+u^{\left(2\right)}$
we get,
\begin{equation} \label{decomposition u0}
\int_{-1}^{-\delta}\left(c_0\chi+u^{\left(1\right)}+u^{\left(2\right)}\right)'\chi'+\int_{\delta}^{1}\left(c_0\chi+u^{\left(1\right)}+u^{\left(2\right)}\right)'\chi'=\int_{-1}^{1}f\chi,
\end{equation}
which after simplification (using the definition of $\chi$ and the fundamental theorem of calculus) gives 
\[
\int_{-1}^{-\delta}c_0\left(\chi'\right)^{2}+\int_{\delta}^{1}c_0\left(\chi'\right)^{2}=\int_{-1}^{1}f\chi,
\]
from which we get 
\begin{equation}\label{const d}
c_0=\frac{\int_{-1}^{1}f\chi}{\int_{-1}^{-\delta}\left(\chi'\right)^{2}+\int_{\delta}^{1}\left(\chi'\right)^{2}}=\frac{1-\delta}{2}\int_{-1}^{1}f\chi,
\end{equation}
where we have used the derivative of $\chi$ defined in \eqref{Chideriv} for $(-1,-\delta)$ and $(\delta,1)$ respectively.

\item We now test (\ref{eq5 wf u1}) against $v^{\left(1\right)}\in H_{0}^{1}\left(-1,-\delta\right)$ (extended by zero), that is, we take $v=v^{\left(1\right)}$ in $v^{\left(1\right)}\in H_{0}^{1}\left(-1,-\delta\right)$
to get
\[
\int_{-1}^{-\delta}\left(c_0\chi+u^{\left(1\right)}+u^{\left(2\right)}\right)'\left(v^{\left(1\right)}\right)'+\int_{\delta}^{1}\left(c_0\chi+u^{\left(1\right)}+u^{\left(2\right)}\right)'\left(v^{\left(1\right)}\right)'=\int_{-1}^{1}fv^{\left(1\right)}.
\]
Note that ${\displaystyle \int_{\delta}^{1}\left(c_0\chi+u^{\left(1\right)}+u^{\left(2\right)}\right)'\left(v^{\left(1\right)}\right)'=0}$
(since $v^{\left(1\right)}$ is supported in $\left(-1,-\delta\right)$).
Using this, we have that
\[
\int_{-1}^{-\delta}\left(c_0\chi+u^{\left(1\right)}+u^{\left(2\right)}\right)'\left(v^{\left(1\right)}\right)'=\int_{-1}^{-\delta}fv^{\left(1\right)}.
\]
After simplifying (using the definition of $\chi$ and the fundamental theorem of calculus) we have
\begin{equation}
\int_{-1}^{-\delta}\left(u^{\left(1\right)}\right)'\left(v^{\left(1\right)}\right)'=\int_{-1}^{-\delta}fv^{\left(1\right)},\quad\mbox{for all }v^{\left(1\right)}\in H_{0}^{1}\left(-1,-\delta\right).\label{eq6 wf u(1)}
\end{equation}
Recalling the boundary values of $u^{\left(1\right)}$ we see that (\ref{eq6 wf u(1)}) is the weak formulation of Dirichlet problem
\begin{equation}\label{eq:dirproblemastrong1}
\left\{ \begin{array}{c}
-\left(u^{\left(1\right)}\right)''=f, \quad \mbox{ in } (-1,-\delta), \\
u^{\left(1\right)}\left(-1\right)=0,\, u^{\left(1\right)}\left(-\delta\right)=0.
\end{array}\right.
\end{equation}

\item Testing against $v^{\left(2\right)}\in H_{0}^{1}\left(\delta,1\right)$
we get (in a similar fashion) that
\begin{equation}
\int_{\delta}^{1}\left(u^{\left(2\right)}\right)'\left(v^{\left(2\right)}\right)'=\int_{\delta}^{1}fv^{\left(2\right)},\quad\mbox{for all }v^{\left(2\right)}\in H_{0}^{1}\left(\delta,1\right).\label{eq7 wf u(2)}
\end{equation}
Then $u^{\left(2\right)}$ is the weak formulation to the Dirichlet problem 
\[
\left\{ \begin{array}{c}
-\left(u^{\left(2\right)}\right)''=f, \quad \mbox{ in } (\delta,1),\\
u^{\left(2\right)}\left(\delta\right)=0,\, u^{\left(2\right)}\left(1\right)=0.
\end{array}\right.
\]
\end{enumerate}

\subsubsection{A Remark about Dirichlet and Neumann Sub-Domain Problems}
Now we make the following remark that will be central for the upcoming arguments. 
Looking back to (\ref{eq5 wf u1}) we make the following observation
about $u_{0}$. Taking test functions $v\in H_{0}^{1}\left(-1,-\delta\right)$
we see that $u_{0}$ solve the following Dirichlet problem
\[
\int_{-1}^{-\delta}u'_{0}v'=\int_{-1}^{-\delta}fv,\quad\mbox{for all }v\in H_{0}^{1}\left(-1,-\delta\right),
\]
with the corresponding boundary data. The strong form for this Dirichlet problem is given by
\[
\left\{ \begin{array}{c}
-z''=f, \quad \mbox{ in }  (-1,-\delta)\\
z\left(-1\right)=0,\, z\left(-\delta\right)=u_{0}\left(-\delta\right),
\end{array}\right.
\]
with only solution $z=u_{0}$ in $\left(-1,-\delta\right)$. From here we conclude that $z=u_{0}$ is also the solution of the following mixed Dirichlet and Neumann boundary condition problem
\[
\left\{ \begin{array}{c}
-z''=f, \quad \mbox{ in } (-1,-\delta),\\
z\left(-1\right)=0,\, z'\left(-\delta^{-}\right)=u'_{0}\left(-\delta^{-}\right),
\end{array}\right.
\]
with weak formulation given by
\[
\int_{-1}^{-\delta}z'v'=u'_{0}\left(-\delta^{-}\right)v\left(-\delta^{-}\right)+\int_{-1}^{-\delta}fv,\quad \mbox{ for all }v\in H^1(-1,-\delta),
\]
for all $v\in H^{1}\left(-1,-\delta\right)$. Since $z=u_{0}\in H^{1}\left(-1,-\delta\right)$ is solution of this problem we can write
\begin{equation}\label{eq8 u0 first}
\int_{-1}^{-\delta}u'_{0}v'=u'_{0}\left(-\delta^{-}\right)v\left(-\delta^{-}\right)+\int_{-1}^{-\delta}fv,\quad \mbox{ for all }v\in H^1(-1,-\delta)
\end{equation}
for all $v\in H^{1}\left(-1,-\delta\right)$. Analogously we can write 
\begin{equation}\label{eq9 u0 second}
\int_{\delta}^{1}u'_{0}v'=u'_{0}\left(\delta^{+}\right)v\left(\delta^{+}\right)+\int_{\delta}^{1}fv,\quad \mbox{ for all }v\in H^1(\delta,1)
\end{equation}
for all $v\in H^{1}\left(\delta,1\right)$.\\

By substituting the equations (\ref{eq8 u0 first}) and (\ref{eq9 u0 second})
into (\ref{eq5 wf u1}) we have
\[
\left(u'_{0}\left(-\delta^{-}\right)v\left(-\delta^{-}\right)+\int_{-1}^{-\delta}fv\right)+\int_{-\delta}^{\delta}u'_{1}v'+\left(\int_{\delta}^{1}fv-u'_{0}\left(\delta^{+}\right)v\left(\delta^{+}\right)\right)=\int_{-1}^{1}fv,
\]
so that, after simplifying it gives, for all $v\in H^{1}\left(-1,1\right)$, that
\begin{equation}
\int_{-\delta}^{\delta}u'_{1}v'=\int_{-\delta}^{\delta}fv+u'_{0}\left(\delta^{+}\right)v\left(\delta^{+}\right)-u'_{0}\left(-\delta^{-}\right)v\left(-\delta^{-}\right).\label{eq10 wf u1 in -d,d}
\end{equation}
This last equation (\ref{eq10 wf u1 in -d,d}) is the weak formulation of the Neumann problem for $u_{1}$ defined by
\begin{equation}
\left\{ \begin{array}{c}
-u_{1}''=f, \quad \mbox{ in } (-\delta,\delta),\\
u'_{1}\left(-\delta^{+}\right)=u'_{0}\left(-\delta^{-}\right),\, u'_{1}\left(\delta^{-}\right)=u'_{0}\left(\delta^{+}\right).
\end{array}\right.\label{eq11 Neumann for u1}
\end{equation}
This classical Neumann problem has solution only if the compatibility condition is satisfied. The compatibility condition comes from the fact that if we integrate directly the first equation in (\ref{eq11 Neumann for u1}), we have
\[
-\int_{-\delta}^{\delta}u''_{1}=\int_{-\delta}^{\delta}f,\,\mbox{ and then we need }u'_{1}(-\delta^{+})-u'_{1}(\delta^{-})=\int_{-\delta}^{\delta}f.
\]
Here, the derivatives are defined using side limits for the function $u_{1}$. Using the fact that $u'_{1}(-\delta^{+})=u'_{0}(-\delta^{-})$
and $u'_{1}(\delta^{-})=u'_{0}(\delta^{+})$ then, the compatibility conditions becomes
\begin{equation}
u'_{0}(-\delta^{-})-u'_{0}(\delta^{+})=\int_{-\delta}^{\delta}f.\label{eq12 compati u1}
\end{equation}
In other to verify this compatibility condition we first observe that if we take $v=\chi$ in (\ref{eq8 u0 first}) and (\ref{eq9 u0 second}) and recalling the definition of $\chi$ we obtain
\[
\int_{-1}^{-\delta}u'_{0}\chi'=u'_{0}(-\delta^{-})(1)+\int_{-1}^{-\delta}f\chi\quad\mbox{and}\quad\int_{\delta}^{1}u'_{0}\chi'=u'_{0}(\delta^{+})(-1)+\int_{\delta}^{1}f\chi.
\]
On the other hand, if we take $v=\chi$ in (\ref{eq5 wf u1}) we have
\[
\int_{-1}^{-\delta}u'_{0}\chi'+0+\int_{\delta}^{1}u'_{0}\chi'=\int_{-1}^{1}f\chi.
\]
Combining these three equations we conclude that (\ref{eq12 compati u1})
holds true.\\
%

Now, observe that \eqref{eq11 Neumann for u1} has unique solution up to a constant so, the solution take the form $u_{1}=\widetilde{u}_{1}+c_{1}$, with $c_{1}$ an integration constant; in addition, $\widetilde{u}_{1}$ is a function with the property that its means measure is $0$, i.e.,
\[
\int_{-\delta}^{\delta}\widetilde{u}_{1}=0.
\]
In this way, we need to determine the value of constant $c_{1}$, for this we substitute the function $u_{1}$ with a total function $\widetilde{u}_{1}+c_{1}$, but the constant $c_{1}$ cannot be computed in this part, so it will be specified later.\\

So $\widetilde{u}_1$ solves the Neumann problem in $(-\delta,\delta)$
\begin{equation}\label{problem u'1}
\int_{-\delta}^{\delta}\widetilde{u}_1'v'=\int_{-\delta}^{\delta}fv-\left[u_0'(\delta^{+})-u_0'(-\delta^{-})\right],\mbox{ for all }v\in H^1(-\delta,\delta).
\end{equation}

\subsubsection{Terms Corresponding  to $\eta^{-1}$}

For the other parts of $u_{1}$ in the interval, we need the term of $\eta$ with $j=1$, which is given from the equation \eqref{eq4 wf-ueta-exp}, we have
\begin{equation}
\int_{-1}^{-\delta}u'_{1}v'+\int_{-\delta}^{\delta}u'_{2}v'+\int_{\delta}^{1}u'_{1}v'=0,\quad\mbox{for all }v\in H_{0}^{1}\left(-1,1\right).\label{eq13 wf u2}
\end{equation}
Note that if we restrict this equation to test functions $v\in H_{0}^{1}(-1,-\delta)$ and $v\in H_{0}^{1}(\delta,1)$ such as in (\ref{eq5 wf u1}), i.e., $\int_{-\delta}^{\delta}u'_{2}v'=0$, we have
\[
\int_{-1}^{-\delta}u'_{1}v'=0, \mbox{ for all }v\in H_{0}^{1}(-1,-\delta),
\]
and
\[
\int_{\delta}^{1}u'_{1}v'=0,\mbox{ for all }v\in H_{0}^{1}(\delta,1),
\]
where each integral is a weak formulation to problems with Dirichlet conditions
\begin{equation}
\left\{ \begin{array}{l}
-u_{1}''=0,\quad \mbox{ in } (-1,-\delta)\\
u_{1}(-1)=0,\\
u_{1}(-\delta^{-})=u_{1}(-\delta^{+})=\widetilde{u}_1(-\delta^{+})+c_1,
\end{array}\right.\label{eq14 D-N u1}
\end{equation}
and
\begin{equation}
\left\{ \begin{array}{l}
-u_{1}''=0,\quad \mbox{ in } (\delta,1)\\
u_{1}(\delta^{+})=u_{1}(\delta^{-})=\widetilde{u}_1(\delta^{-})+c_1,\\ 
u_{1}(1)=0,
\end{array}\right.\label{eq15 D-N u1}
\end{equation}
respectively.\\

Back to the problem (\ref{eq13 wf u2}) above, we compute $u_{2}$ with given solutions in \eqref{eq14 D-N u1} and \eqref{eq15 D-N u1} we get
\[
\int_{-\delta}^{\delta}u_{2}'v'=u_{1}'(-\delta^{-})v(-\delta^{-})-u_{1}'(\delta^{+})v(\delta^{+}).
\]
This equation is the weak formulation to the Neumann problem 
\begin{equation}\label{u2''forcompa}
\left\{ \begin{array}{c}
-u_{2}''=0,\quad \mbox{ in } (-\delta,\delta)\\
u'_{2}(-\delta^{+})=u'_{1}(-\delta^{-}),\, u_{2}'(\delta^{-})=u_{1}'(\delta^{+}).
\end{array}\right.
\end{equation}
Note that, since $u_2$ depends only on the (normal) derivative of $u_1$, then, it does not depend on the value of $c_1$. But the value of $c_1$ is chosen such that compatibility condition holds.

\subsubsection{Terms Corresponding  to $\eta^{-j}$, $j\geq2$}

In order to determinate the other parts of $u_{2}$ we need the term of $\eta$ but this procedure is similar for the case of $u_j$ with $j\geq 2$, so we present the deduction for general $u_j$ with $j\geq 2$. Thus we have that 
\begin{equation} 
\int_{-1}^{-\delta}u'_{j}v'+\int_{-\delta}^{\delta}u'_{j+1}v'+\int_{\delta}^{1}u'_{j}v'=0,\quad\mbox{for all }v\in H_{0}^{1}(-1,1). \label{wf u3}
\end{equation}
Again, if we restrict this last equation to $v\in H_{0}^{1}(-1,-\delta)$ and $v\in H_0^1(\delta,1)$ to be 
defined in (\ref{eq5 wf u1}), we have
\[
\int_{-1}^{-\delta}u'_{j}v'=0, \mbox{ for all }v\in H_{0}^{1}(-1,-\delta),
\]
and
\[
\int_{\delta}^{1}u'_{j}v'=0, \mbox{ for all }v\in H_{0}^{1}(\delta,1),
\]
respectively. Where each integral is a weak formulation to the Dirichlet problem
\begin{equation}
\left\{ \begin{array}{c}
-u_{j}''=0, \quad \mbox{ in } (-1,\delta), \\
u_{j}(-1)=0,\\
u_{j}(-\delta^{-})=u_{j}(-\delta^{+})=\widetilde{u}_j(-\delta^{+})+c_j,
\end{array}\right.\label{eq16 D-N u2}
\end{equation}
and
\begin{equation}
\left\{ \begin{array}{l}
-u_{j}''=0, \quad \mbox{ in } (\delta,1),\\
u_{j}(\delta^{+})=u_{j}(\delta^{-})=\widetilde{u}_j(\delta^{-})+c_j,\\
u_{j}(1)=0.
\end{array}\right.\label{eq17 D-N u2}
\end{equation}
Following a similar argument to the one given above, we conclude that $u_{j}$ is harmonic in the intervals $(-1,-\delta)$ and $(\delta,1)$ for all $j\geq 1$ and $u_{j-1}$ is harmonic in $(-\delta,\delta)$ for $j\geq 2$. As before, we have
\[
u_{j}'(\delta^{-})-u_{j}'(-\delta^{+})=-\left[u_{j-1}'(\delta^{+})-u_{j-1}'(-\delta^{-})\right],\quad \mbox{ for all }j\geq 2.
\]
Note that $u_{j}$ is given by the solution of a Neumann problem  in $(-\delta,\delta)$. Thus, the function take the form $u_{j}=\widetilde{u}_{j}+c_{j}$, with
$c_{j}$ being a integration constant, though, $\widetilde{u}_{j}$ is a function
with the property that its integral is $0$, i.e.,
\[
\int_{-\delta}^{\delta}\widetilde{u}_{j}=0,\quad\mbox{ for all }j\geq 2.
\]
In this way, we need to determinate the value of constant $c_{j}$, for this, we substitute the function $u_{j}$ with a total function $\widetilde{u}_{j}+c_{j}$. Note that $\widetilde{u}_j$ solves the Neumann problem in $(-\delta,\delta)$\\
\begin{equation}\label{Neumann u_j-New}
\int_{-\delta}^{\delta}\widetilde{u}_j'v'=-\left[u_{j-1}'(\delta^{+})-u_{j-1}'(-\delta^{-})\right],\quad\mbox{for all }v\in H^1(-\delta,\delta).
\end{equation}

Since $c_{j}$, with $j=2,\dots$, are constants, their harmonic extensions are given by $c_{j}\chi$ in $(-1,1)$; see Remark \ref{RemarkNew}, We have 
\[
u_j=\widetilde{u}_j+c_j\chi,\quad\mbox{in }(-\delta,\delta).
\]
This complete the construction of $u_{j}$.\\

From the equation \eqref{wf u3}, and solutions of Dirichlet problems \eqref{eq16 D-N u2} and \eqref{eq17 D-N u2} we have
\begin{equation} \label{Weakneumann u_j+1}
\int_{-\delta}^{\delta}u_{j+1}'v'=u_j'(-\delta^{-})v(-\delta^{-})-u_j'(\delta^{+})v(\delta^{+})
\end{equation}
This equation is the weak formulation to the Neumann problem
\begin{equation}\label{uj+1forcompa}
\left\{ \begin{array}{c}
-u_{j+1}''=0,\quad \mbox{ in } (-\delta,\delta)\\
u_{j+1}'(-\delta^{+})=u'_{j}(-\delta^{-}),\, u_{j+1}'(\delta^{-})=u'_{j}(\delta^{+}).
\end{array}\right.
\end{equation}
As before, we have
\[
u_{j+1}'(\delta^{-})-u_{j+1}'(-\delta^{+})=-\left[u_j'(\delta^{+})-u_j'(-\delta^{-})\right].
\]
The compatibility conditions need to be satisfied. Observe that
\begin{eqnarray*}
u_{j+1}'(\delta^{-})-u_{j+1}'(-\delta^{+}) & = & -\left[u_j'(\delta^{+})-u_j'(-\delta^{-})\right]\\
& = & -\widetilde{u}_j'(\delta^{+})-c_j\chi'(\delta^{+})+\widetilde{u}_j'(-\delta^{-})+c_j\chi'(-\delta^{-})\\
& = & \widetilde{u}_j'(-\delta^{-})-\widetilde{u}_j'(\delta^{+})+c_j\left[\chi'(-\delta^{-})-\chi'(\delta^{+})\right]\\
& = & 0.
\end{eqnarray*}
By the latter we conclude that, in order to have the compatibility condition of \eqref{uj+1forcompa} it is enough to set,
\[
c_j=-\frac{\widetilde{u}_j'(-\delta^{-})-\widetilde{u}_j'(\delta^{+})}{\chi'(-\delta^{-})-\chi'(\delta^{+})}.
\]
We can choose $u_{j+1}$ in $(-\delta,\delta)$ such that
\[
u_{j+1}=\widetilde{u}_{j+1}+c_{j+1},\quad\mbox{ where }\int_{-\delta}^{\delta}\widetilde{u}_{j+1}=0,
\]
and $\widetilde{u}_{j+1}$ solves the Neumann problem
\begin{equation}\label{Neumann u_j+1-New}
\int_{-\delta}^{\delta}\widetilde{u}_{j+1}'v'=-\left[u_{j}'(\delta^{+})-u_{j}'(-\delta^{-})\right],\quad\mbox{for all }v\in H^1(-\delta,\delta).
\end{equation}

and, as before
\[
c_{j+1}=-\frac{\widetilde{u}_{j+1}'(-\delta^{-})-\widetilde{u}_{j+1}'(\delta^{+})}{\chi'(-\delta^{-})-\chi'(\delta^{+})}.
\]

\subsection{Illustrative Example in One Dimension}

In this part we show a simple example of the weak formulation with the purpose of illustrating of the development presented above. First take the next (strong) problem 
\begin{equation}\label{example1DS}
\left\{ \begin{array}{c}
\left(\kappa \left(x\right)u'\left(x\right)\right)'=0, \quad \mbox{ in } (-2,2),\\
u\left(-2\right)=0,\, u\left(2\right)=4,
\end{array}\right.
\end{equation}
with the function $\kappa \left(x\right)$ defined by 
\[
\kappa \left(x\right)=\left\{ \begin{array}{ll}
1,&-2\leq x<-1,\\
\eta,&-1\leq x<1,\\
1,&1\leq x\leq2.
\end{array}\right.
\]
The weak formulation for problem \eqref{example1DS} is to find a function $u\in H^{1}(-2,2)$ such that  
\begin{equation}
\left\{ \begin{array}{c}
\int_{-2}^{2} \kappa (x)u'(x)v'(x)dx=0\\
u(-2)=0,\, u(2)=4,
\end{array}\right.\label{example1DW}
\end{equation}
for all $v\in H_{0}^{1}(-2,2)$.\\

Note that the boundary condition is not homogeneous but this case is similar and only the term $u_0$ inherit a non-homogeneous boundary condition.\\

As before, for $j=0$ we have the equation \eqref{eq5 wf u1}, then  $u_{0}$ is constant in $(-1,1)$ and we write the decomposition $u_{0}=c_0\chi+u^{(1)}+u^{(2)}$, and if $v=\chi\in H_{0}^{1}(-2,2)$ from equation \eqref{decomposition u0} we have that $c_0=2$.\\

Now, similarly, we can take $v=v^{(1)}\in H_{0}^{1}(-2,-1)$ and we obtain the weak formulation of Dirichlet problem 
\[
\left\{ \begin{array}{c}
-\left(u^{(1)}\right)''=0, \quad \mbox{ in } (-2,-1),\\
u^{\left(1\right)}(-2)=0,\, u^{\left(1\right)}(-1)=0,
\end{array}\right.
\]
that after integrating directly twice gives a linear function $u^{(1)}=\alpha_1(x+2)$ in $(-2,-1)$ and $\alpha_1$ is an integration constant. Using the boundary data we have $u^{(1)}=0$.\\

If $v=v^{(2)}\in H_{0}^{1}(1,2)$ we have the weak formulation to the Dirichlet problem  (\ref{eq:dirproblemastrong1}) that in this case becomes,
\[
\left\{ \begin{array}{c}
-\left(u^{(2)}\right)''=0, \quad \mbox{ in } (1,2),\\
u^{(2)}(1)=0,\, u^{(2)}(2)=4.
\end{array}\right.
\]
Integrating twice in the interval $(1,2)$ we obtain the solution,  $u^{(2)}=\alpha_{2}(x-1)$, which becomes $u^{(2)}=4(x-1)$. We use the boundary condition to determine the integration constant $\alpha_2$. Then, we find the decomposition for $u_{0}=c_0\chi+u^{(1)}+u^{(2)}$ for each part of interval which is given by 
\[
u_{0}=\left\{ \begin{array}{ll}
2(x+2), & x\in (-2,-1),\\
2,& x\in (-1,1),\\
2x,& x\in (1,2).
\end{array}\right.
\]
With the $u_0$ already computed we can to obtain the boundary data for the  Neumann problem that determines $u_1$  in the equation \eqref{eq11 Neumann for u1}.  We have
\[
\left\{ \begin{array}{c}
-u_{1}''=0, \quad \mbox{ in } (-1,1),\\
u_{1}'(-1)=2,\, u_{1}'(1)=2.
\end{array}\right.
\]
As before, we see that the compatibility condition holds and therefore it has solution in the interval $(-1,1)$. Easy calculation gives $u_{1}=2x$. Computing the constant $c_1$, we consider the Neumann problem \eqref{problem u'1}, which has the solution $\widetilde{u}_1=2x$. By the definition of $u_1=\widetilde{u}_1+c_1$ we conclude that $c_1=0$.  Now, in order to compute $u_1$ in the $(-2,-1)$ and $(1,2)$ we apply the condition in \eqref{eq13 wf u2}. It follows from the Dirichlet problems \eqref{eq14 D-N u1} and \eqref{eq15 D-N u1}, that the solutions are $u_{1}=-2(x+2)$ and $u_{1}=-2x+4$.
We summarize the expression for $u_1$ as, 
\[
u_{1}=\left\{ \begin{array}{ll}
-2(x+2),& x\in (-2,-1),\\
2x,& x\in (-1,1),\\
-2(x-2),& x\in (1,2).
\end{array}\right.
\]
Returning to equation \eqref{eq13 wf u2} we can obtain $u_{2}$ in the interval $(-1,1)$ by solving the Neumann problem
\[
\left\{ \begin{array}{c}
-u_{2}''=0 \mbox{ in } (-1,1),\\
u_{2}'(-1)=-2,\, u_{2}'(1)=2.
\end{array}\right.
\]
Note again that the compatibility condition holds. The solution is given by $u_{2}=-2x$. Computing the constant $c_2$, we consider the Neumann problem \eqref{Neumann u_j-New} for $j=2$, has the solution $\widetilde{u}_2=-2x$. Note that we assume the boundary conditions of Dirichlet problems \eqref{eq16 D-N u2} and \eqref{eq17 D-N u2}. Then $c_{2}=0$. So, in order to find $u_2$ in the intervals $(-2,-1)$ and $(1,2)$, we apply the condition in \eqref{wf u3}. It follows from \eqref{eq16 D-N u2} and \eqref{eq17 D-N u2} that the solutions are $u_{2}=2(x+2)$ and $u_{2}=2(x-2)$, respectively.\\

For this case of terms $u_{j+1}$ with $j\geq2$, we consider the Neumann problem \eqref{uj+1forcompa} with solution $u_{j+1}=\pm 2x$ in $(-1,1)$. Note that the compatibility condition holds. Again we recall the Neumann problem \eqref{Neumann u_j+1-New} with solution $\widetilde{u}_{j+1}=\pm 2x$ in this case. We conclude that $c_j=0$ for each $j=2,3\dots$. Again in order to find $u_{j+1}$ in the intervals $(-2,-1)$ and $(1,2)$. It follows of analogous form above that the solutions $u_{j+1}=\mp 2(x+2)$ in $(-2,-1)$ and $u_{j+1}=\mp 2(x-2)$ in $(1,2)$, for $j=2,3,\dots$. So, we can be calculated following terms of the power series.\\

We note that above we computed an approximation of the solution by solving local problems to the inclusion and the background. In this example, we can directly compute the solution for the problem \eqref{example1DS} and verify that the expansion is correct. We have
\[
u(x)=\alpha\int_{-2}^{x}\frac{1}{\kappa (t)}dt.
\]
Then with boundary condition of problem \eqref{example1DS} we calculate a constant $\displaystyle \alpha=\frac{2\eta}{1+\eta}$, so we have 
\begin{equation}
u\left(x\right)=\left\{ \begin{array}{ll}
\displaystyle\frac{\eta}{1+\eta}2\left(x+2\right)& x\in[-2,-1),\\
\displaystyle \frac{\eta}{1+\eta}2\left[1+\frac{1}{\eta}\left(x+1\right)\right]& x\in[-1,1),\\
\displaystyle \frac{\eta}{1+\eta}2\left[1+\frac{2}{\eta}+\left(x-1\right)\right]& x\in[1,2].
\end{array}\right. \label{Exact solution exa}
\end{equation}
Recall that the term $\frac{\eta}{1+\eta}$ can be written as a power series given by
$\frac{\eta}{1+\eta}=\sum_{j=0}^{\infty}\frac{(-1)^{j}}{\eta^{j}}$.\\

By inserting this expression into \eqref{Exact solution exa} and after some manipulations, we have 
\begin{equation}\label{Series1D}
u(x)=\left\{ \begin{array}{ll}
\displaystyle 2(x+2)\sum_{j=0}^{\infty}\frac{(-1)^{j}}{\eta^{j}}, & x\in[-2,-1),\\
\displaystyle 2-2x\sum_{j=1}^{\infty}\frac{(-1)^{j}}{\eta^{j}}, & x\in[-1,1),\\
\displaystyle 2x\sum_{j=0}^{\infty}\frac{(-1)^{j}}{\eta^{j}}-4x\sum_{j=1}^{\infty}\frac{(-1)^{j}}{\eta^{j}}, & x\in[1,2].
\end{array}\right.
\end{equation}
We can rewrite this expression to get
\begin{equation}
u\left(x\right)=\underset{u_{0}}{\underbrace{\left\{ \begin{array}{ll}
2\left(x+2\right)\\
2\\
2x
\end{array}\right\} }}+\frac{1}{\eta}\underset{u_{1}}{\underbrace{\left\{ \begin{array}{c}
-2\left(x+2\right)\\
2x\\
-\left(2x-4\right)
\end{array}\right\} }}+\frac{1}{\eta^{2}}\underset{u_{2}}{\underbrace{\left\{ \begin{array}{c}
2\left(x+2\right)\\
-2x\\
\left(2x-4\right)
\end{array}\right\} }}+\cdots,\label{expand series exact sol}
\end{equation}
Observe that these were the same terms computed before by solving local problems in the inclusions and the background.

\section{An Application}

We return to the two dimension case and show some  calculations, that may be useful for applications. In particular, we compute the energy of the solution. Other functionals of the solution can be also considered as well.

\subsection{Energy of Solution}

In this section we discuss about the energy of solutions for the elliptic problems. This is special application for our study. We show the case for problems with high-contrast coefficient, now we consider the problem \eqref{eq:problem}
\[
\int_D\kappa(x)\nabla u(x) \cdot \nabla v(x)dx=\int_Df(x)v(x)dx, \mbox{ for all }v\in H_0^1(D),
\]
if we replace again $u(x)=u_\eta(x)$ and additionally, the test function $v(x)=u_\eta(x)\in V_\const$ we have 
\begin{equation}\label{solenergy1}
\int_D\kappa(x) \nabla u_\eta(x) \cdot \nabla u_\eta(x)dx=\int_Df(x)u_\eta(x)dx, \mbox{ for all }u_\eta\in V_\const,
\end{equation}
with $u_\eta(x)=0$ on $\partial D$. We rewrite the equation \eqref{solenergy1}
\begin{equation}\label{solenergy2}
\int_D\kappa(x) |\nabla u_\eta(x)|^2dx=\int_Df(x)u_\eta(x)dx, \mbox{ for all }u_\eta\in H_0^1(D).
\end{equation}
By the equation \eqref{solenergy2} we consider, for simplicity the integral on the right side and $u_\eta$ defined in \eqref{expansionu_eta}, so
\begin{eqnarray*}
\int_Df(x)u_\eta(x)dx & = &\int_Df(x)\left(\sum_{j=0}^{\infty}\eta^{-j}u_j(x)\right)dx\\
& = & \sum_{j=0}^{\infty}\eta^{-j}\int_Df(x)u_j(x)dx\\
& = & \int_Df(x)u_0(x)dx+\frac{1}{\eta}\int_Df(x)u_1(x)dx+\cdots.
\end{eqnarray*}  
First we assume that $u_0\in V_\const$ for the integral
\begin{equation}
\int_Df(x)u_0(x)dx=\int_D|\nabla u_0(x)|^2dx, \mbox{ for all }u_0\in V_\const.
\end{equation}
Now if $u_0\in V_\const$ and  we consider the equation \eqref{arreglo2} in the integral $\frac{1}{\eta}\int_Dfu_1$, we have 
\begin{equation}
\frac{1}{\eta}\int_Df(x)u_1(x)dx = \frac{1}{\eta}\left(\int_{D_0}\nabla u_1(x)\cdot \nabla u_0(x)dx+\int_{D_1}|\nabla u_1(x)|^2\right)dx, \mbox{ for all }u_0\in V_\const.
\end{equation}
Observe that if $u_0=u_{0,0}+c_0\chi_{D_1}$ given in \eqref{eq:formulau0} we have
\[
\int_{D_0}\nabla u_0(x) \cdot \nabla u_1(x)dx=\int_{D_0}\nabla u_{0,0}(x)\cdot \nabla u_1(x)dx+c_0\int_{D_0}\nabla \chi_{D_1}(x) \cdot \nabla u_1(x)dx,
\]
where
\begin{equation}\label{energyu00}
\int_{D_0}\nabla u_{0,0}(x)\cdot \nabla u_1(x)dx=0,
\end{equation}
for $u_{0,0}\in H_0^1(D_0)$ and
\begin{equation}\label{energychi}
\int_{D_0}\nabla \chi_{D_1}(x)\cdot \nabla u_1(x)dx=0,
\end{equation}
this is because we consider the problem
\begin{equation}\label{energyu1}
\left\{ \begin{array}{ll}
-\Delta u_1(x)=0,&\mbox{ in }D_0\\
\hspace{0.25in}u_1(x)=0,&\mbox{ on }\partial D_{0}.
\end{array}\right.
\end{equation}
By the weak formulation of the equation \eqref{energyu1} with test function $z\in H_0^1(D_0)$, in particular $z=\chi_{D_1}$ we have the result. Thus by the equations \eqref{energyu00} and \eqref{energychi} we conclude that
\[
\int_{D_0}\nabla u_0(x) \cdot \nabla u_1(x)dx=0.
\]
Thereby, we find an expression for the energy of the solution in terms of the expansion as follows
\[
\int_{D}\kappa(x)\left|\nabla u_\eta(x)\right|^2dx=\int_{D_0}|\nabla u_0(x)|^2dx+\frac{1}{\eta}\int_{D_1}|\nabla u_1(x)|^2dx+{\cal O}\left(\frac{1}{\eta^2}\right).
\]



\chapter{Numerical Computation of Expansions Using the Finite Element Method}\label{Chapter3}

\minitoc

This chapter is dedicated to illustrate numerically some examples of the elliptic problems with high-contrast coefficients. We recall substantially procedures of the asymptotic expansion in the previous chapter. We develop the numerical approximation of few terms of the asymptotic expansion through the Finite Element Methods, which was implemented in a code of \textsc{MatLab} using PDE toolbox, see Appendix \ref{AppendixD} for a \textsc{MatLab} code.

\section{Overview of the Using of the Finite Element Method}\label{Sec:3.1}

In this section we describe the development of the numerical implementation for example used in the Section \ref{SectionNumerical}. We refer to techniques of the approximation of the Element Finite Method to find numerical solutions. In this case we relates few terms of the asymptotic expansion with Dirichlet and Neumann problems generate in the sub-domains.\\

The first condition is that the triangulation resolves the different geometries presented in different examples of the Section \ref{SectionNumerical}. The triangulation of the implementation is well defined and conforming geometrically in the domain. For more details of the triangulation see Appendix \ref{AppendixC}.\\

Configuring the geometry for each type of domain, we solve the problem \eqref{Chimultiinclusions} (for the case in one inclusion, see \eqref{Chioneinclusion}), using Finite Elements on the background, in general the background is considered $D_0$ for all examples.\\

The second condition we compute $\mathbf{A}_{\geom}$ defined in \eqref{Def:aml} and $\mathbf{b}$ defined in \eqref{Def:bl}. Then we can find  the vector $\mathbf{X}$ defined in \eqref{Def:X}.\\

We compute the function $u_{0,0}$ using Finite Element in the background $D_0$. We recall that $u_{0,0}$ solves the problem defined in \eqref{Def:u00}. Then we construct the function $u_{0}=u_{0,0}+\sum_{m=1}^{M}c_m\chi_{D_m}$.\\

For the other terms in the expansion we solve using Finite Element the corresponding problems \eqref{Def:aml2} and \eqref{Def:ujDiric}. Finally, we compute the corresponding constants in \eqref{Def:uj} by solving $\mathbf{A}_{\geom}\mathbf{X}_j=\mathbf{Y}_j$.

\section{Numerical Computation of Terms of the Expansion in Two Dimensions}\label{SectionNumerical}

In this section we show some examples of the expansion terms in two dimensions. In particular, few terms are computed numerically using a Finite Element Method.\\

For details on the Finite Element Method, see the Appendix \ref{AppendixC}, and the general theory see for instance \cite{johnson2012numerical, Galvis-book2009}. We recall that, apart for verifying the derived asymptotic expansions numerically, these numerical studies are a first step to understand the approximation $u_0$ and with this understanding then device efficient numerical approximations for $u_0$ (and then for $u_\eta$).

\begin{figure}[!h] 
\begin{centering}
\includegraphics[scale=0.4]{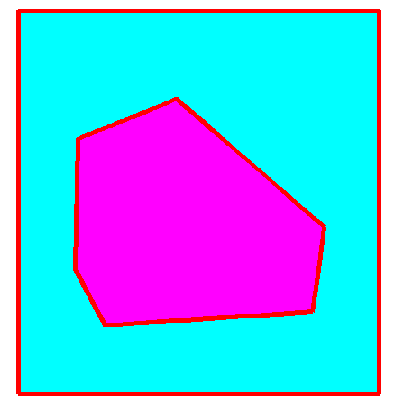}
\caption{Example of geometry configurations with interior inclusion to 
solve   the problem \eqref{Strongexample} with $u(x)=0$ on $\partial D$ and the coefficient 
$\kappa$ in \eqref{eq:coeff1inc-multiple}.
} \label{Figure1}
\end{centering}
\end{figure}

\subsection{Example 1: One inclusion}

\begin{figure}[!h]
\begin{centering}
\subfloat[$u_{0}$]{\begin{centering}
\includegraphics[scale=0.2]{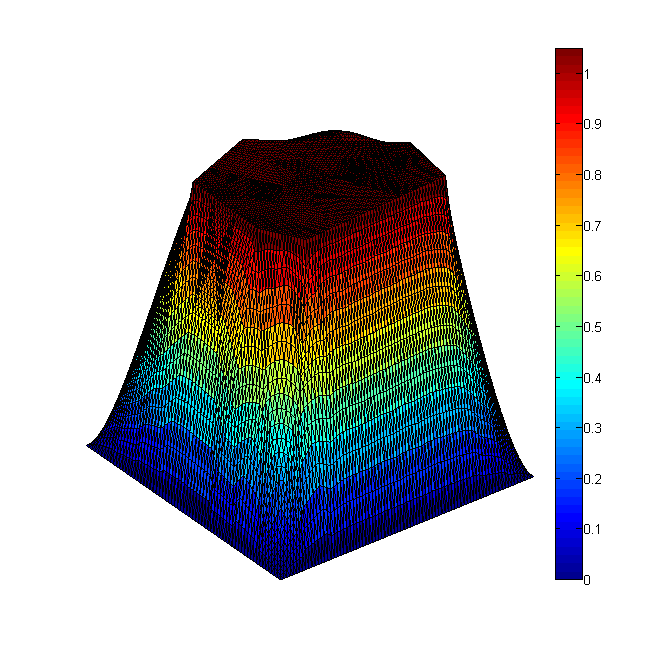}
\par\end{centering}

}\subfloat[$u_{1}$]{\begin{centering}
\includegraphics[scale=0.2]{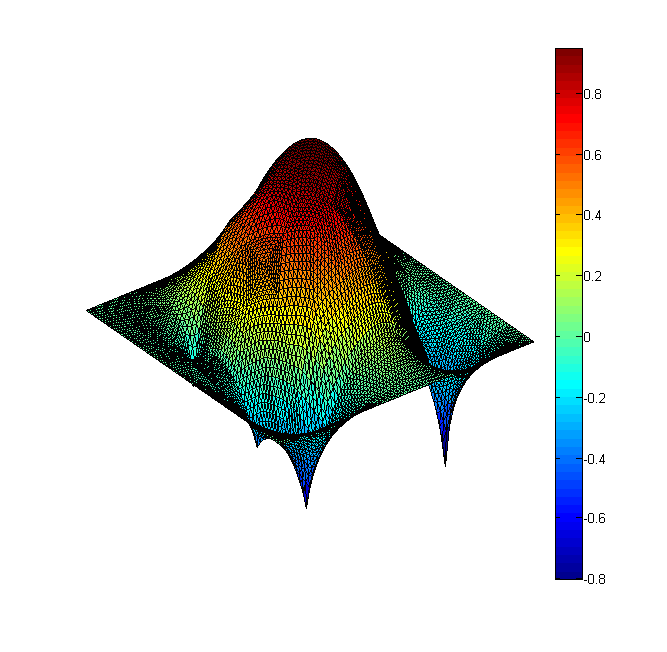}
\par\end{centering}

}
\par\end{centering}

\begin{centering}
\subfloat[$u_{2}$]{\begin{centering}
\includegraphics[scale=0.2]{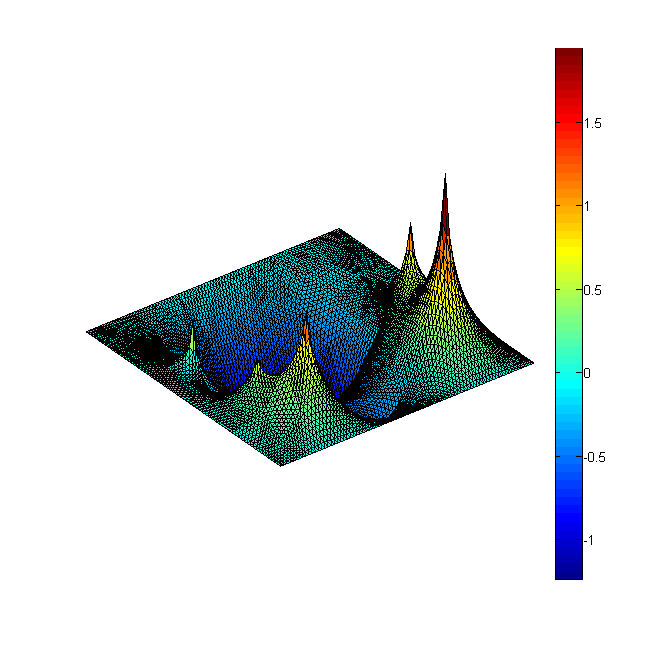}
\par\end{centering}

}\subfloat[$u_{3}$]{\begin{centering}
\includegraphics[scale=0.2]{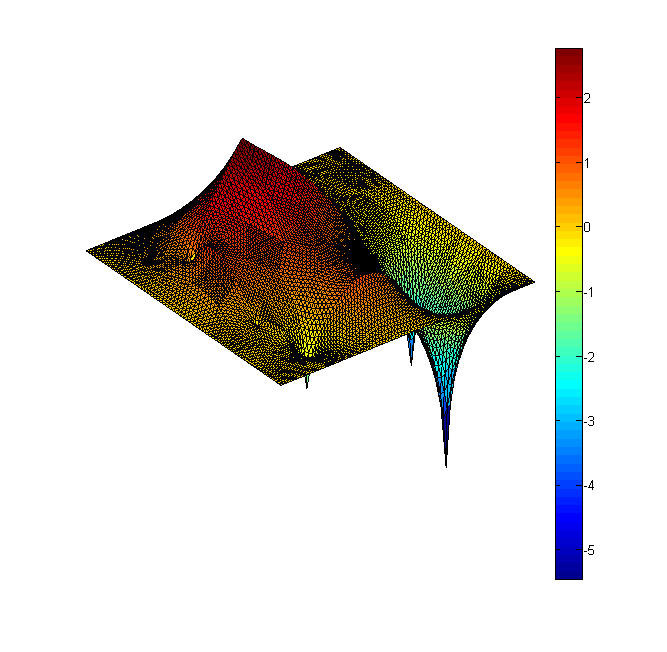}
\par\end{centering}

}
\par\end{centering}
\caption{Few terms of the asymptotic expansion \eqref{expansionu_eta} for the solution
of \eqref{Strongexample} with $u(x)=0$ on $\partial D$ and the coefficient 
$\kappa$ in \eqref{eq:coeff1inc-multiple}  where the domain configuration is  illustrated in Figure (\ref{Figure1}).
Here we consider $\eta=10$.} \label{Figure2}
\end{figure}

We numerically solve the problem 
\begin{equation} \label{Strongexample}
-\dive(\kappa(x)\nabla u(x))=f(x),\ \mbox{  in }  D,  
\end{equation}
with $u(x)=0$ on $\partial D$ and the coefficient 
$\kappa$ in \eqref{eq:coeff1inc-multiple}  where the domain configuration is  illustrated in Figure \ref{Figure1}.
This configuration contains one polygonal interior inclusion. We applied the conditions above in a numerical implementation in \textsc{MatLab} using the PDE toolbox. We set $\eta=10$ and $\eta=100$ for this example.\\

\begin{figure}[!h]
\centering
\includegraphics[width=0.4\textwidth]{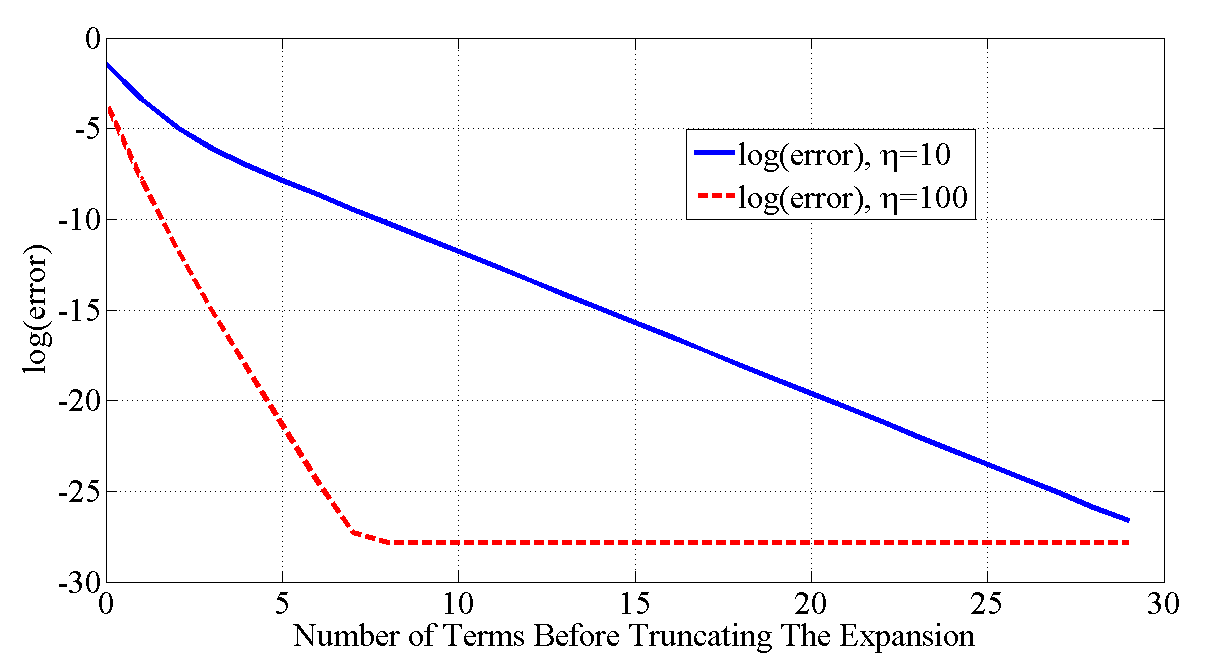}
\caption{Error of the difference of the solution $u_\eta$ computed directly and the addition of the truncation of the expansion \eqref{expansionu_eta}. We consider $\eta=10$ (solid line) and $\eta=100$ (dashed line).}
\label{fig:error} 
\end{figure}

We show the computed first four terms $u_0,u_1,u_2$ and $u_3$ of the asymptotic expansions in Figure \ref{Figure2}. In particular we show the asymptotic limit $u_0$. In addition, observe that the term $u_0$ is constant inside the inclusion region. For this case $\eta=10$ and the errors form the truncated series and the whole domain solution $u_\eta$ is reported in Figure \ref{fig:error}.  We see a linear decay of the logarithm of the error with respect to the number of terms that corresponds to the decay of the power series tail.\\


\begin{figure}[!h]
\begin{centering}
\includegraphics[scale=0.22]{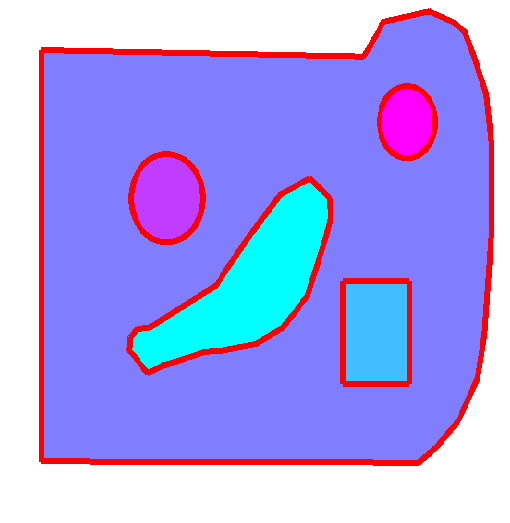}\hspace{1cm}\includegraphics[scale=0.2]{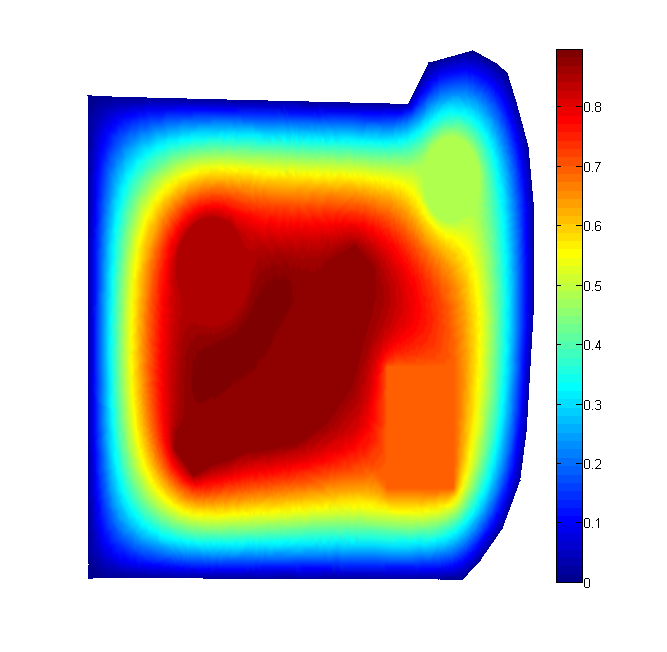}
\caption{Example of geometry configurations with interior inclusion (left picture).
Asymptotic limit $u_0$  for the solution
of \eqref{Strongexample} with $u(x)=0$ on $\partial D$ and the coefficient 
$\kappa$ in \eqref{eq:coeff1inc-multiple}  where the domain configuration is  illustrated in the left picture.
Here we consider $\eta=10$ (right picture).
} \label{Figure3}
\end{centering}
\end{figure}

\subsection{Example 2: Few Inclusions}

\begin{figure}[!h]
\begin{centering}
\subfloat[$u_{0}$]{\begin{centering}
\includegraphics[scale=0.2]{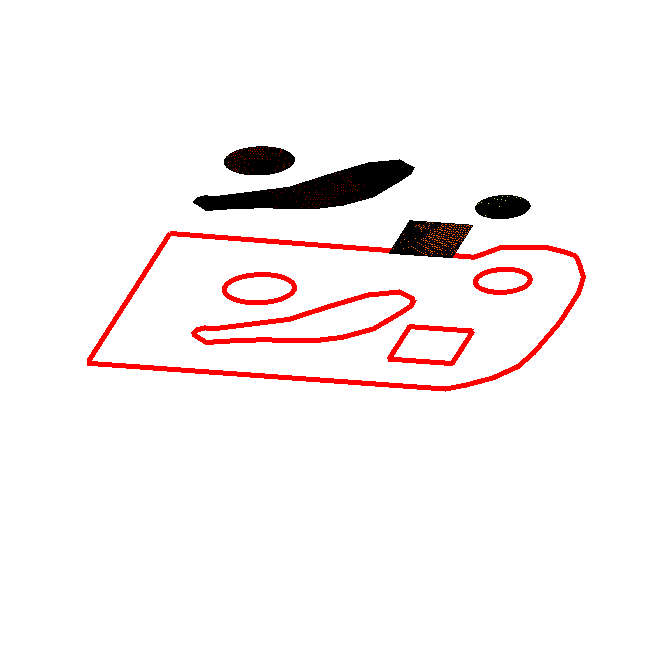}
\par\end{centering}

}\subfloat[$u_{1}$]{\begin{centering}
\includegraphics[scale=0.2]{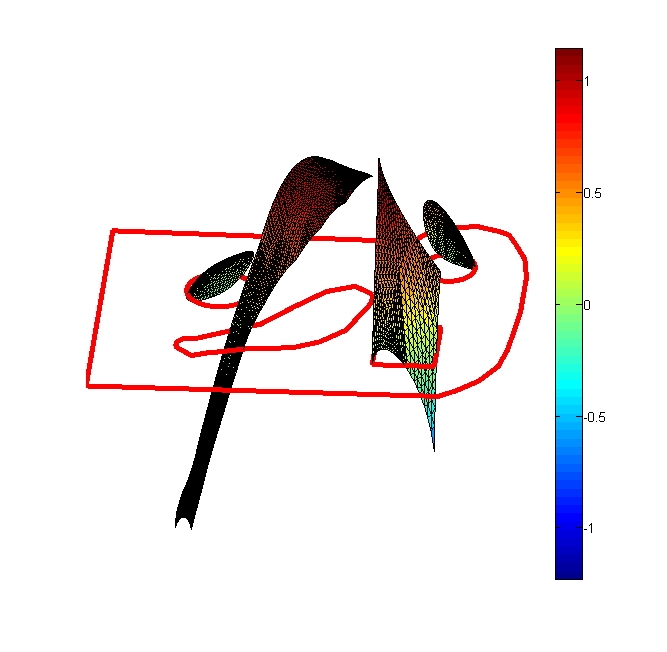}
\par\end{centering}

}\subfloat[$u_{2}$]{\begin{centering}
\includegraphics[scale=0.2]{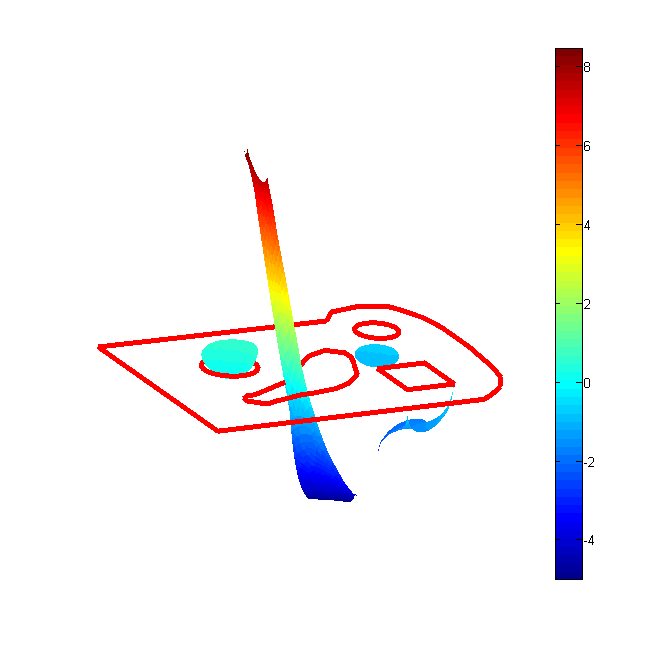}
\par\end{centering}

}
\par\end{centering}

\caption{Illustration, inside inclusions, of first few terms of the asymptotic expansion \eqref{expansionu_eta} for the solution
of \eqref{Strongexample} with $u(x)=0$ on $\partial D$ and the coefficient 
$\kappa$ in \eqref{eq:coeff1inc-multiple}  where the domain configuration is  illustrated in Figure \ref{Figure3}.
Here we consider $\eta=10$.}\label{Figure4}
\end{figure}

As a second example we consider several inclusions. The geometry for this problem posed to the defined triangulation is coupled to the finite element. The domain is illustrated in  Figure \ref{Figure3}  as well as the corresponding $u_0$ term.\\

In the Figure \ref{Figure4} we display of the asymptotic solution and show its behaviour on inclusions. Observe that the term $u_0$,  Figure \ref{Figure4} (a), is constant inside the respectively inclusions regions. This checks that $u_0$ is a combination of harmonic characteristic functions.


\subsection{Example 3: Several Inclusions}\label{Sec:3.2.3}

In this case we design a domain with several inclusions are distributed homogeneously and analyze numerically the derived asymptotic expansions.\\

We consider $D=B(0,1)$ the circle with center $(0,0)$, radius $1$ and  $36$ (identical) circular inclusions of radius $0.07$, this is illustrated in the Figure \ref{fig:figure1} (left side). Then we numerically solve the problem
\begin{equation}\label{Example1Multiscale}
\left\{\begin{array}{ll}
-\dive(\kappa(x)\nabla u_\eta(x))=1, & \mbox{ in }D\\
\hspace{0.1in}u(x)=x_1+x_2^2, & \mbox{ on }\partial D.
\end{array}\right.
\end{equation}
We show in the Figure \ref{fig:figure1} an approximation for function $u_0$ with $\eta=6$. For this case we have the behavior of term $u_0$ inside the background. We refer to the term $u_{\coarse}$ for the number of inclusions present in the geometry of problem \eqref{Example1Multiscale}.\\ 

\begin{figure}[!h]
\centering
\includegraphics[width=0.3\textwidth]{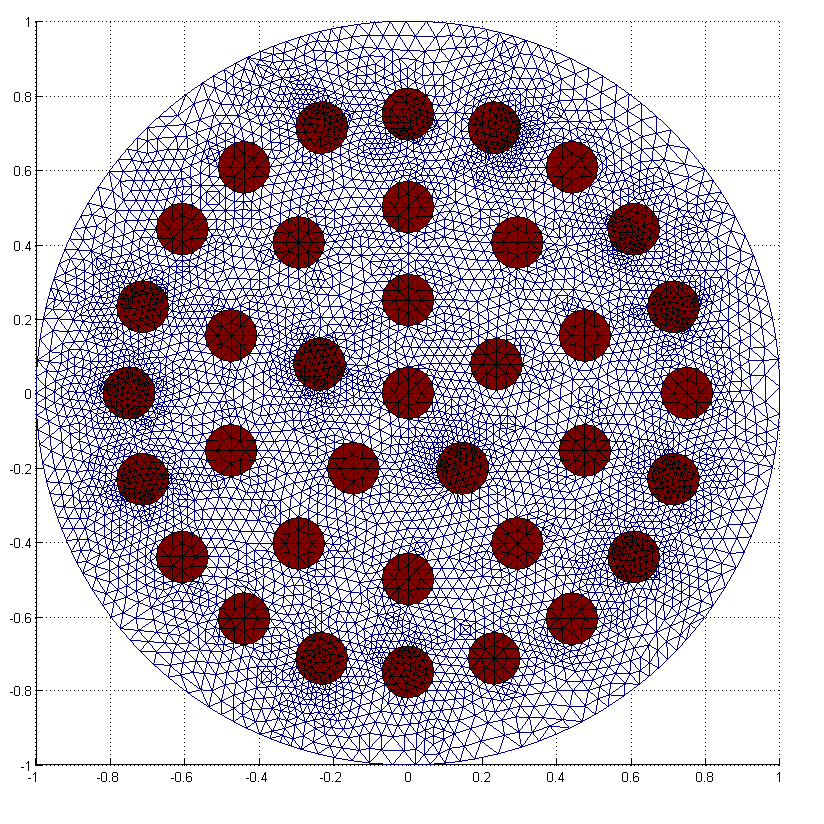}
\includegraphics[width=0.3\textwidth]{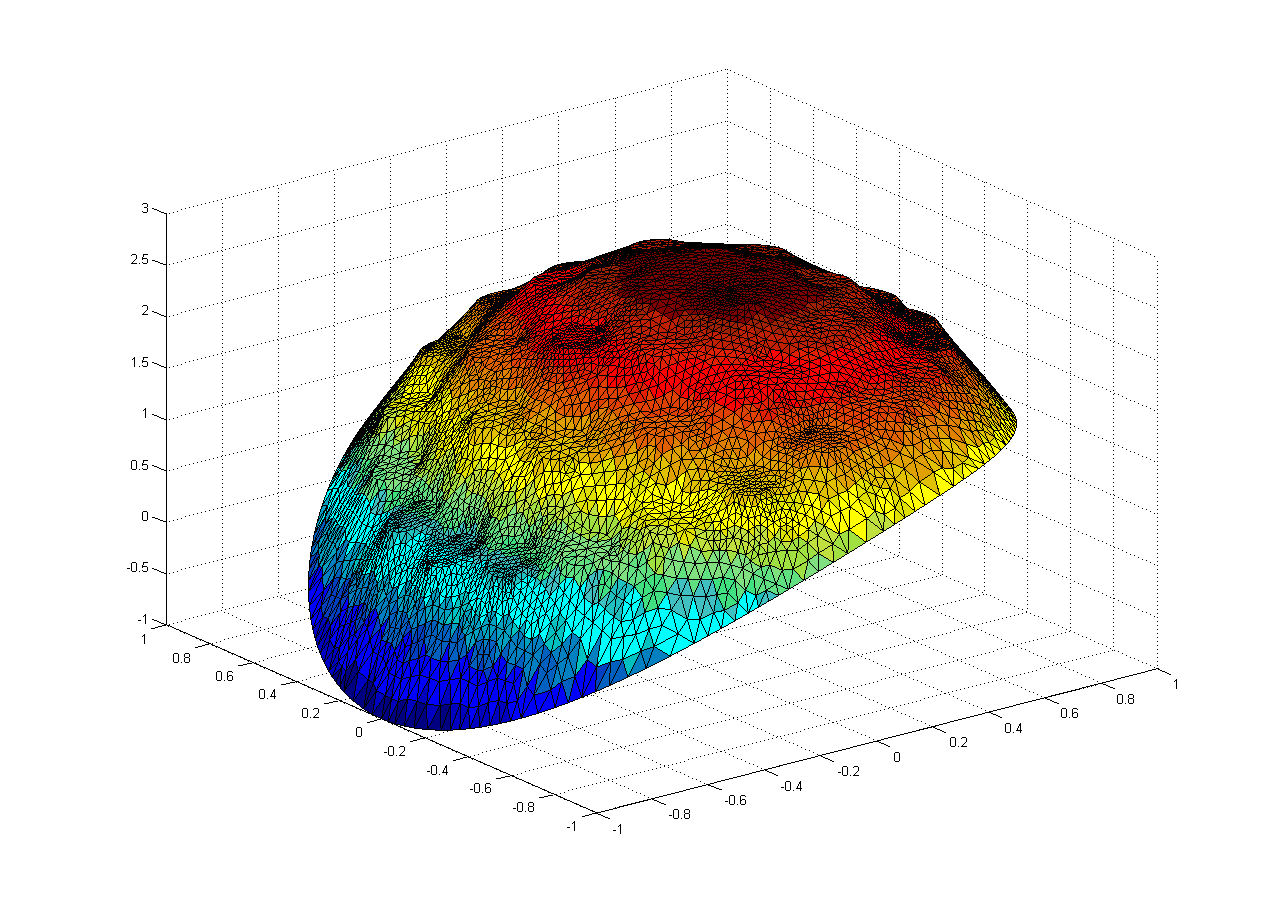}
\includegraphics[width=0.3\textwidth]{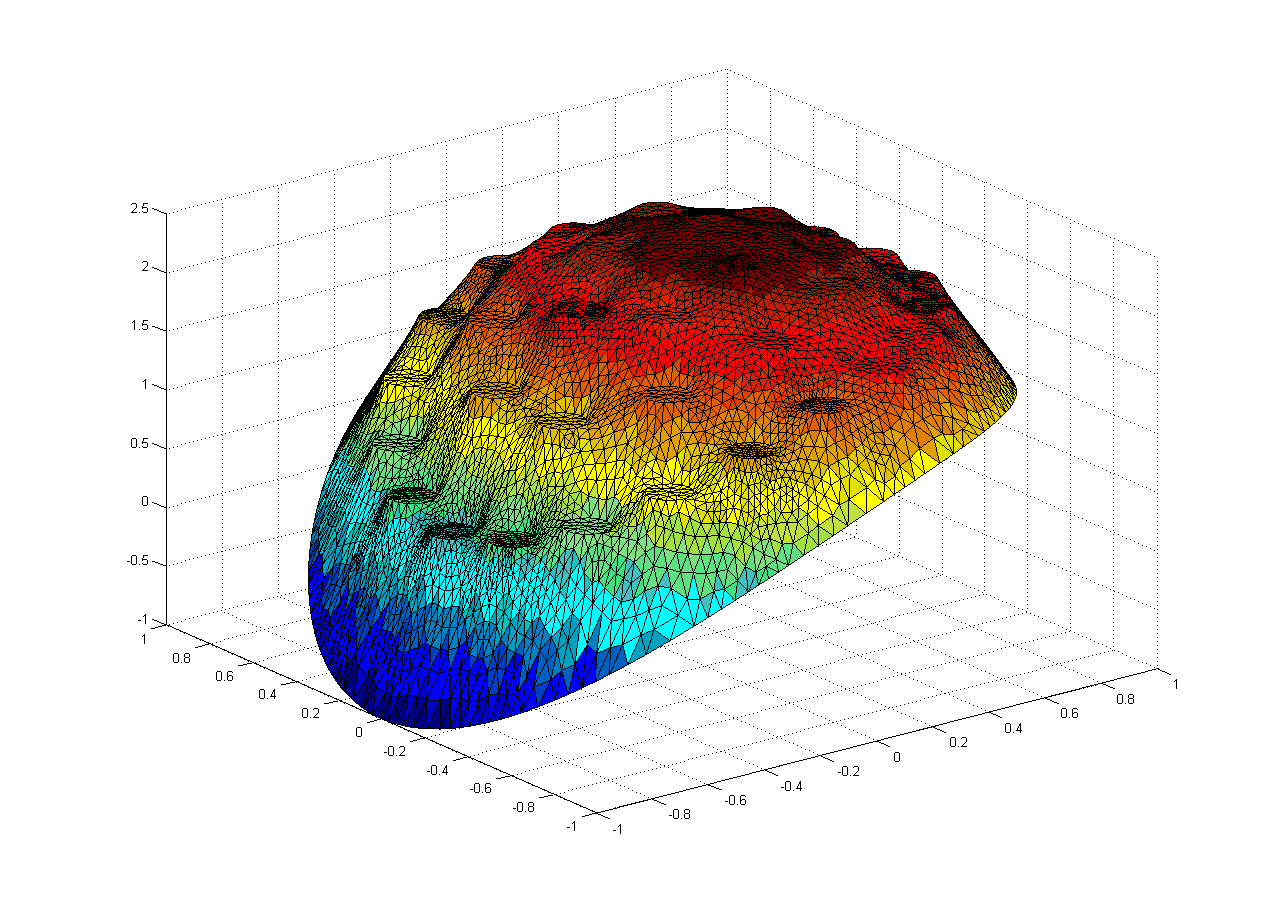}
\caption{Geometry of the problem \eqref{Example1Multiscale}. with $\eta=6$. 
Asymptotic solution $u_0=u_{00}+u_{\coarse}$ homogeneously}
\label{fig:figure1} 
\end{figure}

In the Figure \ref{fig:figure2} we show the computed two terms $u_1$ and $u_2$ of the asymptotic expansion. In the Top we plot the behavior of functions $u_1$ and $u_2$ in the background. For the Botton of the Figure \ref{fig:figure2} we show the same functions only restricted to the inclusions.\\

An interesting case for different values of $\eta$ and the number of terms needed to obtain the relative error of the approximation is reported in the Table \ref{table:3.1}. We observe that to the problem \eqref{Example1Multiscale} is necessary using, for instance with a value of $\eta=10^8$ one term of the expansion for that the relative error was approximately of the order of $10^{-8}$.

\begin{figure}[!htb]
\centering
\includegraphics[width=0.3\textwidth]{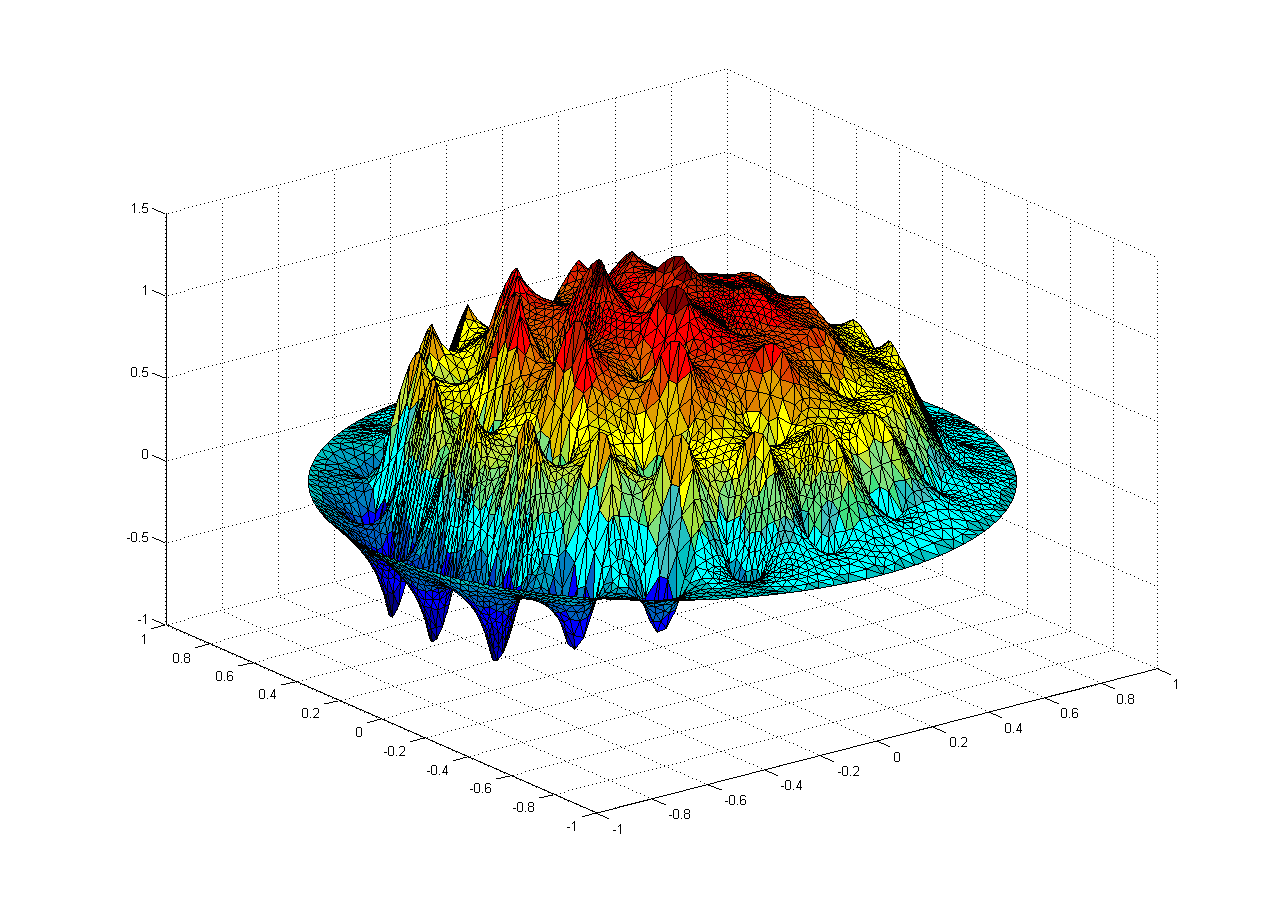}
\includegraphics[width=0.3\textwidth]{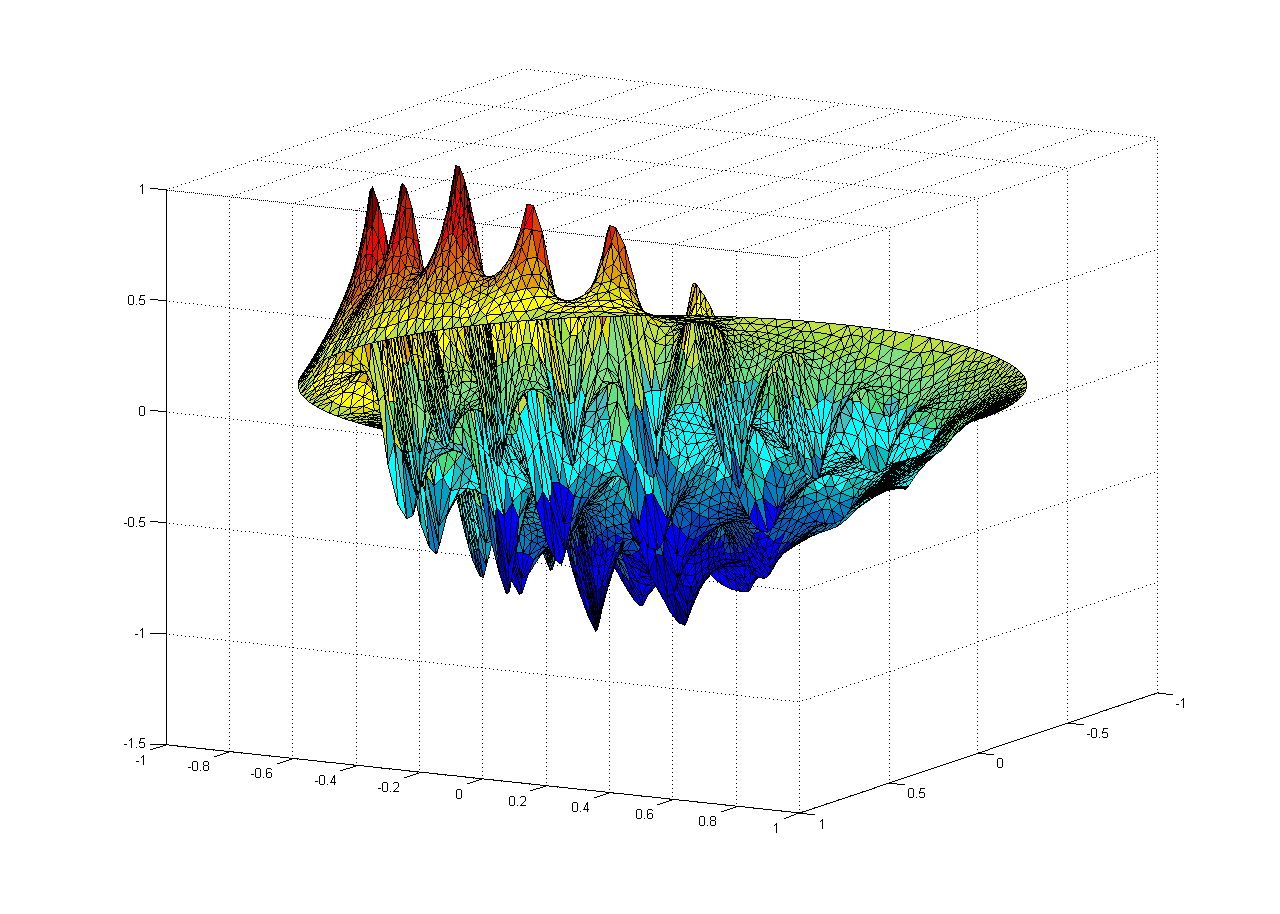}\\
\includegraphics[width=0.3\textwidth]{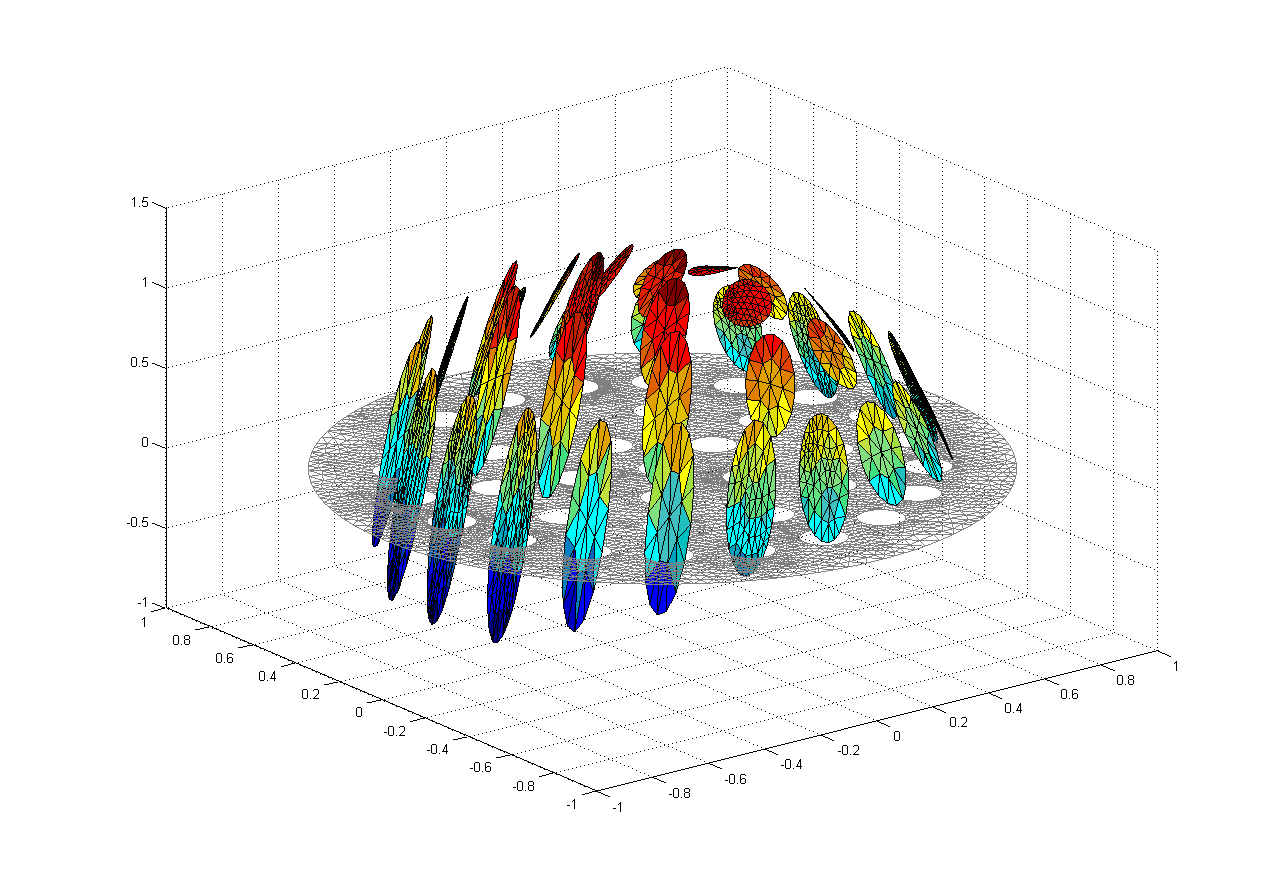}
\includegraphics[width=0.3\textwidth]{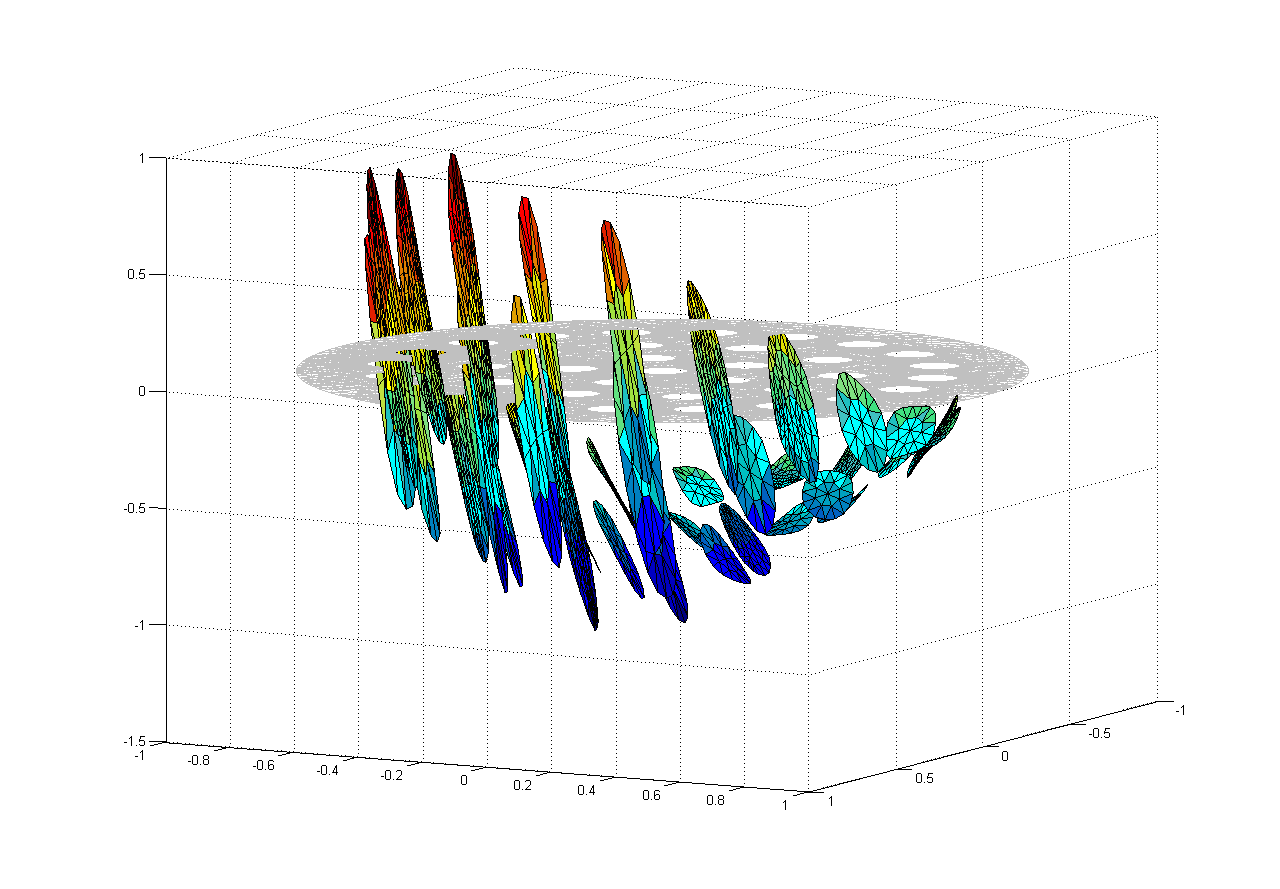}
\caption{Top: Functions $u_1$ and $u_2$. Bottom: 
Functions $u_1$ and $u_2$ restricted to the inclusions.}
\label{fig:figure2} 
\end{figure}

\begin{table}[!h]\
\begin{center}\footnotesize
\begin{tabular}{|c|c|c|c|c|c|c|c|c|c|c|c|c|c|c|c|} \hline\hline
$\eta$ & 3 &4 &5&6&7&8&9&10&$10^2$&$10^3$&$10^4$& $10^5$&$10^6$&$10^7$&$10^8$\\ \hline
\# & 25 &    16&    13&    11&    10&     9&     8&     8&     4&  
   3&     2&     2&     2&     1&     1 \\
\hline\hline
\end{tabular}
\end{center}
\caption{Number of terms needed to obtain a relative error 
of $10^{-8}$ for a given value of $\eta$.}\label{table:3.1}
\end{table}


\section{Applications to Multiscale Finite Elements: Approximation of $u_0$ with Localized Harmonic Characteristic}

In Section \ref{SectionNumerical} we showed examples of the expansions terms in two dimensions. In this section we also present some numerically examples, but in this case we consider that the domain $D$ is the union of a background and multiple inclusions homogeneously distributed. In particular, we compute a few terms using Finite Element Method as above. The main issue here is that, instead of using the harmonic characteristic functions defined in \eqref{chi multi}, we use a modification of them that allows computation in a local domain (instead of the whole background domain as in Section \ref{SectionNumerical}).\\

In particular, we consider one approximation of the terms of the expansion $u_0$ with localized harmonic characteristics. This is motivated due to the fact that, a main difficulty is that the computation of the harmonic characteristic functions is computationally expensive. One option is to approximate functions by solving a local problem (instead of a whole background problem). For instance, the domain where a harmonic characteristic functions can computed is illustrated in Figure \ref{fig:deltaneigh} where the approximated harmonic characteristic function will be zero on the a boundary of a neighborhood of the inclusions. This approximation will be consider for the case of many (highly dense) high-contrast inclusions. The analysis of the resulting methods is under study and results will be presented elsewhere.\\

\begin{figure}[!h]
\centering
\includegraphics[width=0.3\textwidth]{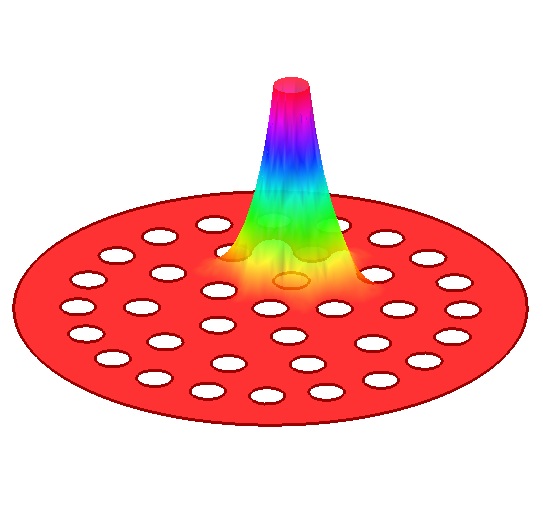}
\includegraphics[width=0.3\textwidth]{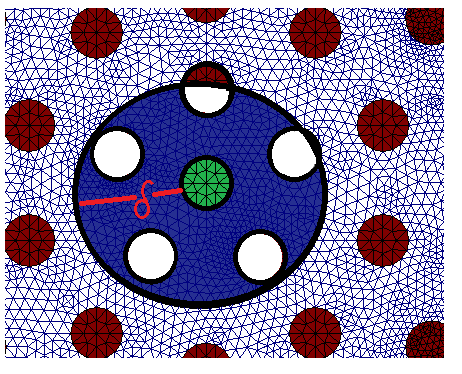}
\caption{Illustration of the Globally characteristic function  and $\delta-$ neighborhood of an inclusion (blue region).}
\label{fig:deltaneigh} 
\end{figure}

We know that the harmonic characteristic function are define by solving a problem in the background, which is a global problem. The idea is then to solve a similar problem only a neighborhood of each inclusion.\\

The exact characteristic functions are defined by the solution of 
\[
\left\{\begin{array}{ll}
\int_D\nabla \chi_m\cdot \nabla v=0, & \mbox{ for }v\in H_0^1(D_0)\\
\chi_m=0, & \mbox{ on }\partial D_\ell, \, m\neq \ell, \, \ell=1,2,\dots,M\\
\chi_m=1, & \mbox{ on }\partial D_m,
\end{array}\right.
\]
for $m\geq 1$. We define the neighborhood of the inclusion $D_m$ by
\[
D_{m,\delta}=\overline{D}_{m}\cup\left\{x\in D_0:d(x,D_m)<\delta\right\},\quad\mbox{(see Figure \ref{fig:deltaneigh})},
\]
and we define the characteristic function for the $\delta$-neighborhood
\[
\left\{\begin{array}{ll}
\int_D\nabla \chi_m^{\delta}\cdot \nabla v=0, & \mbox{ for }v\in H_0^1(D_{m,\delta})\\
\chi_m^{\delta}=0, & \mbox{ on }\partial D_{\ell} \cap \partial D_{m,\delta}\\
\chi_m^{\delta}=1, & \mbox{ on }\partial D_m.
\end{array}\right.
\]
We remain that the exact expression for $u_0$ is given by
\begin{equation}\label{u0multiscale}
u_0=u_{0,0}+\sum_{m=1}^{M}c_m\chi_{D_m},
\end{equation}
and the matrix problem for $u_0$ with globally supported basis is then,
\[
\mathbf{Ac}=\mathbf{b},
\]
with $\mathbf{A}=[a_{m\ell}]$, with $a_{m\ell}=\int_{D_m}\nabla \chi_m \cdot \nabla \chi_{\ell}$, $\mathbf{c}=[c_0,\dots,c_M]$ and $\mathbf{b}=[b_{\ell}]=\int_Df\chi_{D_{\ell}}$. Define $u_0^{\delta}$ using similar expansion given by
\begin{equation}\label{u0trnmultiscale}
u_0^{\delta}=u_{0,0}^{\delta}+\sum_{m=1}^{M}c_m^{\delta}\chi_{D_m}^{\delta},
\end{equation}
where $c_m^{\delta}$ is computed similarly to $c_{m}$ using an alternative matrix problem with basis $\chi_{D_m}^{\delta}$ instead of $\chi_{D_m}$. This system is given by
\[
\mathbf{A}^{\delta}\mathbf{c}^{\delta}=\mathbf{b}^{\delta}.
\]
with $\mathbf{A}^{\delta}=[a_{m\ell}^{\delta}]$, with $a_{m\ell}^{\delta}=\int_{D_m}\nabla \chi^{\delta}_m \cdot \nabla \chi^{\delta}_{\ell}$, $\mathbf{c}=[c^{\delta}_0,\dots,c^{\delta}_M]$ and $\mathbf{b}^{\delta}=[b_{\ell}^{\delta}]=\int_Df\chi_{D_{\ell}}^{\delta}$.\\
$u_{0,0}^{\delta}$ is an approximation to $u_{0,0}$ that solves the problem
\begin{align}
\int_{D^{\delta}}\nabla u_{0,0}^{\delta}\cdot \nabla v& = \int_{D^{\delta}}fv, &&\mbox{for all}v \in
H_0^1(D^{\delta})\\
u& =g, && \mbox{on }\partial D\\ 
u&=0, &&\mbox{on }\partial D^{\delta}\cap D,
\end{align}
where
\[
D^{\delta}=\left\{x\in D:d(x,\partial D)<\delta\right\}.
\]
We recall relative errors given by 
\begin{align}
e(u_0-u_0^{\delta})&=\frac{\|u_0-u_0^{\delta}\|_{H^1}}{\|u_0\|_{H^1}},&&\\
e(u_{0,0}-u_{0,0}^{\delta})&=\frac{\|u_{0,0}-u_{0,0}^{\delta}\|_{H^1}}{\|u_{0,0}\|_{H^1}}.&&
\end{align}
Which relate the exact solution \eqref{u0multiscale} and the alternative solution  \eqref{u0trnmultiscale}. We use the numerical implementation for this case and show two numerical examples.

\subsection{Example 1: 36 Inclusions}\label{Sec:3.3.1}

We study of the expansion with $\chi$ in problem \eqref{Example1Multiscale} was presented in the Section \ref{Sec:3.2.3}. The geometry of this problem is illustrated in the Figure \ref{fig:figure1}.\\

\begin{table}[!h]
\begin{center}\footnotesize
\begin{tabular}{|c|c|c|c|} \hline
$\delta$ & $e(u_0-u_{0}^{\delta})$ & 
$e(u_{0,0}-u_{0,0}^{\delta})$&$e(u_{c}-u_{c}^{\delta})$ \\
\hline\hline
 0.001000 & 0.830673 & 0.999907& 0.555113\\ 
  0.050000 & 0.530459 & 0.768135& 0.549068\\ 
  0.100000 & 0.336229 & 0.639191& 0.512751\\ 
  0.200000 & 0.081500 & 0.261912& 0.216649\\ 
  0.300000 & 0.044613 & 0.088706& 0.048173\\ 
  0.400000 & 0.041061 & 0.047743& 0.007886\\ 
  0.500000 & 0.033781 & 0.034508& 0.001225\\ 
  0.600000 & 0.029269 & 0.029362& 0.000174\\ 
  0.700000 & 0.020881 & 0.020888& 0.000021\\ 
  0.800000 & 0.012772 & 0.012773& 0.000003\\ 
  0.900000 & 0.006172 & 0.006172& 0.000000\\  \hline
\end{tabular}\\
\includegraphics[width=0.5\textwidth]{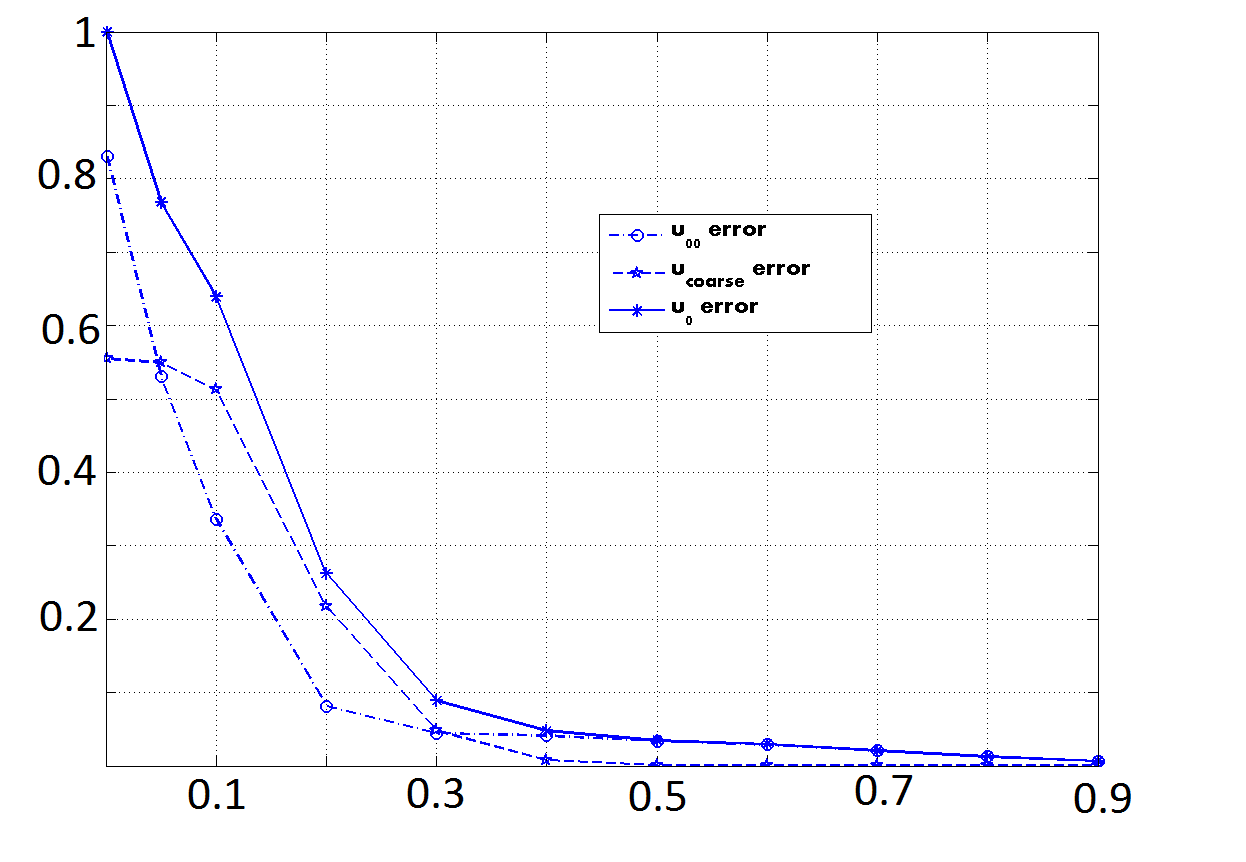}
\end{center}
\caption[Relative error in the approximation of $u_0$ by using locally computed 
basis functions and truncated boundary condition effect. 
$e(w)=\|w\|_{H^1}/\|u_0\|_{H^1}$.]{ Relative error in the approximation of $u_0$ by using locally computed 
basis functions and truncated boundary condition effect. 
$e(w)=\|w\|_{H^1}/\|u_0\|_{H^1}$. Here 
$u_0=u_{0,0}+u_c$ where $u_c$ is combination of 
harmonic characteristic functions and $u_{0}^{\delta}=u_{0,0}^{\delta}+u_{c}^{\delta}$
is computed by solving $u_{0,0}^{\delta}$ on a $\delta-$strip of the boundary 
$\partial D$ and the basis functions on a $\delta-$strip of the 
boundary of each inclusion.}\label{tab:deltaerror}
\end{table}

We have the numerical results for the different values of $\delta$ in the Table \ref{tab:deltaerror}. We solve for the harmonic characteristic function with zero Dirichlet boundary condition within a distance $\delta$ of the boundary of the inclusion.\\

In the Table \ref{tab:deltaerror} we observe that as the neighborhood approximates total domain, its error is relatively small. For instance, if we have a $\delta = 0.2$, the error that relates the exact solution $u_0$ and the truncated solution $u_0^{\delta}$ is around $8\%$. In analogous way an error presents around $6 \%$ between the function $u_{0,0}$ and $u_{0,0}^{\delta}$.\\

\subsection{Example 2: 60 Inclusions}

Similarly, we study of the expansion with $\chi$ in $D=(0,1)$  the circle with center $(0,0)$, radius $1$ and $60$ (identical) circular inclusions of radius $0.07$. We numerically solve the following problem 
\begin{equation}\label{Example2Multiscale}
\left\{\begin{array}{ll}
-\dive(\kappa(x)\nabla u_\eta(x))=1, & \mbox{ in }D\\
\hspace{0.1in}u(x)=x_1+x_2^2, & \mbox{ on }\partial D,
\end{array}\right.
\end{equation}

\begin{figure}[!h]
\centering
\includegraphics[width=0.35\textwidth]{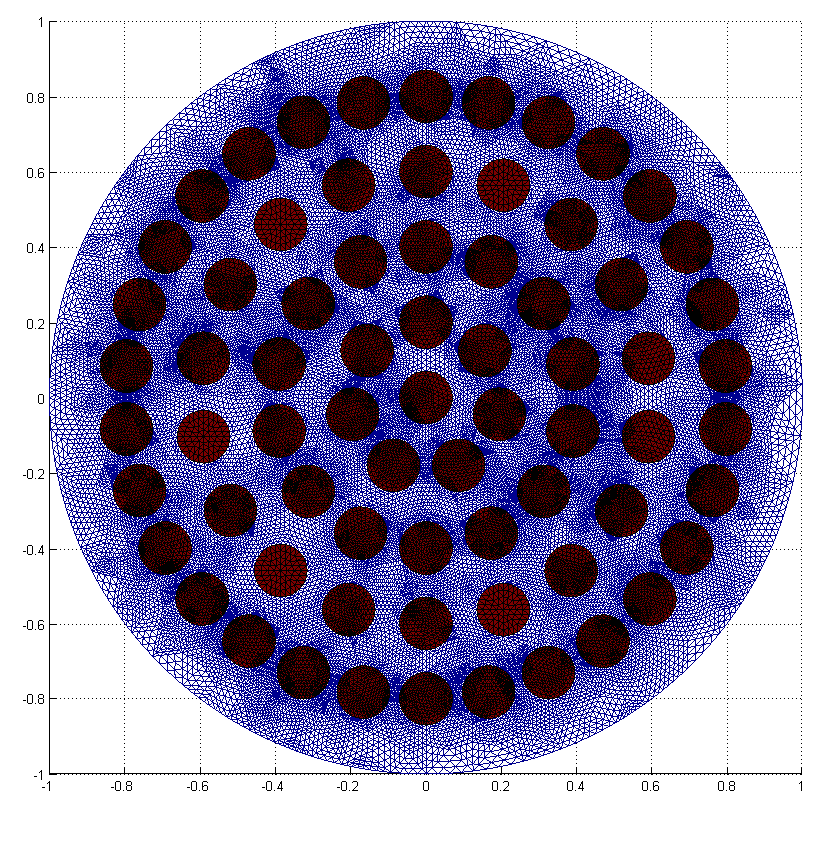}
\caption{Geometry for the problem \eqref{Example2Multiscale}.}
\label{fig:decay} 
\end{figure}

In the Figure \ref{fig:decay} we illustrate the geometry for the problem \eqref{Example2Multiscale}. We have the similar numeric results for the different values of $\delta$ in the Table \ref{tab:deltaerror2} for the problem \eqref{Example2Multiscale}. We solve for the harmonic characteristic function with zero Dirichlet boundary condition within a distance $\delta$ of the boundary of the inclusion.\\

For this example we have a difference in the numerically analysis to the problem \eqref{Example1Multiscale} in the Section \ref{Sec:3.3.1} it is due to numbers inclusions presents in the problem \eqref{Example2Multiscale}. In the Table \ref{tab:deltaerror2} we observe that as the neighborhood approximates total domain, its error is relatively small. For instance, if we have a $\delta = 0.20000$, the error that relates the exact solution $u_0$ and the truncated solution $u_0^{\delta}$ is around $1\%$. The relative error is $2\%$ between the function $u_{0,0}$ and $u_{0,0}^{\delta}$.\\

We state that if there are several high-conductivity inclusions in the geometry, which are closely distributed,  the localized harmonic characteristic function decay rapidly to zero for any $\delta$-neighborhood.\\


\textcolor{white}{We state that if there are several high-conductivity inclusions in the geometry, which are closely distributed,  the localized harmonic characteristic function decay rapidly to zero for any $\delta$-neighborhood.
We state that if there are several high-conductivity inclusions in the geometry, which are closely distributed,  the localized harmonic characteristic function decay rapidly to zero for any $\delta$-neighborhood.
}
\begin{table}
\begin{center}\footnotesize
\begin{tabular}{|c|c|c|c|} \hline\hline
$\delta$ & $e(u_0-u_{0}^{\delta})$ & 
$e(u_{00}-u_{0,0}^{\delta})$&$e(u_{c}-u_{c}^{\delta})$ \\
\hline\hline
0.001000 & 0.912746 & 0.999972& 0.408063\\ 
  0.050000 & 0.369838 & 0.549332& 0.399472\\ 
  0.100000 & 0.181871 & 0.351184& 0.258946\\ 
  0.200000 & 0.013781 & 0.020172& 0.011061\\ 
  0.300000 & 0.013332 & 0.013433& 0.000737\\ 
  0.400000 & 0.010394 & 0.010396& 0.000057\\ 
  0.500000 & 0.009228 & 0.009228& 0.000004\\ 
  0.600000 & 0.006102 & 0.006102& 0.000000\\ 
  0.700000 & 0.005561 & 0.005561& 0.000000\\ 
  0.800000 & 0.002239 & 0.002239& 0.000000\\ 
  0.900000 & 0.001724 & 0.001724& 0.000000\\ \hline
\hline
\end{tabular}\\
\includegraphics[width=0.5\textwidth]{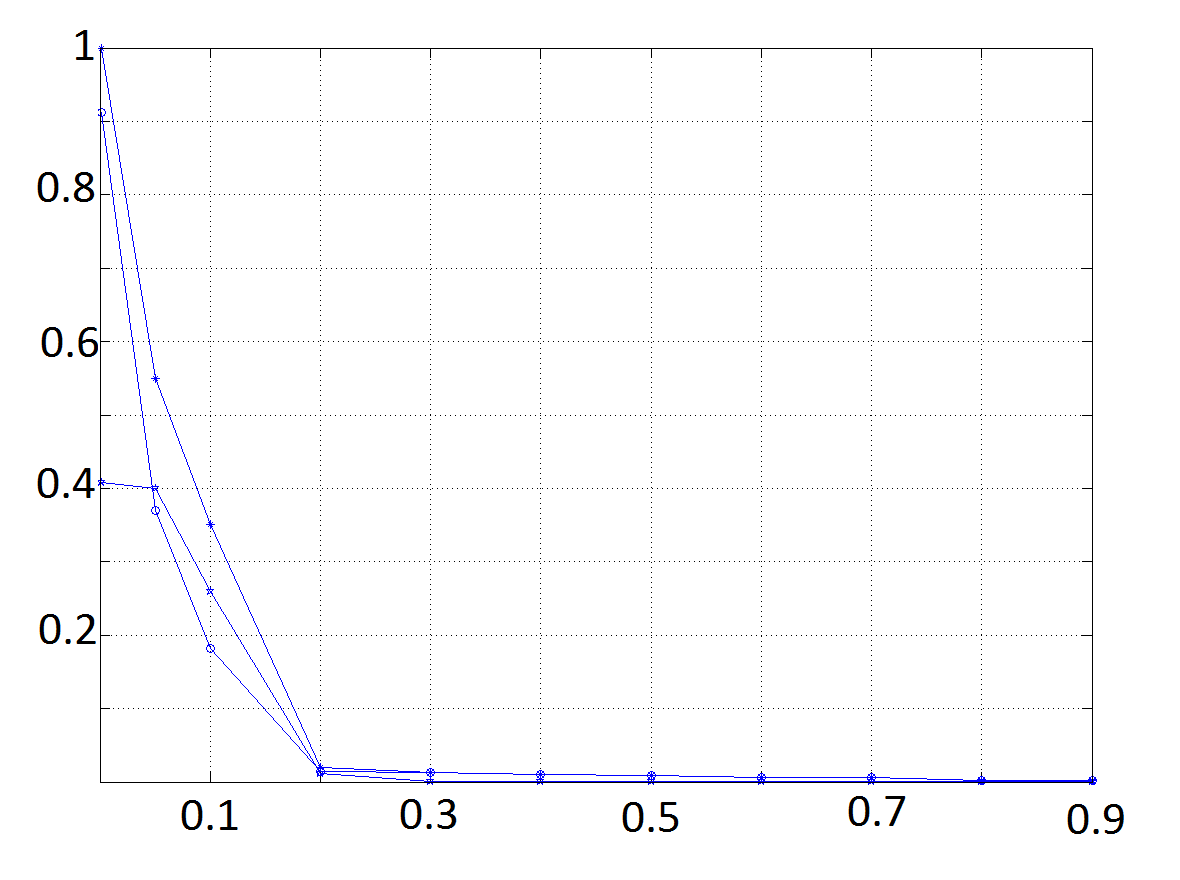}
\end{center}
\caption{ Relative error in the approximation of $u_0$ by using locally computed 
basis functions and truncated boundary condition effect. 
$e(w)=\|w\|_{H^1}/\|u_0\|_{H^1}$. Here 
$u_0=u_{0,0}+u_c$ where $u_c$ is combination of 
harmonic characteristic functions and $u_{0}^{\delta}=u_{0,0}^{\delta}+u_{c}^{\delta}$
is computed by solving $u_{0,0}^{\delta}$ on a $\delta-$strip of the boundary 
$\partial D$ and the basis functions on a $\delta-$strip of the 
boundary of each inclusion.}\label{tab:deltaerror2}
\end{table}



\chapter{Asymptotic Expansions for High-Contrast Linear Elasticity}\label{Chapter4}

\minitoc

In this chapter we recall one of the fundamental problems in solid mechanics, in which relates (locally) strains and deformations. This relationship is known as the constitutive law, that depending of the type of the material. We focus in the formulation of the constitutive models, which present a measure of the stiffness of an elastic isotropic material and is used to characterize of these material. Although, this formulation of models is complex when the deformations are very large.

\section{Problem Setting}\label{sec:problem setting}

In this section we consider the possible simplest case, the linear elasticity. This simple model is in the base of calculus of solid mechanics and structures. The fundamental problem is the formulation of the constitutive models express functionals that allow to compute the value of the stress $\tau$ at a point from of the value of the strain $\epsilon$ in that moment and in all previous.\\

Let $D\subset \R^d$ polygonal domain or a domain with smooth boundary. Given $u\in H^1(D)^d$ and a vector field $f$, we consider the linear elasticity Dirichlet problem
\begin{equation} \label{sformelast}
-\dive (\tau(u))=f \ \mbox{in }D, 
\end{equation} 
with $u=g$ on $\partial D$. Where $\tau(u)$ is stress tensor, defined in the Section \ref{LinearElasticity}.\\

We define the following weak formulation for the problem \eqref{sformelast}: find $u\in H_{0}^{1}(D)^{d}$ such that
\begin{equation}\label{eq:7}
\left\{\begin{array}{ll}
 {\cal A}(u,v)={\cal F}(v), & \mbox{for all } v\in H_0^1(D)^d,\\
\hspace{.37in} u=g, &\mbox{ on } \partial D,
\end{array}\right.
\end{equation}
where the bilinear form ${\cal A}$ and the linear functional $f$ are defined by
\begin{align}\label{eq:8}
{\cal A}(u,v)&=\int_{D}2\widetilde{\mu}E\epsilon(u)\cdot \epsilon(v)+\widetilde{\lambda}E\dive (u)\dive (v),  
 &&\mbox{for all }  u,v\in H_0^1(D)^d\\
{\cal F}(v)&=\int_Dfv, &&\mbox{for all } v\in H_0^1(D)^d. 
\end{align}
Here, the functions $\widetilde{\mu},E,\epsilon$ and $\widetilde{\lambda}$ are defined in the Section \ref{LinearElasticity}. As above, we assume that $D$ is the disjoint union of a background domain and inclusions, that is, $D=D_{0} \cup (\bigcup_{m=1}^{M}\overline{D}_m)$. We assume that $D_0, D_1, \dots ,D_M$ are polygonal domains or domains with smooth boundaries. We also assume that each $D_m$  is a connected domain, $m=0,1,\dots ,M$. Let $D_0$ represent the background domain and the sub-domains $\{D_m\}_{m=1}^M$ represent the inclusions. For simplicity of the presentation we consider only interior inclusions.
Given $w\in H^1(D)^d$ we will use the notation $w^{(m)}$, for the restriction of $w$ to the domain $D_m$, that is
\[
w^{(m)}=w|_{D_m}, \quad m=0,1,\dots ,M.
\]
For later use we introduce the following notation. Given $\Omega \subset D$ we denote by ${\cal A}_{\Omega}$ the bilinear form
\[
{\cal A}_\Omega (u,v)=\int_\Omega 2\widetilde{\mu}\epsilon(u)\cdot \epsilon(v)+\widetilde{\lambda}\dive (u)\dive (v), 
\]
defined for functions $u,v\in H^1(\Omega)^d$. Note that ${\cal A}_\Omega$ does not depend on the Young's modulus $E(x)$. We also denote by ${\cal RB}(\Omega)$ the subset of \emph{rigid body motions} defined on $\Omega$,
\begin{equation}\label{eq:10}
\left\{\begin{array}{ll}
 {\cal RB}(\Omega)=\{(a_1,a_2)+b(x_2,-x_1): \ a_1,a_2,b \in \R \} & \mbox{if } d=2\\
 \\
{\cal RB}(\Omega)=\{(a_1,a_2,a_3)+(b_1,b_2,b_3)\times (x_1,x_2,x_3): \ a_i,b_i \in \R , i=1,2,3. \} & \mbox{if } d=3.
\end{array}\right.
\end{equation}
Observe that for $d=2$, ${\cal RB}(\Omega)=\Span\left\{ (1,0),(0,1),(x_2,-x_1)\right\}$. For $d=3$ the dimension is $6$.

\section{Expansion for One Highly Inelastic Inclusions}

In this section we derive and analyze expansion for the case of highly inelastic inclusions. For the sake of readability and presentation, we consider first the case of only one highly inelastic inclusion in Section \ref{subsec:3.1}.

\subsection{Derivation for One High-Inelastic Inclusion}\label{subsec:3.1}
Let $E$ be defined by
\begin{equation} \label{eq:11}
E(x)=\left\{\begin{array}{ll}
\eta, & x \in D_1,\\
1, & x \in D_0=D\setminus \overline{D}_1,
\end{array}\right.
\end{equation}
and denote by $u_\eta$ the solution of the weak formulation \eqref{eq:7}. We assume that $D_1$ is compactly included in $D$ ($\overline{D}_1 \subset D$). Since $u_\eta$ is solution of \eqref{eq:7} with the coefficient \eqref{eq:11}, we have
\begin{equation}\label{eq:12}
{\cal A}_{D_0}(u_\eta,v)+\eta {\cal A}_{D_1}(u_\eta,v)={\cal F}(v),\quad  \mbox{for all } v \in H_0^1(D).
\end{equation}
We seek to determine $\{u_j\}_{j=0}^\infty \subset H^1(D)^d$ such that 
\begin{equation} \label{eq:13}
u_\eta =u_0+\frac{1}{\eta}u_1+\frac{1}{\eta ^2}u_2+\cdots =\sum_{j=0}^\infty \eta^{-j}u_j,
\end{equation}
and such that they satisfy the following Dirichlet boundary conditions
\begin{equation}\label{eq:14}
u_0=g \mbox{ on } \partial D \quad \mbox{and } \quad u_j=0 \mbox{ on } \partial D \mbox{ for } j\geq 1. 
\end{equation}
We substitute \eqref{eq:13} into \eqref{eq:12} to obtain that for all $v\in H_0^1(D)$ we have
\begin{equation} \label{eq:15new}
\eta{\cal A}_{D_1}(u_0,v)+\sum_{j=0}^\infty \eta^{-j}({\cal A}_{D_0}(u_j,v)+{\cal A}_{D_1}(u_{j+1},v))={\cal F}(v).
\end{equation}
Now we collect terms with equal powers of $\eta$ and analyze the resulting subdomains equations.

\subsubsection{Term corresponding to $\eta =\eta ^1$}
In \eqref{eq:15new} there is one term corresponding to $\eta$ to the power $1$, thus we obtain the following equation
\begin{equation} \label{eq:16}
{\cal A}_{D_1}(u_0,v)=0, \mbox{ for all }v\in H_0^1(D)^d.
\end{equation}
The problem above corresponds to an elasticity equation posed on $D_1$ with homogeneous Neumann boundary condition. Since we are assuming that $\overline{D}_1\subset D$, we conclude that $u_0^{(1)}$ is a \emph{rigid body motion}, that is, $u_0^{(1)}\in {\cal RB}(D_1)$ where ${\cal RB}$ is defined in \eqref{eq:10}. \\

In the general case, the meaning of this equation depends on the relative position of the inclusion $D_1$ with respect to the boundary. It may need to take the boundary data into account.

\subsubsection{Terms corresponding to $\eta^0=1$}
The equation \eqref{eq:14} contains three terms corresponding to $\eta$ to the power $0$, which are
\begin{equation} \label{eq:17new}
{\cal A}_{D_0}(u_0,v)+{\cal A}_{D_1}(u_1,v)={\cal F}(v), \mbox{ for all }\in H_0^1(D)^d.
\end{equation}
Let
\[
V_{\cal RB}=\{ v\in H_0^1(D)^d, \mbox{ such that } v^{(1)}=v|_{D_1} \in {\cal RB}^1(D_1)\}.
\]
If consider  $z\in V_{\cal RB}$ in equation \eqref{eq:17new} we conclude that $u_0$ satifies the following problem
\begin{align}\label{eq:18new}
{\cal A}_{D_0}(u_0,z)&=\int_{D}fz, &&\mbox{for all }  z\in V_{\cal RB}\\
u_0&=g &&\mbox{ on } \partial D. 
\end{align}
The problem \eqref{eq:18new} is elliptic and it has a unique solution, for details see \cite{MR1477663}. To analyze this problem further we proceed as follows. Let $\{ \xi_{D_1;\ell}\}_{\ell=1}^{L_d}$ be a basis for the ${\cal RB} (D_1)$ space. Note that $L_2=3$ and $L_3=6$ then $u_0^{(1)}=\sum_{\ell=1}^{L_d}c_{0;\ell}\xi_{D_1;\ell}$. We define the harmonic extension of the rigid body motions, $\chi_{D_1;1}\in H_0^1(D)^d$ such that
\[
\chi_{D_1;\ell}^{(1)}=\xi_{D_1;\ell},\quad \mbox{ in } D_1
\]
and as the harmonic extension of its boundary data in $D_0$ is given by
\begin{align}\label{eq:20}
{\cal A}_{D_0}(\chi _{D_1;\ell}^{(0)},z)&=0, &&\mbox{for all }  z\in H_0^1(D_0)^d \nonumber \\
\chi_{D_1;\ell}^{(0)}&=\xi_{D_1;\ell}, &&\mbox{ on } \partial D_1.\\
 \chi_{D_1;\ell}^{(0)}&=0, &&\mbox{ on } \partial D.\nonumber
\end{align}
To obtain an explicit formula for $u_0$ we will use the following facts: 
\begin{enumerate}
\item [(i)] problem \eqref{eq:18new} is elliptic and has a unique solution, and
\item [(ii)] a property of the harmonic characteristic functions described in the Remark below.
\end{enumerate}
\begin{remark} \label{remark1}
Let $w$ be a harmonic extension to $D_0$ of its Neumann data on $\partial D_0$. That is, $w$ satisfies the following problem
\[
{\cal A}_{D_0}(w,z)=\int_{\partial D_0}\widetilde{\tau}(w)\cdot n_0z, \quad \mbox{ for all }z\in H^1(D_0)^d.
\]
Since $\chi_{D_1;\ell}=0$ in $\partial D$  and $\chi_{D_1;\ell}=\xi_{D_1;\ell}$ on $\partial D_1$, we readily have that
\[
{\cal A}_{D_0}(\chi_{D_1;\ell},w)=0\left( \int_{\partial D}\widetilde{\tau}(w)\cdot n_0z \right)+1\left( \int_{\partial D_1}\widetilde{\tau}(w)n_0\chi_{D_1;\ell}\right),
\]
and we conclude that for every harmonic function on $D_0$
\begin{equation}\label{eq:21}
{\cal A}_{D_0}(w,\chi_{D_1;\ell})=\int_{\partial D_1}\widetilde{\tau}(w)\cdot n_0\chi_{D_1:\ell}.
\end{equation}
In particular, taking $w=\chi_{D_1}$ we have
\begin{equation}\label{eq:22}
{\cal A}_{D_0}(\chi_{D_1;\ell},\chi_{D_1;\ell})=\int_{\partial D_1} \widetilde{\tau}(\chi_{D_1;\ell})\cdot n_0\chi_{D_1;\ell}.
\end{equation}
\end{remark}
We can decompose $u_0$ into harmonic extension of its value in $D_1$, given by $ u_0^{(1)}=\sum_{\ell=1}^{L_d}c_{0;\ell}\xi_{D_1;\ell}$, plus the remainder $u_{0,0}\in H^1(D_0)$. Thus, we write
\begin{equation}\label{eq:23}
u_0=u_{0,0}+\sum_{\ell=1}^{L_d}c_{0;\ell}\xi_{D_1;\ell},
\end{equation}
where $u_{0,0}\in H^1(D)$ is defined by $u_{0,0}^{(1)}=0$ in $D_1$ and $u_{0,0}^{(0)}$ solves the following Dirichlet problem
\begin{align}\label{eq:24}
{\cal A}_{D_0}(u_{0,0}^{(0)},z)&=\int_{D_0}fz, && \mbox{ for all }z\in H_0^1(D_0)^d,  \\
u_{0,0}^{(0)}&=0, && \mbox{ on }\partial D_1, \nonumber\\
u_{0,0}^{(0)}&=g, && \mbox{ on }\partial D \nonumber.
\end{align}
From \eqref{eq:18new} and  \eqref{eq:23} we get that
\begin{equation} \label{eq:25}
\sum_{\ell=1}^{L_d}c_{0;\ell}{\cal A}_{D_0}(\chi_{D_1;\ell},\chi_{D_1;m})=\int_Df\chi_{D_1;m}-{\cal A}_{D_0}(u_{0,0},\chi_{D_1;m}),
\end{equation}
from which we can obtain the constants $c_{0;\ell}$, $\ell=1,\dots,L_d$ by solving a $L_d\times L_d$ linear system. As readily seen, the $L_d\times L_d$ matrix 
\begin{equation}\label{matrixa_lm}
\mathbf{A}_{\geom}=\left[ a_{\ell m}\right]_{\ell, m=1}^{L_d},\quad\mbox{where }a_{\ell m}={\cal A}_{D_0}(\chi_{D_1;\ell},\chi_{D_1;m}).
\end{equation}
The matrix $\mathbf{A}_{\geom}$ is positive and suitable constants exist.
Given the explicit form of $u_0$, we use it in \eqref{eq:17new} to find $u_1^{(1)}=u_1|_{D_1}$, from if conclude that $u_0^{(0)}$ and $u_1^{(1)}$ satisfy the local Dirichlet problems
\begin{align}
{\cal A}_{D_0}(u_0^{(0)},z)&=\int_{D_0}fz, && \mbox{ for all }z\in H_0^1(D_0)^d,\nonumber \\
{\cal A}_{D_1}(u_1^{(1)},z)&=\int_{D_1}fz, && \mbox{ for all }z\in H_0^1(D_1)^d,\nonumber
\end{align}
with given  boundary data $\partial D_0$ and $\partial D_1$. Equation \eqref{eq:17new} also represents the transmission conditions across $\partial D_1$ for the functions $u_0^{(0)}$ and $u_1^{(1)}$. This is easier to see when the forcing $f$ is square integrable. From now on, in order to simplify the presentation, we assume that $f\in L^2(D)$. If $f\in L^2(D)$ , we have that $u_0^{(0)}$ and $u_1^{(1)}$ are the only solutions of the problems
\[
{\cal A}_{D_0}(u_0^{(0)},z)=\int_{D_0}fz+\int_{\partial D_0\setminus \partial D}\widetilde{\tau}(u_0^{(0)})\cdot n_0z, \quad \mbox{ for all }z\in H^1(D_0)^d, 
\]
with $z=0$ on $\partial D$ and $u_0^{(0)}=g$ on $\partial D$, and
\[
{\cal A}_{D_1}(u_1^{(1)},z)=\int_{D_1}fz+\int_{\partial D_1}\widetilde{\tau}(u_1^{(1)})\cdot n_1z, \quad \mbox{ for all }z\in H^1(D_1)^d. 
\]
Replacing these last two equations back into \eqref{eq:17new} we conclude that
\begin{equation}\label{eq:26}
\widetilde{\tau}(u_1^{(1)})\cdot n_1=-\widetilde{\tau}(u_0^{(0)})\cdot n_0, \quad \mbox{ on }\partial D_1.
\end{equation}
Using this interface condition we can obtain $u_1^{(1)}$ in $D_1$ by writing
\begin{equation}\label{eq:27-1}
u_1^{(1)}=\widetilde{u}_1^{(1)}+\sum_{\ell=1}^{L_d}c_{1;\ell}\xi_{D_1;\ell},
\end{equation}
where $\widetilde{u}_1^{(1)}$ solves the Neumann problem
\begin{equation} \label{eq:27}
{\cal A}_{D_1}(\widetilde{u}_1^{(1)},z)=\int_{D_1}fz-\int_ {\partial D_1}\widetilde{\tau}(u_0^{(0)})\cdot n_1z, \quad \mbox{for all }z\in H^1(D_1)^d.
\end{equation}
Here the constants $c_{1;\ell}$ will be chosen later. Problem \eqref{eq:27} needs the following compatibility conditions
\[
\int_{D_1}f\xi+\int_{\partial D_1}\widetilde{\tau}(u_0^{(0)})\cdot n_1\xi =0, \quad \mbox{for all }\xi \in {\cal RB},
\]
which, using \eqref{eq:23} and \eqref{eq:26} and noting that $\chi_{D_1;\ell}$ in $D_1$, reduces to
\begin{equation} \label{eq:28}
\sum_{\ell =1}^{L_d}c_{0;\ell}\int_{\partial D_1}\widetilde{\tau}(\chi_{D_1;\ell})\cdot n_1\chi_{D_1;m}=\int_{D_1}f\chi_{D_1;m}-\int_{\partial D_1}\widetilde{\tau}(\widetilde{u}_{0,0})\cdot n_1\chi_{D_1;m},
\end{equation}
for $m=1,\dots ,L_d$. This system of $L_d$ equations is the same encountered before in \eqref{eq:25}. The fact that the two systems are the same follows from the next two integration by parts relations:
\begin{enumerate}
\item [(i)] according to Remark \ref{remark1}
\begin{equation} \label{eq:29}
\int_{\partial D_1}\widetilde{\tau}(\chi_{D_1;\ell})\cdot n_1\chi_{D_1;m}={\cal A}_{D_0}(\chi_{D_1;\ell},\chi_{D_1;m}).
\end{equation}
\item [(ii)] we have
\begin{equation} \label{eq:30}
\int_{\partial D_1}\widetilde{\tau}(\widetilde{u}_{0,0}\cdot n_1\chi_{D_1;m}={\cal A}_{D_0}(u_{0,0},\chi_{D_1;m})-\int_{D_0}f\chi_{D_1;m}.
\end{equation}
\end{enumerate}
By replacing the relations in (i) and (ii) into above \eqref{eq:28} we obtain \eqref{eq:25} and conclude that the compatibility condition of problem \eqref{eq:27} is satisfied.
Next, we discuss how to compute $u_1^{(0)}$ and $\widetilde{u}_1^{(0)}$ to completely define the functions $u_1\in H^1(D)^d$  and $\widetilde{u}_1\in H^1(D)^d$. These are presented for general $j\geq 1$ since the construction is independent of $j$ in this range.

\subsubsection{Term corresponding to $\eta^{-j}$ with $j\geq 1$}
For powers $1/\eta$ larger or equal to one there are only two terms in the summation that lead to the following system
\begin{equation}\label{eq:31}
{\cal A}_{D_0}(u_j,v)+{\cal A}_{D_1}(u_{j+1},v)=0, \quad \mbox{for all }v\in H_0^1(D)^d.
\end{equation}
This equation represents both the subdomain problems and the transmission conditions across $\partial D_1$ for $u_j^{(0)}$ and $u_{j+1}^{(1)}$. Following a similar argument to the one give above, we conclude that $u_j^{(0)}$ is harmonic in $D_0$ for all $j\geq 1$ and that $u_{j+1}^{(1)}$ is harmonic in $D_1$ for $j\geq 2$. As before, we have
\begin{equation}\label{eq:32}
\widetilde{\tau}(u_{j+1}^{(1)})\cdot n_1=-\tau (u_j^{(0)})\cdot n_0.
\end{equation}
we note that $u_j^{(1)}$ in $D_1$, (e.g., $u_1^{(1)}$ above) is given by the solution of a Neumann problem in $D_1$. Recall that the solution of a Neumann linear elasticity problem is defined up to a rigid body motion. To uniquely determine $u_j^{(1)}$, we write
\begin{equation} \label{eq:33}
u_j^{(1)}=\widetilde{u}_j^{(1)}+\sum_{\ell =1}^{L_d}c_{j;\ell}\xi_{D_1;\ell},
\end{equation}
where $u_j^{(1)}$ is $L^2$-orthogonal to the rigid body motion and the appropriate $c_{j;\ell}$ will be determinated later.
Given $u_j^{(1)}$ in $D_1$ we find $u_j^{(0)}$ in $D_0$ by solving a Dirichlet problem with known Dirichlet data, that is,
\begin{equation}\label{eq:34}
\begin{array}{l}
\mathcal{A}_{D_0}(u_{j}^{(0)}, z)=0,\quad \mbox{ for all } 
z\in H^1_0(D_0)^d \\
u_j^{(0)}=u_j^{(1)} \ \left(=\widetilde{u}_j^{(1)}+\displaystyle \sum_{\ell=1}^{L_d}
c_{j;\ell}\xi_{D_1;\ell}\right), \mbox{ on } 
\partial D_1
\quad \mbox{ and } \quad u_j=0 \mbox{ on } \partial D.
\end{array} 
\end{equation}
We conclude that
\begin{equation} \label{eq:35}
u_j=\widetilde{u_j}+\sum_{\ell =1}^{L_d}c_{j;\ell}\chi_{D_1;\ell},
\end{equation}
where $\widetilde{u}_{j}^{(0)}$ is defined by \eqref{eq:34} replacing $c_{j;\ell}$ by $0$. This completes the construction of $u_j$.
Now we proceed to show how to find $u_{j+1}^{(1)}$ in $D_1$. For this, we use \eqref{eq:30} and \eqref{eq:31} which lead to the following Neumann problem
\begin{equation} \label{eq:36}
{\cal A}_{D_1}(\widetilde{u}_{j+1}^{(1)},z)=-\int_{\partial D_1}\widetilde{\tau}(u_{j}^{(0)})\cdot n_0z, \quad \mbox{for all }z\in H^1(D_1)^d.
\end{equation}
The compatibility condition for this Neumann problem is satisfied if we choose $c_{j;\ell}$ the solution of the $L_d\times L_d$ system 
\begin{equation} \label{eq:37}
\sum_{\ell =1}^{L_d}c_{j;\ell}\int_{\partial D_1}\widetilde{\tau}(\chi_{D_1;\ell}\cdot n_1\chi_{D_1;m}=-\int_{\partial D_1}\widetilde{\tau}(\widetilde{u}_j)\cdot n_1\chi_{D_1;m},
\end{equation}
with $m=1,\dots ,L_d$. As pointed out before, see \eqref{eq:29}, this system can be written as
\[
\sum_{\ell =1}^{L_d}c_{j;\ell}{\cal A}_{D_0}(\chi_{D_1;\ell},\chi_{D_1;m})=-{\cal A}_{D_0}(\widetilde{u}_j,\chi_{D_1;m}).
\]
In this form we readily see that this $L_d\times L_d$ system matrix is positive definite and therefore solvable.
We can choose $u_{j+1}^{(1)}$ in $D_1$ such that
\[
u_{j+1}^{(1)}=\widetilde{u}_{j+1}^{(1)}+\sum_{\ell =1}^{L_d}c_{j+1;\ell}\xi_{D_1;\ell},
\]
where $\widetilde{u}_{j+1}^{(1)}$ is properly chosen and, as before
\[
\sum_{\ell =1}^{L_d}c_{j+1,\ell}\int_{\partial D_1}\widetilde{\tau}(\chi_{D_1;\ell})\cdot \chi_{D_1;m}=-\int_{\partial D_1}\widetilde{u}_{j+1}\cdot n_1\chi_ {D_1;m},
\]
$m=1,\dots , L_d$ and therefore we have the compatibility condition of the Neumann problem to compute $u_{j+2}^{(1)}$. See the equation \eqref{eq:36}.

\subsection{Convergence in $H^1(D)^d$}\label{convergenceelasti}

We study the convergence of the expansion \eqref{eq:13} with the Dirichlet data \eqref{eq:14}. For simplicity of the presentation we consider the case of one high-inelastic inclusions. We assume that $\partial D$ and $\partial D_1$ are sufficiently smooth. We follow the analysis in \cite{calo2014asymptotic}.\\

We use standard Sobolev spaces. Given a sub-domain $D$, we use the $H^1(D)^d$ norm given by
\[
\| v\|_{H^1(D)^d}^2=\|v\|_{L^2(D)^d}^2+\|\nabla v\|_{L^2(D)^d}^2, 
\]
and the seminorm
\[
|v|_{H^1(D)^d}^2=\|\nabla v\|_{L^2(D)^d}^2.
\]
We also use standard trace spaces $H^{1/2}(\partial D)^d$ and dual space $H^{-1}(D)^d$.

\begin{lemma}\label{mainlemma}
Let $\widetilde{w}\in H^1(D)$ be harmonic in $D_0$ and define 
\[
w=\widetilde{w}+\sum_{\ell =1}^{L_d}c_{0;\ell}\xi_{D_1;\ell},
\]
where $Y=(c_{0;1},\dots ,c_{0;L_d})$ is the solution of the system $L_d$-dimensional linear system
\begin{equation}\label{Ld system}
A_\geom Y=-W
\end{equation}
with $W=({\cal A}_{D_0}(\widetilde{w},\xi_{D_1;1}),\dots, {\cal A}_{D_0}(\widetilde{w},\xi_{D_1;L_d}))$. Then,
\[
\|w\|_{H^1(D)^d}\preceq\|\widetilde{w}\|_{H^1(D)^d},
\]
where the hidden constant is the Korn inequality constant of $D$.
\end{lemma}

\begin{proof}
We remember that
\[
\mathbf{A}_{\geom}=\left[a_{\ell m}\right]_{\ell ,m=1}^{L_d},\quad\mbox{with }a_{\ell ,m}={\cal A}_{D_0}(\xi_{D_1;\ell},\xi_{D_1;m}).
\]
Note that $\sum_{\ell =1}^{L_d}c_{0;\ell}\xi_{D_{1};\ell}$ is the Galerkin projection of $\widetilde{w}$ into the space $\Span \{ \xi_{D_1;\ell}\}_{\ell =1}^{L_d}$. Then, as usual in finite element analysis of Galerkin formulations, we have
\begin{eqnarray*}
{\cal A}_{D_0}\left(\sum_{\ell=1}^{L_d}c_{0;\ell}\xi_{D_1;\ell},\sum_{\ell=1}^{L_d}c_{0;\ell}\xi_{D_1;\ell}\right) & = & Y^{T}A_{\geom}Y=-Y^{T}W\\
& = & -\sum_{\ell=1}^{L_d}c_{0;\ell}{\cal A}_{D_0}(\widetilde{w},\xi_{D_1;\ell})\\
& = & -{\cal A}_{D_0}\left(\widetilde{w},\sum_{\ell=1}^{L_d}c_{0;\ell}\xi_{D_1;\ell}\right)\\
& \leq & |\widetilde{w}|_{H^1(D)^d}\left|\sum_{\ell=1}^{L_d}c_{0;\ell}\xi_{D_1;\ell}\right|_{H^1(D_0)^d},
\end{eqnarray*}
by the Korn inequality
\begin{eqnarray*}
\left| \sum_{\ell =1}^{L_d}c_{0;\ell}\xi_{D_1;\ell} \right|_{H^1(D_0)^d}^2 & \leq  & C{\cal A}_{D_0}\left(\sum_{\ell=1}^{L_d}c_{0;\ell}\xi_{D_1;\ell},\sum_{\ell=1}^{L_d}c_{0;\ell}\xi_{D_1;\ell}\right)\\
& \leq & |\widetilde{w}|_{H^1(D_0)^d}\left|\sum_{\ell =1}^{L_d}C_{0;\ell}\xi_{D_1;\ell}\right|_{H^1(D_0)^d},
\end{eqnarray*}
so
\[
\left|\sum_{\ell=1}^{L_d}c_{0;\ell}\xi_{D_1;\ell}\right|_{H^1(D_0)^d}\leq |\widetilde{w}|_{H^1(D_0)}.
\]
Using the fact above, we get
\[
\|w\|_{H^1(D)^d}\leq \| \widetilde{w}\|_{H^1(D)^d}+\left\Vert \sum_{\ell=1}^{L_d}c_{0;\ell}\xi_{D_1;\ell}\right\Vert _{H^1(D)^d}\preceq \| \widetilde{w}\|_{H^1(D)^d}.
\]
\end{proof}

For the prove of the convergence of the expansion \eqref{eq:13} with the boundary condition \eqref{eq:14}, we consider the following additional results obtained by applying the Lax-Milgram theorem in Section \ref{SectionB.3} and the trace theorem in Section \ref{Sec:B.5.2}.   


\begin{lemma}\label{lemmaconver1}
Let $u_0$ in \eqref{eq:23}, with $u_{0,0}$ defined in \eqref{eq:24}, and $u_1$ be defined by \eqref{eq:27} and \eqref{eq:34} with $j=1$. We have that
\begin{equation} \label{eq:4.35}
\|u_0\|_{H^1(D)^d}\preceq  \|f\|_{H^{-1}(D)^d}+\|g\|_{H^{1/2}(\partial D)^d},
\end{equation}
\begin{equation}\label{eq:4.36}
\|\widetilde{u}_1\|_{H^1(D_1)^d} \preceq  \|f\|_{H^{-1}(D_1)^d}+\|g\|_{H^{1/2}(\partial D)^d}
\end{equation}
and
\begin{equation}\label{eq:4.37}
\|\widetilde{u}_1\|_{H^1(D_0)^d}\preceq  \|\widetilde{u}_1\|_{H^{1/2}(\partial D_1)^d} \preceq \|\widetilde{u}_1\|_{H^{1}( D_1)^d}.
\end{equation}
\end{lemma}
\begin{proof}
From the definition of $u_{0,0}$ in (\ref{eq:24}) we have that
\[
\|u_{0,0}\|_{H^1(D_0)^d}\preceq\|f\|_{H^{-1}(D_0)^d}+\|g\|_{H^{1/2}(\partial D)^d},
\]
for more details see for instance, \cite{Adams-book2003}.

Now, using Korn inequality for Dirichlet data and Lax-Milgram theorem we have
\[
\|u_0\|_{H^1(D)^d}\preceq |u_0|_{H^1(D)^d}=|u_0|_{H^1(D_0)^d}\leq |u_{0,0}|_{H^1(D_0)^d}+\left|\sum_{\ell=1}^{L_d}c_{0;\ell}\xi_{D_1;\ell}\right|_{H^1(D_0)^d},
\]
and using a similar fact in the Lemma \ref{mainlemma}, we have that
\begin{eqnarray*}
\left|\sum_{\ell=1}^{L_d}c_{0;\ell}\xi_{D_1;\ell}\right|_{H^1(D_0)^d}^2 & \leq & C\left|u_{0,0}\right|_{H^1(D_0)^d}\left|\sum_{\ell=1}^{L_d}c_{0;\ell}\xi_{D_1;\ell}\right|_{H^1(D_0)^d}+\int_{D} f\left(\sum_{\ell=1}^{L_d}c_{0;\ell}\xi_{D_1;\ell}\right)\\
& \preceq & \left|u_{0,0}\right|_{H^1(D_0)^d}\left|\sum_{\ell=1}^{L_d}c_{0;\ell}\xi_{D_1;\ell}\right|_{H^1(D_0)^d}+\|f\|_{H^{-1}(D)^d}\left|\sum_{\ell=1}^{L_d}c_{0;\ell}\xi_{D_1;\ell}\right|_{H^1(D_0)^d},
\end{eqnarray*}
so
\[
\left|\sum_{\ell=1}^{L_d}c_{0;\ell}\xi_{D_1;\ell}\right|_{H^1(D_0)^d}\preceq \left|u_{0,0}\right|_{H^1(D_0)^d}+\|f\|_{H^{-1}(D)^d}.
\]
Using this fact and the definition of $u_{0,0}$, we conclude that
\[
\|u_0\|_{H^1(D)^d}\preceq\|f\|_{H^{-1}(D)^d}+\|g\|_{H^{1/2}(\partial D)^d}.
\]
This concludes the proof.\\

The equation (\ref{eq:4.36}) uses the similar argument well as in the above proof, we use the Korn inequality for Neumann conditions and the Trace theorem. Finally, the equation (\ref{eq:4.37}) is given using Korn inequality for Dirichlet conditions and the Trace Theorem.
\end{proof}
\begin{lemma}\label{lemmaconver2}
Let $u_j$ defined on $D_0$ by \eqref{eq:34} with $c_{j;\ell}$ and $u_{j+1}$ defined on $D_1$ by \eqref{eq:36}. For $j\geq 1$ we have that
\[
\|u_{j+1}\|_{H^1(D)^d}\preceq \|u_j\|_{H^1(D_0)^d}.
\] 
\end{lemma}

\begin{proof}
Let $j\geq 1$. Consider $\widetilde{u}_{j+1}$ defined by the Dirichlet in \eqref{eq:34}. From the Lemma \ref{mainlemma} and combining \eqref{dirichletprob} with the trace theorem \ref{TheoremA.2}, we have
\[
\|u_{j+1}\|_{H^1(D)^d}\preceq \|\widetilde{u}_{j+1}\|_{H^1(D)^d}\leq C\|\widetilde{u}_{j+1}\|_{H^1(D_1)^d},
\]
applying \eqref{dirichletprob} in the last equation we obtain
\[
\|\widetilde{u}_{j+1}\|_{H^1(D_1)^d}\preceq \|u_j\|_{H^1(D_0)^d}.
\]
Combining these inequalities we have
\[
\|u_{j+1}\|_{H^1(D)^d}\preceq \|u_j\|_{H^1(D_0)^d}.
\]
This concludes the proof.
\end{proof}

\begin{theorem}
There is a constant $C>0$ such that for every $\eta>C$, the expansion \eqref{eq:13} converges (absolutely) in $H^1(D)$. The asymptotic limit $u_0$ satisfies the problem \eqref{eq:18} and $u_0$ can be computed using formula \eqref{eq:23}.
\end{theorem}

\begin{proof}
From the Lemma \ref{lemmaconver2} applied repeatedly $j-1$ times, we get that for every $j\geq 2$ there is a constant $C$ such that
\begin{eqnarray*}
\|u_j\|_{H^1(D)^d} & \leq & C\|u_{j-1}\|_{H^1(D_0)^d}\leq C\|u_{j-1}\|_{H^1(D)^d}\\
& \leq & \cdots \leq C^{j-1}\|\widetilde{u}_1\|_{H^1(D_0)^d}
\end{eqnarray*}
and then
\[
\left\Vert \sum_{j=2}^{\infty}\eta^{-j}u_j\right\Vert_{H^1(D)^d}\leq \frac{\|\widetilde{u}_1\|_{H^1(D_0)^d}}{C}\sum_{j=2}^{\infty}\left(\frac{C}{\eta}\right)^{j}.
\]
The last expansion converges when $\eta>C$. Using \eqref{eq:4.35} and \eqref{eq:4.36} we conclude that there is a constant $C_1$ such that
\[
\left\Vert \sum_{j=0}^{\infty}\eta^{-j}u_j\right\Vert_{H^1(D)^d}\preceq C_1\left(\|f\|_{H^{-1}(D)^d}+\|g\|_{H^{1/2}(\partial D)^d}\right)\sum_{j=0}^{\infty}\left(\frac{C}{\eta}\right)^j.
\]
Moreover, the asymptotic limit $u_0$ satisfies \eqref{eq:18new}.
\end{proof}
Combining Lemma \ref{mainlemma} with the results analogous to lemmas \ref{lemmaconver1} and \ref{lemmaconver2} we get convergence for the expansion \eqref{eq:13} with the boundary condition \eqref{eq:14}.

\begin{corollary}
There are positive constants $C$ and $C_1$ such that for every $\eta >C$, we have
\[
\left\Vert u-\sum_{j=0}^{J}\eta^{-j}u_j\right\Vert_{H^1(D)^d}\leq C_1\left(\|f\|_{H^{-1}(D)^d}+\|g\|_{H^{1/2}(D)^d}\right)\sum_{j=J+1}^{\infty}\left(\frac{C}{\eta}\right)^j,
\]
for $J\geq 0$.
\end{corollary}

\section{The Case of Highly Elastic Inclusions}

In this section we derive and analyze expansions for the case of high-elastic inclusions. As before, we present the case of one single inclusion. The analysis is similar to the one in previous section and it is not presented here.

\subsection{Expansion Derivation: One High-Elastic Inclusion}\label{section4.1}

Let $E$ defined by
\begin{equation} \label{eq:38}
E(x)=\left\{\begin{array}{ll}
\epsilon, & x \in D_1,\\
1, & x \in D_0=D\setminus \overline{D}_1,
\end{array}\right.
\end{equation}
and denote by $u_\epsilon$ the solution of \eqref{eq:7}. We assume that $D_1$ is compactly included in $D$ ($\overline{D}_1\subset D$). Since $u_\epsilon$ of \eqref{eq:7} with the coefficient \eqref{eq:38} we have
\begin{equation}\label{eq:39}
{\cal A}_{D_0}(u_\epsilon,v)+\epsilon {\cal A}_{D_1}(u_\epsilon,v)=\int_Dfv,\quad \mbox{for all }v\in H_o^1(D)^d.
\end{equation}
We try to determine $\{ u_j\}_{j=-1}^\infty\subset H_0^1(D)^d$ such that
\begin{equation}\label{eq:40}
u_\epsilon =\epsilon^{-1}u_{-1}+u_0+\epsilon u_1+\epsilon^2u_2+\cdots=\sum_{j=-1}^{\infty}\epsilon^ju_j,
\end{equation}
and such that they satisfy the following Dirichlet boundary conditions
\begin{equation}
u_0=g \mbox{ on }\partial D \quad u_j=0 \mbox{ on }\partial D \mbox{ for } j=-1, \mbox{ and }j\geq 1.
\end{equation}
Observe that when $u_{-1}\neq0$, then, $u_\epsilon$ does not converge when $\epsilon \to 0$.
If we substitute \eqref{eq:40} in \eqref{eq:39} we obtain that for all $v\in H_0^1(D)^d$ we have
\[
\epsilon^{-1}{\cal A}_{D_0}(u_0,v)+\sum_{j=0}^{\infty}\epsilon^{j}({\cal A}_{D_0}(u_j;v)+{\cal A}_{D_1}(u_{j-1},v))=\int_Dfv.
\]
Now we equate powers of $\epsilon$ and analyze all the resulting subdomain equations.

\subsubsection{Term corresponding to  $\epsilon^{-1}$} 

We obtain the equation 
\begin{equation}\label{eq:42}
{\cal A}_{D_0}(u_{-1}, v)=0 \mbox{ for all } v\in H^1_0(D).
\end{equation}
Since we assumed $u_{-1}=0$ on $\partial D$, we conclude that $u_{-1}^{(0)}=0$ in $D_0$. 

\subsubsection{Term corresponding to $\epsilon^0=1$} 
We get the equation 
\begin{equation}\label{eq:43}
{\cal A}_{D_0}(u_{0}, v)+{\cal A}_{D_1}(u_{-1}, v)=\int_{D}fv 
\mbox{ for all } v\in H^1_0(D).
\end{equation}
Since $u_{-1}^{(0)}=0$ in $D_0$, we conclude that 
$u_{-1}^{(1)}$ satisfies the following Dirichlet problem in $D_1$, 
\begin{equation}\label{eq:44}
\begin{array}{c}\displaystyle
{\cal A}_{D_1} (u_{-1}^{(1)}, z) =\int_{D_1}fz \quad \mbox{ for all } z\in 
H^1_0(D_1)\\\\
u_{-1}^{(1)}=0 \mbox{ on } \partial D_1.
\end{array}
\end{equation}
Now we compute $u_{0}^{(0)}$ in $D_0$. As before, from \eqref{eq:43}, 
\[
\widetilde{\tau}( u_0^{(0)})\cdot n_0=-\widetilde{\tau}( u_{-1}^{(1)})\cdot n_1 \mbox{ on } 
\partial D_1.
\]
Then we can obtain $u_0^{(0)}$ in $D_0$ by solving 
the following problem
\begin{equation}\label{eq:45}
\begin{array}{c}
{\cal A}_{D_0}(u_0^{(0)}, z)=\int_{D_0}fz -
\int_{\partial D_1 }\widetilde{\tau}(u_{-1}^{(1)})\cdot n_1  z
\quad \mbox{ for all } z\in H^1(D_0)^d \mbox{ with }
z=0 \mbox{ on } \partial D,\\\\
u_{0}^{(0)}=g \mbox{ on } \partial D\subset \partial D_0.
\end{array}
\end{equation}

\subsubsection{Term corresponding to $\epsilon^j$ with $j\geq 1$}

We get the equation
\[
{\cal A}_{D_0}( u_{j},v)+{\cal A}_{D_1}(u_{j-1}, v)=0
\]
which implies that $u_j^{(1)}$ is harmonic in $D_1$ for all $j\geq 0$ and that 
$u_j^{(0)}$ is harmonic in $D_0$ for $j\geq 1$. Also, 
\[
\widetilde{\tau}(u_{j}^{(0)})\cdot n_0=-\widetilde{\tau}(u_{j-1}^{(1)})\cdot n_1.
\]
Given $u_{j-1}^{(0)}$ in $D_0$ (e.g., $u_{0}$ in $D_0$ above) 
we can find $u_{j-1}^{(1)}$ in $D_1$ by  solving the Dirichlet problem with the known Dirichlet data, 
\begin{equation}\label{eq:46}
\begin{array}{c}
{\cal A}_{D_1}(u_{j-1}^{(1)}, z)=0, \quad\mbox{ for all } 
z\in H^1_0(D_1)^d  \\
u_{j-1}^{(1)}=u_{j-1}^{(0)},\quad \mbox{on } \partial D_1.
\end{array}
\end{equation}
To find 
$u_{j}^{(0)}$ in $D_0$ we solve the problem
\begin{equation}\label{eq:47}
\begin{array}{c}\displaystyle
{\cal A}_{D_0}(u_{j}^{(0)}, z)=-\int_{\partial D_0 }\widetilde{\tau}( 
u_{j-1}^{(1)})\cdot n_1z, \quad \mbox{for all } 
z\in H^1(D_0)^d, \mbox{ with } z=0, \quad \mbox{on } \partial D,\\
u_{j}^{(0)}=0\mbox{ on } \partial D.
\end{array}
\end{equation}
For the convergence is similar to the case in the Section \ref{convergenceelasti}.



\chapter{Final Comments and Conclusions}\label{Chapter5}

We state the summary about the procedure to compute the terms of the asymptotic expansion for $u_\eta$ with high-conductivity inclusions. Other cases can be considered. For instance, an expansion for the case where we interchange $D_0$ and $D_1$ can also be analyzed. In this case the asymptotic solution is not constant in the high-conducting part. Other case is the study about domains that contain low-high conductivity inclusions. This is part of ongoing research.\\

We reviewed some results and examples concerning asymptotic expansions for high-contrast coefficient elliptic equations (pressure equation). In particular, we gave some explicit examples of the computations of the few terms in one dimension and several numerical examples in two dimensions. We mention that a main application in mind is to find ways to quickly compute the first few terms, in particular the term $u_0$, which, as seen in the manuscript, it is an approximation of order $\eta^{-1}$ to the solution.\\

We will also consider an additional aim that is to implement a code in \textsc{MatLab} to illustrate some examples, which help us to understand the behavior of low-contrast and low-high contrast coefficients in elliptic problems. In particular, we will mention explicit applications in two dimension and compute a few consecutive terms. In this manuscript presented numerical results the high-contrast case.\\

We recall that a main difficulty is that the computation of the harmonic characteristic functions is computationally expensive. We developed numerically the idea of trying to approximate these functions by solving a local problem in the background where the approximated harmonic characteristic function is zero on the a boundary of a neighborhood of the inclusions.\\

Applications of the expansion for the numerical solution of the elasticity problem will be consider in the future.


\appendix



\chapter{Sobolev Spaces}

\minitoc

In this appendix, we collect and present, mainly  without proofs, a structured review in the study of the Sobolev spaces, which viewed largely form the analytical development to its applications in numerical methods, in particular,  to the Finite Element Method.

\section{Domain Boundary and its Regularity}

Before presenting the definition of Sobolev spaces, we present some properties of subsets of $\R^d$, in particular, open sets. We begin by introducing the notion of a domain.

\begin{definition}{\bf (Domain)}
A subset $D\subset \R ^d$ is said to be domain if it is nonempty, open and connected.
\end{definition}
For more details see \cite{Solin-book2005}.\\

We assume an additional property about of the domain boundaries, specifically, the continuity of the boundaries known as Lipschitz-continuity, which we define it as follows.

\begin{definition}{\bf (Lipschitz-continuity)}\label{DefA.2}
 The boundary $\partial D$ is Lipschitz continuous if there exist a finite number of open sets ${\cal O}_i$ with $i=1,\dots ,m$, that cover $\partial D$, such that, for every $i$, the intersection $\partial D \cap {\cal O}_i$ is the graph of a Lipschitz continuous function and $D\cap {\cal O}_i$ lies on one side of this graph.
\end{definition}

\section{Distributions and Weak Derivatives}

In this section we consider the space $L^p(D)$, where the functions are equal if they coincide almost everywhere in $D$ with $D\subset \R^d$. The functions that belong to $L^p(D)$ are thus equivalence classes of measurable functions that satisfy the condition of the following definition.  

\begin{definition}{\bf (The space $L^p$)}
Let $D$ be a domain in $\R^d$ and let $p$ be a positive real number. We denote by $L^p(D)$ the class of all measurable functions $f$ defined on $D$ for which
\[
\|f\|_{L^p(D)}=\left(\int_{D}|f(x)|^p\right)^{1/p}<+\infty.
\]
\end{definition}

The following notation is useful for operations with partial derivatives.

\begin{definition}{\bf (Multi-index)}
Let $d$ be the spatial dimension. A multi-index is a vector $(\alpha_1,\alpha_2,\dots ,\alpha_d)$ consisting of $d$ nonnegative integers. We denote by $|\alpha|=\sum_{j=1}^{d}\alpha_j$ as the length of the multi-index $\alpha$. Let $u$ be an $m$-times continuously differentiable function. We denote the $\alpha th$ partial derivative of $u$ by
\[
D^{\alpha}u=\frac{\partial^{|\alpha |}u}{\partial x_1^{\alpha_1}\partial x_2^{\alpha_2}\cdots \partial x_d^{\alpha_d}}.
\]
\end{definition}

In $L^p$ spaces there are discontinuous and nonsmooth functions whose derivatives are not defined in the classical sense, see \cite{Solin-book2005}, then we have a generalized notion of derivatives of functions in $L^p$ spaces. These derivatives are known as weak derivatives, whose suitable idea depends of following definition.

\begin{definition}{\bf (Test functions)}\label{DefA.5} 
Let $D\subset \R^d$ be an open set. The space of test functions (infinitely smooth functions with compact support) is defined by
\[
C_0^{\infty}(D)=\left\{\varphi \in C^{\infty}:\supp (\varphi)\subset D; \supp (\varphi) \mbox{ is compact}\right\}.
\]
\end{definition}
Recall that the support is defined by
\[
\supp(\varphi )=\overline{\left\{x\in D: \varphi (x)\neq 0\right\}}.
\]
For more details see \cite{Adams-book2003,Kesavan-book1989,Solin-book2005}.\\

Additionally, we consider the following space to complete the definition of weak derivatives. 

\begin{definition}{\bf (Space of locally-integrable functions)} Let $D\subset \R^d$ be an open set and $1\leq p<\infty$. A function $u:D\to \R$ is said to be locally $p$-integrable in $D$ if $u\in L^p(K)$ for every compact subset $K\subset D$. The space of all locally $p$-integrable functions in $D$ is denoted by $L_{\loc}^p(D)$.
\end{definition}

With this last concept we define the weak derivatives. 

\begin{definition}{\bf (Weak derivatives)}
Let $D\subset \R^d$ be an open set, $f\in L_{\loc}^1(D)$ and let $\alpha$ be a multi-index. The function $D_{w}^{\alpha}f\in L_{\loc}^1(D)$ is said to be the weak $\alpha$th derivative of $f$ if
\[
\int_DD_{w}^{\alpha}f\varphi dx=(-1)^{|\alpha|}\int_DfD^{\alpha}\varphi dx, \mbox{ for all }\varphi \in C_0^{\infty}(D).
\]
\end{definition}

\section{Sobolev Spaces $H^k$}\label{SectionA.3}

The Sobolev spaces are subspaces of $L^p$ spaces where some control of the regularity of the derivatives, see \cite{Adams-book2003}. Moreover, the structure and properties of these spaces achieved a suitable purpose for the analysis of partial differential equations.\\

The Sobolev spaces are defined as follows.

\begin{definition}{\bf(Sobolev space $W^{k,p}$)}\label{DefA.8}
Let $D\subset \R^d$ be an open set, $k\geq 1$ an integer number and $p\in [1,\infty]$. We define
\[
W^{k,p}(D)=\left\{ f\in L^p(D): D_{w}^{\alpha}f \mbox{ exists and belongs to }L^p(D) \mbox{ for all multi-index }\alpha, |\alpha|\leq k \right\}.
\]
\end{definition}

The Sobolev spaces are equipped with the Sobolev norm defined next.

\begin{definition}{\bf(Sobolev norm)}
For every $1\leq p<\infty$ the norm $\| \cdot\|_{k,p}$ is defined by
\[
\|f\|_{k,p}=\left(\int_D \sum_{|\alpha|\leq k}|D_w^{\alpha}f|^p dx\right)^{1/p}=\left(\sum_{|\alpha|\leq k}\|D_w^{\alpha}f\|_p^p\right)^{1/p}.
\]
For $p=\infty$ we define
\[
\|f\|_{k,\infty}=\max_{|\alpha|\leq k}\|D_w^{\alpha}f\|_{\infty}.
\]
\end{definition}

The relevant case for this document is the case $p=2$. We denote $H^k(D)=W^{k,p}(D)$ and recall the definition of space of square-summable functions on $D$, which is defined by
\[
L^2(D)=\left\{f:D\to \R : \int_D|f|^2 dx<\infty\right\}.
\]
It is a Hilbert space with the scalar product
\[
(f,g)_{L^2(D)}=\int_Dfgdx,
\]
and norm given by
\[
\|f\|_{L^2(D)}^{2}=(f,f)_{L^2(D)}=\int_D|f|^2dx.
\]
We define the Sobolev space $H^k(D)$ for any integer $k\geq 1$. This Spaces are usually used in the development and analysis of the numerical methods for partial differential equations, in particular the Finite Element Methods applied to elliptic problems. We write the definition of $H^k$ next due to its relevance for this document.

\begin{definition}{\bf(Sobolev space $H^k$)}
 A function $f$ belongs to $H^k(D)$ if, for every multi-index $\alpha$, with $|\alpha|\leq k$, there exists $D^{\alpha}f\in L^2(D)$, such that
\[
\left\langle D_w^{\alpha}f,\varphi \right\rangle=\int_DD_w^{\alpha}f\varphi dx, \quad \mbox{for all }\varphi \in C_0^{\infty}(D).
\]
\end{definition}

It is important to say that the space $H^k(D)$ is a Hilbert space, we consider the following result.
\begin{theorem}
Let $D\subset \R^d$ be a open set, and let $k\geq 1$ be a integer. The Sobolev space $W^{k,2}(D)=H^k(D)$ equipped with the scalar product
\[
(f,g)_{H^k}=\int_D \sum_{|\alpha|\leq k}D_w^{\alpha}fD_w^{\alpha}g dx=\sum_{|\alpha|\leq k}(D_w^{\alpha}fD_w^{\alpha}g)_{L^2(D)},
\]
is a Hilbert space.
\end{theorem}

\begin{proof}
See \cite{Tartar-book2000}.
\end{proof}
The space $H^k(D)$ with the scalar product induce a norm $\| \cdot \|_{H^k(D)}$ which is given by
\[
\|f\|_{H^k(D)}^2=(f,f)_{H^k(D)}=\sum_{|\alpha|\leq k}\int_D|D_w^{\alpha}f|^2dx.
\]
And a seminorm $|\cdot |_{H^k(D)}$ that is given by
\[
|f|_{H^k(D)}^2=\sum_{|\alpha|=k}\int_D|D_w^{\alpha}f|^2dx.
\]
In the case $k=1$, we have
\[
\|f\|_{H^1(D)}^2=\|f\|_{L^2(D)}^2+|f|_{H^1(D)}^2=\|f\|_{L^2(D)}^2+\int_D|\nabla f|^2dx,
\]
where we denote the operator $\nabla$ by
\[
\nabla =\left(\frac{\partial}{\partial x_1},\dots,\frac{\partial}{\partial x_d}\right).
\]
Frequently it is common to considered subspaces of the space $H^k(D)$ with important properties or restrictions. We will consider, in particular,  the subspace $H_0^k(D)$ defined next.

\begin{definition}{\bf (Spaces $H_0^k$)}
Let $D\subset \R^d$ be a open set. We denote the closure of $C_0^{\infty}(D)$ in $H^k(D)$ by
\[
H_0^k(D)=\left\{g\in H^k(D):D^{\alpha}g=0 \mbox{ on }\partial D, \mbox{ for all }|\alpha|<k\right\}.
\]
Additionally, the space $H_0^k(D)$ is equipped with the norm of $H^k(D)$. In particular, $H_0^1(D)$ is equipped with the scalar product of $H^1(D)$.
\end{definition}

\section{The Spaces of Fractional Order $H^s$, with $s$ not an Integer}

In this section we consider the notion of the standard Sobolev spaces of fractional order, in particular for $s\geq 0$. This Spaces are used in important results, e.g., the Poincar\'e-Friedrich inequalities and Trace Theorem. Under special  conditions we get the following lemma.
\begin{definition}
Let $\sigma \in (0,1)$. Then, the norm
\[
\left(\|f\|_{L^2(D)}^2+|f|_{H^{\sigma}(D)}\right)^{1/2},
\]
with the seminorm
\[
|f|_{H^{\sigma}(D)}=\int_{D}\int_{D}\frac{|f(x)-f(y)|^2}{|x-y|^{2\sigma +n}}dxdy,
\]
provides a norm in $H^{\sigma}(D)$. Let $s>0$, with $[s]$ the integer part of $s$ and $\sigma=s-[s]$. Then, the norm in $H^s(D)$ is given by
\[
\left(\|f\|_{H^{[s]}(D)}+|f|_{H^{s}(D)}^2\right)^{1/2},
\]
with the seminorm
\[
|f|_{H^s(D)}^2=\sum_{|\alpha|=[s]}\int_D\int_D\frac{|D^{\alpha}f(x)-D^{\alpha}f(y)|^2}{|x-y|^{2\sigma +n}}.
\]
\end{definition}
For more details see \cite{Adams-book2003}.\\

We can also consider subspaces as before.
\begin{remark}
We define $H_0^s(D)$ as the closure of $C_0^{\infty}(D)$ in $H^s(D)$. We observe that $H_0^s(D)$ is a proper subspace of $H^s(D)$ if and only if $s>1/2$:
\[
\left\{ \begin{array}{cc}
H_0^s(D)=H^s(D), & s\leq 1/2,\\
H_0^s(D)\neq H^s(D), & s> 1/2.
\end{array}\right.
\]
\end{remark}
For details see the book \cite{Adams-book2003} and for a discussion of the case $s=1/2$ see \cite{mclean2000strongly}.

\begin{definition}
Let $s$ a non-negative real number, then the space $H^{-s}(D)$ is by definition a dual space of $H_0^s(D)$. Given a functional $u\in H^{-s}(D)$ and a function $v\in H_0^s(D)$, we consider the value of $u$ at $v$ as $\langle u,v \rangle $. The space $H^{-s}(D)$ is then equipped with the dual norm
\[
\|u\|_{H^{-s}(D)}=\sup_{v\in H_0^s(D)}\frac{\langle u,v \rangle}{\|v\|_{H^s(D)}}.
\]
\end{definition}
For more details and fundamental properties see \cite{Adams-book2003,MR1777711}.

\section{Trace Spaces}\label{Sec:A.5}

In general Sobolev space are defined using the spaces $L^p(D)$ and the weak derivatives. Hence, functions in Sobolev spaces are defined only almost everywhere in D. The boundary $\partial D$ usually has measure zero in $\R^d$, then, we see that the value of a Sobolev function is not necessarily well defined on $\partial D$. However, it is possible to define the trace of the Sobolev function on the boundary $\partial D$, its trace coincides with the boundary value. We consider the following definition.

\begin{definition}{\bf (Trace of the function $W^{k,p}$)}
For a function $f\in W^{k,p}(D)$ that is continuous to the boundary $\partial D$, its trace in the boundary $\partial D$ is defined by the function $\overline{f}$, such that
\[
\overline{f}(x)=f(x) \quad \mbox{for all }x\in \partial D.
\]
\end{definition}
Now we get the result about the traces of the functions $H^1$

\begin{theorem}{\bf (Traces of functions $H^1$)}\label{TheoremA.2}
Let $D\subset \R^d$ be a bounded domain with Lipschitz-continuous boundary. Then there exists a continuous linear operator ${\cal T}:H^1\to H^{1/2}(\partial D)$ such that
\begin{enumerate}
\item [(i)] $({\cal T}f)(x)=f(x)$ for all $x\in \partial D$ if $f\in H^1\cap C(\overline{D})$.

\item [(ii)] There exists a constant $C>0$ such that
\[
\|{\cal T}f\|_{H^{1/2}(\partial D)}\leq C\|f\|_{H^1(D)},
\]
for all $f\in H^1(D)$.
\end{enumerate}
\end{theorem}

\begin{proof}
See \cite{Adams-book2003,lions1972non}.
\end{proof}

\section{Poincar\'e and Friedrichs Type Inequalities}

The Poincar\'e and Friedrichs inequalities are used in our case for the analysis of convergence of the asymptotic  expansions, in particular to bound the terms $\{u_j\}_{j=0}^{\infty}$ and its approximations. We consider the following general result, see \cite{toselli2005domain}.

\begin{theorem} \label{TheoremA.3}
Let $D\subset \R^d$ be a bounded Lipschitz domain and let $f_j$, $j=1,\dots, L$, $L\geq 1$, be functionals (not necessarily linear) in $H^1(D)$, such that, if $v$ is constant in $D$
\[
\sum_{j=1}^L|f_j(v)|^2=0, \quad \mbox{if and only if }v=0.
\]
Then, there exist constants, depending only on $D$ and the functionals $f_j$, such that, for $u\in H^1(D)$
\[
\|u\|_{L^2(D)}^{2}\leq C_1|u|_{H^1(D)}^2+C_2\sum_{j=1}^L|f_j(u)|^2.
\]
\end{theorem}

\begin{proof}
The proof of the theorem is an application of the Rellich's Theorem, see e.g., \cite{Neecas-book2012}.
\end{proof}

With these important result we get the following two lemmas. 

\begin{lemma}\label{LemmaA.3}{\bf (Poincar\'e inequality)}
Let $u\in H^1(D)$. Then there exist constants, depending only on $D$, such that 
\[
\|u\|_{L^2(D)}^{2}\leq C_1|u|_{H^1(D)}^2+C_2\left(\int_Dudx\right)^2.
\]
\end{lemma}

\begin{proof}
Apply the theorem \ref{TheoremA.3} with $L=1$ and the linear functional
\[
f_1(u)=\int_Dudx.
\]
\end{proof}
Note that if $\int_Du=0$ then we can bound the full $H^1$ norm by the semi-norm.

\begin{lemma}\label{LemmaA.4}{\bf (Friedrichs inequality)}\
Let $\Gamma \subseteq \partial D$ have nonvanishing $(n-1)$-dimensional measure. Then, there exist constants, depending only on $D$ and $\Gamma$, such that, for $u\in H^1(D)$
\[
\|u\|_{L^2(D)}^{2}\leq C_1|u|_{H^1(D)}^2+C_2\|u\|_{L^2(\Gamma)}^2.
\]
In particular, if $u$ vanishes on $\Gamma$
\[
\|u\|_{L^2(D)}^2\leq C_1|u|_{H^1(D)}^2,
\]
and thus
\[
|u|_{H^1(D)}^2\leq \|u\|_{H^1(D)}^2\leq (C_1+1)|u|_{H^1(D)}^2.
\]
\end{lemma}

\begin{proof}
Apply the theorem \ref{TheoremA.3} with $L=1$ and $f_1(u)=\|u\|_{L^2(\Gamma)}$.
\end{proof}

For more details about these inequalities, see \cite{Adams-book2003,Brezis-book2010,Neecas-book2012}.



\chapter{Elliptic Problems}

\minitoc
In this appendix we are dedicated to discuss and recall some theoretical results about the second-order linear elliptic problems.\\

First we assume a bounded open set $D\subset \R^d$ with Lipschitz-continuous boundary (defined in \ref{DefA.2}), and introduce for $u\in C^2(D)$ the second-order linear differential equation
\[
-\sum_{i,j=1}^d\frac{\partial}{\partial x_j}\left(\kappa_{ij}(x)\frac{\partial u}{\partial x_i}\right)+\sum_{i=1}^db_i(x)\frac{\partial u}{\partial x_i}+c(x)u=f(x),
\]
where the coefficient and the function $f$ satisfy the assumptions of $\kappa_{ij},b_i\in C^1(D)$, $c\in C(D)$ and $f\in C(D)$.

\section{Strong Form for Elliptic Equations}

The purpose of the text is in particular, to find solubility of (uniformly) elliptic second-order partial differential equations with boundary conditions.\\

Let $D\subset \R^d$ be a domain and $u:\overline{D}\to \R $, is the unknown variable $u(x)$, with $x\in \R^d$, such that
\[
\left\{ \begin{array}{ll}
Lu(x)=f(x),& \mbox{ in }D\\
\hspace{0.1in}u(x)=0, & \mbox{ on }\partial D. 
\end{array}\right.
\]
Where $f:D\to \R$ is a given function and $L$ is the second-order partial differential operator defined by
\[
Lu=-\sum_{i,j=1}^d\frac{\partial}{\partial x_j}\left(\kappa_{ij}(x)\frac{\partial u(x)}{\partial x_i}\right)+\sum_{i=1}^db_j(x)\frac{\partial u(x)}{\partial x_i}+c(x)u(x),
\]
or
\[
Lu=-\dive(\kappa(x)\nabla u)+\sum_{i=1}^db_j(x)\frac{\partial u(x)}{\partial x_i}+c(x)u(x),
\]
where the operator $\dive$ is defined by
\[
\dive u(x)=\nabla \cdot u(x)=\sum_{i=1}^d\frac{\partial}{\partial x_i}u(x)
\]
and the $d\times d$ symmetric matrix $\kappa(x)$ is defined by
\[
\kappa(x)=\left[\begin{array}{cccc}
\kappa_{11}(x) & \kappa_{12}(x) & \dots & \kappa_{1d}(x)\\
\kappa_{21}(x) & \kappa_{22}(x) & \dots & \kappa_{2d}(x)\\
\vdots & \vdots & \vdots & \vdots \\
\kappa_{d1}(x) & \kappa_{d2}(x) & \dots & \kappa_{dd}(x)\\
\end{array}\right], \mbox{ with }\kappa_{ij}=\kappa_{ji}, \mbox{ for }i\neq j \mbox{ and }i,j=1,\dots,d.
\]
For more details about of elliptic operators, see \cite{Evans-book1990,Grisvard-book1985,mclean2000strongly,Neecas-book2012}.\\

Now, for the development of the text we assume a special condition about the partial differential operator, in this case we consider the uniformly elliptic conditions, which is defined by.
 
\begin{definition}{\bf (Uniform ellipticity)}
We say the partial differential operator $L$ is uniformly  elliptic if there exists a constant $K>0$ such that 
\begin{equation}
\sum_{i,j=1}^d\kappa_{ij}(x)\xi_i \xi_j\geq K|\xi|^2,
\end{equation}
for almost everywhere $x\in D$ and all $\xi \in \R^d$.
\end{definition}

The ellipticity means that for each $x\in D$ the symmetric $d \times d$ matrix $\kappa(x)=\left[ \kappa_{ij}(x)\right]$ is positive definite,  with smallest eigenvalue greater than or equal to $K$, see \cite{Evans-book1990}. We assume that there exist $\kappa_{\min}$ and $\kappa_{\max}$ for each eigenvalue $\rho_{1}(x),\dots,\rho_{d}(x)$ of the matrix $\kappa(x)$ such that
\[
\kappa_{\min}\leq \rho_{\min}(x) \leq \cdots \leq \rho_{max}(x)\leq \kappa_{max},\quad \mbox{for all }x\in D,
\]
where $\rho_{\min}(x)$ and $\rho_{max}(x)$ are smallest and greatest eigenvalues of the matrix $\kappa(x)$.\\

In particular, we consider $D\subset \R^2$ (or $\R^3$) and the problem in its strong form homogeneous
\[
\left\{\begin{array}{ll}
Lu(x)=f(x), & \mbox{for all }x\in D\\
\hspace{0.1in}u(x)=0, & \mbox{for all }x \in \partial D.
\end{array}\right.
\]
We assume the operator $L$ as
\begin{equation} \label{eq:B.1}
Lu(x)=-\dive (\kappa(x)\nabla u(x)),
\end{equation}
with coefficient functions $b_i(x)=0$ and $c(x)=0$, then we get the strong form of the problem
\begin{equation}\label{eq:B.2}
\left\{\begin{array}{ll}
-\dive \left(\kappa(x)\nabla u(x)\right)=f(x), & \mbox{ for all }x\in D,\\
\hspace{0.5in}u(x)=0, & \mbox{ for all }x\in \partial D.
\end{array}\right.
\end{equation}
If the equation \eqref{eq:B.2} has solution, we say that the problem has a solution in the classic sense. However, only specific cases of elliptic differential equations have classic solutions, so we introduce in the next section a alternative method for the elliptic problems.

\section{Weak Formulation for Elliptic Problems}\label{Weakforsection}

In this section we introduce solutions in the weak sense for elliptic partial differential equations. This method requires building a weak (or variational) formulation of the differential equation, which can be put in a suitable function space framework, i.e., a Sobolev space $W^{k,p}$ given in the definition \ref{DefA.8}, particularly is a Hilbert space if $p=2$.\\

We consider a domain $D\subset \R^d$ and the strong form \eqref{eq:B.2} with the respective boundary conditions. To obtain the weak formulation of the problem \eqref{eq:B.2} first we multiply the equation by a test function $v\in C_0^{\infty}(D)$ and integrate over $D$, we have
\begin{equation}\label{eq:B.3}
-\int_D \dive \left(\kappa(x)\nabla u(x)\right)v(x)dx=\int_Df(x)v(x)dx,\mbox{ for all }v\in C_0^{\infty}(D).
\end{equation}
We assume that $v=0$ on $\partial D$.\\

From the equation \eqref{eq:B.3} and using the integration by parts we can write
\begin{equation}\label{eq:B.4}
{\cal A}(u,v)={\cal F}(v), \mbox{ for all } v\in H_0^1(D),
\end{equation}
where we have introduced the notation
\begin{align}
{\cal A}(u,v)&=\int_D 
\kappa(x)\nabla u(x)\cdot  \nabla v(x)dx,  
 &&\mbox{ for all }u\in H^1(D) \mbox{ and }v\in H_0^1(D),\nonumber 
\end{align}
and
\begin{align}
{\cal F}(v)&=\int_Df(x)v(x)dx, &&\quad  \mbox{for all }v\in H_0^1(D).\nonumber
\end{align}
For more details see for instance \cite{atkinson2009theoretical,Evans-book1990}.

\section{Existence of Weak Solutions}\label{SectionB.3}

We recall results about the existence of weak solutions for elliptic problems. We recall definitions above for ${\cal A}(u,v)$, ${\cal F}(v)$ and we consider $H$ a Hilbert space with the norm $\|\cdot\|_{H}$, scalar product $(\cdot,\cdot)$. We deduce following general result.

\begin{theorem}{\label{Lax-Milgram}\bf (Lax-Milgram)}
We assume that
\[
{\cal A}:H\times H\to \R,
\]
is a bilinear application, for which there exist constants $\alpha,\beta >0$ such that
\begin{enumerate}
\item [(i)] 
\[
|{\cal A}(u,v)|\leq \alpha \|u\|_H\|v\|_H, \quad\mbox{(continuity of }{\cal A}\mbox{)}
\]
and
\item[(ii)]
\[
\beta \|u\|_H^2\leq {\cal A}(u,u), \mbox{ for all }u\in H, \quad\mbox{ (ellipticity of }{\cal A}\mbox{)}.
\]
\end{enumerate}
Finally, let $f:H\to \R$ be a bounded linear functional in $H$. Then, there exists an unique $u\in H$ such that
\[
{\cal A}(u,v)={\cal F}(v), \mbox{ for all }v\in H.
\] 
\end{theorem}

\begin{proof}
See \cite{Brezis-book2010,Evans-book1990}.
\end{proof}

For the case \eqref{eq:B.4} and $D\subset \R^2$ we assume that there exist $\kappa_{\min}$ and $\kappa_{\max}$ such that
\[
0<\kappa_{\min}\leq \rho_{\min}(x)  \leq \rho_{\max}(x)\leq \kappa_{\max}, \quad\mbox{ for all }x\in D,
\]
where $\rho_{\min}(x)$ and $\rho_{\max}(x)$ are smallest and largest eigenvalue of the matrix $\kappa(x)$, then we have
\begin{eqnarray*}
|{\cal A}(u,v)| & = & \left| \int_D\kappa(x)\nabla u \cdot \nabla v\right|\\
& \leq & \left(\int_D\kappa(x)|\nabla u|^2\right)^{1/2}\left(\int_D\kappa(x)|\nabla v|^2\right)^{1/2}\\
& \leq & \kappa_{\max} \left(\int_D|\nabla u|^2\right)^{1/2}\left(\int_D|\nabla v|^2\right)^{1/2}\\
& \leq & \kappa_{\max} \left[\left(\int_D|\nabla u|^2\right)^{1/2}+\left(\int_D|u|^2\right)^{1/2}\right]\left[\left(\int_D|\nabla v|^2\right)^{1/2}+\left(\int_D|v|^2\right)^{1/2}\right]\\
& = & \kappa_{\max} \|u\|_{H^1(D)}\|v\|_{H^1(D)},
\end{eqnarray*}
where $\alpha=\kappa_{\max}$.\\

Now, we have
\[
\beta\|u\|_H^2\leq {\cal A}(u,u),
\]
by the Poincar\'e inequality in the Lemma \ref{LemmaA.3}, we get
\[
\int_D|u|^2\leq C\int_D|\nabla u|^2,
\]
with $C$ is a Poincar\'e inequality constant of $D$, by adding the term $\int_D|\nabla u|^2$ in both sides we have
\begin{eqnarray*}
\int_D|u|^2+\int_D|\nabla u|^2 & \leq & C\int_D|\nabla u|^2+ \int_D|\nabla u|^2\\
& = &(C+1)\kappa_{\min}^{-1}\kappa_{\min}\int_D|\nabla u|^2\\
& = & \kappa_{\min}^{-1}(C+1)\int_D\kappa_{\min}|\nabla u|^2\\
& \leq & \kappa_{\min}^{-1}(C+1)\int_D\kappa(x)|\nabla u|^2,
\end{eqnarray*}
therefore
\[
\beta\|u\|_{H^1(D)}^2\leq{\cal A}(u,u),
\]
where $\beta=\kappa_{min}(C+1)^{-1}$.

\section{About Boundary Conditions}

Above we consider boundary data $u=0$ on $\partial D$. In this case, we say that the problem has a homogeneous Dirichlet data. We present other situations in this section.

\subsection{Non-Homogeneous Dirichlet Boundary Conditions}

We consider the problem in its strong form 
\begin{equation}\label{strongf Dirichlet}
\left\{\begin{array}{ll}
Lu(x)=f(x), & \mbox{ in }D\\
\hspace{0.1in}u(x)=g(x), & \mbox{ on }\partial D,
\end{array}\right.
\end{equation}
where $g(x)\in C(\partial D)$, it means that $u(x)=g(x)$ on $\partial D$ in the trace sense, and the operator $L$ is defined in the equation \eqref{eq:B.1}. This may be possible, if $g(x)$ is the trace of the any function $H^1$, we say $w(x)$. We get that $\tilde{u}(x)=u(x)-w(x)$ belongs $H_0^1(D)$ and is a solution in the weak sense of the elliptic problem
\[
\left\{\begin{array}{ll}
L\tilde{u}(x)=\tilde{f}(x), & \mbox{ in }D\\
\hspace{0.1in}u(x)=0, & \mbox{ on }\partial D,
\end{array}\right.
\]
where $\tilde{f}(x)=f(x)-Lw(x)\in H^{-1}(D)$. The last problem is interpreted in the weak sense.

\subsection{Neumann Boundary Conditions}

We consider the problem in its strong form
\begin{equation}\label{Strongf Neumann}
\left\{\begin{array}{ll}
\hspace{0.3in}Lu(x)=f(x), & \mbox{ in }D\\
\nabla u(x)\cdot n=g(x), & \mbox{ on }\partial D,
\end{array}\right.
\end{equation}
where $g(x)\in C(\partial D)$ and $n$ is the unit outer normal vector  to $\partial D$. The weak formulation of the problem \eqref{Strongf Neumann} is derived next. Multiply \eqref{Strongf Neumann} by the test function $v\in C^{\infty}(D)\cap C^1(D)$, integrate over $D$ and using the Green's theorem we get
\[
\int_D\kappa(x)\nabla u(x) \cdot \nabla v(x)dx-\int_{\partial D}(\kappa(x)g(x))v(x)dS=\int_Df(x)v(x)dx,
\]
in linear forms, we have to find a $v\in H^1(D)$ such that 
\[
{\cal A}(u,v)={\cal F}(v),\quad\mbox{for all }v\in H^1(D),
\]
where
\begin{align}
{\cal A}(u,v)&=\int_D 
\kappa(x)\nabla u(x)\cdot  \nabla v(x)dx,  
 &&\mbox{ for all }u\in H^1(D) \mbox{ and }v\in H^1(D),\\ \nonumber
{\cal F}(v)&=\int_Df(x)v(x)dx+\int_{\partial D}(\kappa(x)g(x))v(x)dS, &&\quad  \mbox{for all }v\in H_0^1(D). 
\end{align}
Note that the bilinear form ${\cal A}(\cdot, \cdot)$ is given by the same equation as in the case of Dirichlet boundary conditions, it is different since the $H^1$ space is changed.
For more details in particular see \cite{Solin-book2005,toselli2005domain}.

\section{Linear Elasticity}\label{LinearElasticity}

In this section we recall the case described in the Chapter \ref{Chapter4} for the linear elasticity problem. We state some definitions for the model \eqref{sformelast} given in the Section \ref{sec:problem setting}.\\

We consider the equilibrium equations for a linear elastic material describes in any domain of $\R^d$. Let $D\subset \R^d$ polygonal domain or a domain with smooth boundary. Given $u\in H^1(D)^d$ we denote
\[
\epsilon =\epsilon (u)= \left[ \epsilon_{ij} =\frac{1}{2} \left(\frac{\partial u_i}{\partial x_j}+ \frac{\partial u_j}{\partial x_i} \right) \right],
\]
where $\epsilon$ is a strain tensor that  linearly depends on the derivatives of the displacement field
and we recall that $\tau(u)$ is stress tensor, which depends of the value of strains and is defined by
\[
\tau =\tau(u)=2\mu\epsilon(u)+\lambda\dive (u)I_{d\times d},
\]
where $I_{d\times d}$ is the identity matrix in $\R^d$ and $u$ is the displacement vector. the Lam\'e coefficients $\lambda=\lambda(x)$ are represent fundamental elastic modulus of isotropic bodies, often used to replace the two conventional modulus $E$ and $\nu$ in engineering. See \cite{kang1996mathematical}. The function $\mu=\mu(x)$ describe the material.\\

We assume that the Poisson ratio $\nu=0.5\lambda/(\lambda+\mu)$ is bounded away from $0.5$, i.e Poisson's ratio satisfies $0<\nu<0.5$. It is easy show that $\nu<0.5$. We recall the volumetric strain modulus is given by
\[
K=\frac{E}{3(1-2\nu)}>0,
\]
and equivalently we have that $1-2\nu>0$, then $\nu<0.5$. See \cite{kang1996mathematical}. We also assume that $\nu=\nu(x)$ has mild variation in $D$.\\

We introduce the heterogeneous function $E=E(x)$ that represent the Young's modulus and we get
\[
\mu(x)=\frac{1}{2}\frac{E(x)}{1+\nu(x)}=\tilde{\mu}(x)E(x),
\]
and
\[
\lambda(x)=\frac{1}{2}\frac{E(x)\nu(x)}{(1+\nu(x))(1-2\nu(x))}=\tilde{\lambda}(x)E(x).
\]
Here $E(x)$ can be calculated by dividing the stress tensor by the strain tensor in the elastic linear portion $\tau /\epsilon$. We introduced $\tilde{\lambda}=\nu/2(1+\nu)(1-2\nu)$ and $\tilde{\mu}=1/2(1+\nu)$. We also denote
\[
\tilde{\tau}(u)=2\tilde{\mu}\epsilon(u)+\tilde{\lambda}\dive (u)I_{d\times d}.
\]
For the general theory of mathematical elasticity, and the particular functions $\tau,\mu,E,\epsilon$ and $\lambda$ see e.g., \cite{MR936420,kang1996mathematical,MR0075755}.

\subsection{Weak Formulation}

Given a vector field $f$, we consider the Dirichlet problem
\begin{equation} \label{eq:1}
-\dive (\tau(u))=f, \quad \mbox{in }D, 
\end{equation} 
with $u=g$ on $\partial D$. A weak formulation of the equation \eqref{eq:1} goes as follows. First we multiply the equation \eqref{eq:1} by a test (vector) function $v\in H_0^1(D)^d$, we integrate over the domain $D$ afterwards so we get
\begin{equation}\label{eq:2}
\int_{D}-\dive(\tau(u))v=\int_{D}fv, \quad \mbox{for all }v \in H_0^1(D)^d.
\end{equation}
Using Green's formula, we get
\begin{equation}\label{eq:3}
\int_{D}\tau(u)\cdot \nabla v=\int_{D}fv+\int_{\partial D}\tau (u)\cdot nv, \quad \mbox{for all }v \in H_0^1(D)^d,
\end{equation}
or
\begin{equation}\label{eq:4}
\int_{D}(2\mu\epsilon(u)+\lambda \dive (u)I_{d\times d})\cdot \nabla v=\int_{D}fv+\int_{\partial D}\tau (u)\cdot nv, \quad \mbox{for all }v \in H_0^1(D)^d.
\end{equation}
Using that $\int_{\partial D}\tau (u)\cdot nv=0$, the equation \eqref{eq:4} is reduced 
\begin{equation}\label{eq:5}
\int_{D}(2\mu\epsilon(u)+\lambda \dive (u)I_{d\times d})\cdot \nabla v=\int_{D}fv, \quad \mbox{for all }v \in H_0^1(D)^d.
\end{equation}
Now, we rewrite the equation \eqref{eq:5} as
\begin{equation}\label{eq:6}
\int_{D}(2\mu\epsilon(u)\cdot \epsilon(v)+\lambda\dive (u)\dive (v)=\int_{D}fv, \quad \mbox{for all }v \in H_0^1(D)^d,
\end{equation}
where $\epsilon(u)\cdot \epsilon(v)=\sum_{i,j=1}^{d}\epsilon_{ij}(u)\epsilon_{ji}(v)$.  The equation \eqref{eq:6} is the weak formulation for the problem \eqref{eq:1}.

\subsection{Korn Inequalities}\label{Sec:B.5.2}

We have the analogous inequalities to Poincar\'e and Friedrichs inequalities. First, we introduce the quotient space $H=H^1(D)^d/{\cal RB}$, see Section \ref{sec:problem setting}, which is defined as a space of equivalence classes, i.e., two vectors in $H^1(D)^d$ are equivalent if they differ by a rigid body motion. We get the following result, known as Korn inequalities for the strain tensor.

\begin{lemma}
Let $D$ be a bounded Lipschitz domain. Then
\begin{equation}\label{eq:15}
|u|_{H^1(D)^d}^2\leq 2\int_D|\epsilon(u)|^2=2\sum_{i,j=1}^{d}\int_D|\epsilon_{ij}(u)|^2, \mbox{ for all }u\in H_0^1(D)^d.
\end{equation}
There exists a constant $C$, depending only on $D$, such that
\begin{equation}\label{eq:16}
\|u\|_{H^1(D)^d}^2\leq C\int_D|\epsilon(u)|^2,\mbox{ for all  }u\in H.
\end{equation}
\end{lemma}

If we assume that $u=0$ in the problem \eqref{eq:1} we get the important results in relation to the Dirichlet problems.

\begin{theorem}{\bf (Dirichlet problem)}\label{dirichletprob}
Let $f\in H^{-1}(D)^d$. Then, there exists a unique $u\in H_0^1(D)^d$, satisfying \eqref{eq:6} and constants, such that
\[
\|u\|_{H^1(D)^d}\leq C_1\|f\|_{H^{-1}(D)^d}+\|g\|_{H^{1/2}(\partial D)^d},\quad \|u\|_{\cal A}\leq C_2 \|f\|_{H^{-1}(D)^d}.
\]
\end{theorem}
Remember that the bilinear form ${\cal A}(\cdot,\cdot)$ provides a scalar product in $H$, see \cite{toselli2005domain}, we denote the norm and $\|u\|_{\cal A}$ and have the equivalence
\[
\beta \|u\|_{H}^2\leq\|u\|_{\cal A}^{2}\leq \alpha\|u\|_{H}^{2}, \quad\mbox{ for all }u\in H,
\]
where $H$ is a Hilbert space.\\
 
Analogously, for the Neumann problem
\begin{equation}\label{eq:17}
-\dive (\tau(u))=f \ \mbox{in }D,
\end{equation}
with $\tau(u)\cdot n=g$ on $\partial D$. If a solution $u$ of problem \eqref{eq:17} exists, then it is defined only up to a rigid body mode in ${\cal RB}$. Moreover, the following compatibility condition must hold
\begin{equation}\label{eq:18}
\int_Df\cdot r +\int_{\partial D}g\cdot r=0, \mbox{ for all } r\in {\cal RB}.
\end{equation}
We have the bilinear form
\[
(u,v)_H=\int_D\epsilon(u)\cdot \epsilon(v), 
\]
defines a scalar product in $H$. We use the equation \label{eq:16}, the corresponding induced norm $\|u\|_H$ is equivalent to
\[
\inf_{r\in {\cal RB}}\|u-r\|_{H^1(D)^d}.
\] 
We assume that $f$ belongs to the dual space of $H^1(D)^d$ and $g$ to $H^{-1/2}(\partial D)^d$. We consider these assumptions and the condition  \eqref{eq:18}, the expression
\[
\langle f,u\rangle=\int_Df\cdot u +\int_{\partial D}g\cdot u, \mbox{ for all } u\in H,
\]
defines a linear functional of $H'$. Using the same  equation \eqref{eq:6} we have the next result.
 
\begin{theorem}{\bf (Neumann problem)}\label{neumannpro}
Assume that $f$ and $g$ satisfy the regularity assumptions given above and that \eqref{eq:18} holds. Then, there exists a unique $u\in H$, satisfying \eqref{eq:6} and constants, such that
\[
\|u\|_{H}\leq C_1\|F\|_{H'},\quad \|u\|_{\cal A}\leq C_2 \|F\|_{H'},
\]
with 
\[
\|F\|_{H'}\leq C\left(\|f\|_{H^{-1}(D)^d}+\|g\|_{H^{-1/2}(\partial D)^d}\right).
\]
\end{theorem}
Note that the solution of a Neumann problem is defined only up to a rigid body motion. That is if $u$ is solution of a Neumann problem, then $u+\zeta$ also is solution for any $\zeta \in {\cal RB}$.\\

For the case non-homogeneous and mixed boundary conditions can also be considered, see e.g. \cite{kang1996mathematical,slaughter2002linearized,MR0075755}.



\chapter{Finite Element Methods}\label{AppendixC}

\minitoc

In this appendix we introduce an important numerical method in the development of partial differential equations, specifically,  the symmetric elliptic problems.  We review the finite element approximations which is mentioned in classical general references  \cite{ciarlet1978finite,MR2477579,johnson2012numerical,Solin-book2005,toselli2005domain} and specific references with examples in $\R^1$ and $\R^2$ \cite{Galvis-book2009,Galvis-book2011}.\\

The finite element method provides a formalism for discrete (finite) approximations of the solutions of differential equations, in general boundary data problems.

\section{Galerkin Formulation}\label{C.1section}

We consider the elliptic problem in strong form
\begin{equation} \label{StrongformApC}
-\dive(\kappa(x)\nabla u(x))=f(x),\ \mbox{  in }  D. 
\end{equation}
We deduce the weak formulation of the problem \eqref{StrongformApC}, which consists in: find $u\in V=H_{0}^{1}(D)$ such that 
\begin{equation}
{\cal A}(u,v)={\cal F}(v),\mbox{ for all }v\in V,\label{Continuousproblem}
\end{equation}
where the space $V$, the bilinear form ${\cal A}(u,v)$
and the linear form ${\cal F}(v)$ satisfy the assumptions of the Lax-Milgram theorem \ref{Lax-Milgram}. Then, the Galerkin method is used to approximate the solution of an problem in its weak form. In this way, we consider the finite dimensional subspace $V_{h}\subset V$, and associate the discrete problem, which consists in finding $u_{h}\in V_{h}$ such that
\begin{equation}\label{gral Galerkin}
{\cal A}(u_{h},v_{h})={\cal F}(v_h),\mbox{ for all }v_{h}\in V_{h}.
\end{equation}
Applying the Lax-Milgram theorem \ref{Lax-Milgram}, it follows that the discrete problem \eqref{gral Galerkin} has an unique solution $u_{h}$ which is called discret solution.\\

In this order, to apply the Galerkin approximation we need to construct finite dimensional subspaces of the spaces $H_{0}^{1}(D), H^{1}(D)$. The finite element method, in its simplest form, is a specific process of constructing subspaces $V_{h}$, which are called finite element spaces. The building is characterized in detail in \cite{ciarlet1978finite}. We have the following result.

\begin{lemma}
\label{Lemma uniqueness}{\bf (Solubility unique)}
The problem \eqref{gral Galerkin} has an unique solution $u_{h}\in V_{h}$.
\end{lemma}

\begin{proof}
The form ${\cal A}(\cdot,\cdot)$, restricted to $V_{h}\times V_{h}$ results bilinear, bounded and  $V_{h}$-elliptic. The form ${\cal F}(v_{h})$ is linear and belongs $V'_h$ (which is dual space of $V_h$). Then the assumptions of the Lax-Milgram theorem \ref{Lax-Milgram} are satisfied, thus there exists an unique solution.
\end{proof}

The solution $u_{h}\in V_{h}$ to the discrete problem \eqref{gral Galerkin} can be found numerically, as the subspace $V_{h}$ contains a finite basis $\left\{ \phi_{i}\right\} _{j=1}^{N_{h}}$, where $N_{h}$ is the dimension of $V_{h}$. So the solution $u_{h}$ can be written as a linear combination  of the basis functions with unknown coefficients
\begin{equation}
u_{h}=\sum_{j=1}^{N_{h}}x_{j}\phi_{j}.\label{BasisComb}
\end{equation}
By substituting  \eqref{BasisComb} in problem \eqref{gral Galerkin} we find that
\[
{\cal A}\left(\sum_{j=1}^{N_{h}}x_{j}\phi_{j},v_{h}\right)=\sum_{j=1}^{N_h}x_j{\cal A}(\phi_j,v_h)={\cal F}(v_h),\, \mbox{ for all }v_h\in V_h.
\]
By the linearity of ${\cal A}$ and ${\cal F}$ we need to verify equation \eqref{gral Galerkin} only for the test functions $v_h=\phi_i$, for $i=1,\dots, N_h$, in the basis \eqref{BasisComb}. This is equivalent to the formulation: find $u_h=\sum_{j=1}^{N_h}x_j\phi_j$ such that
\begin{equation}\label{AlgebraicSystem}
\sum_{i=1}^{N_h}x_i{\cal A}(\phi_i,\phi_j)={\cal F}(\phi_j),\quad j=1,\dots,N_h.
\end{equation}
Now, we define the $N_h\times N_h$ stiffness matrix 
\[
\mathbf{A}= [a_{ij}],\,\quad a_{ij}={\cal A}(v_{i},v_{j}),\, i,j=1,\dots,N_h, 
\]
the vector
\[
\mathbf{b}=[ b_{1}\dots ,b_{N_h}]^T\in \R^{N_h},\quad b_{j}={\cal F}(v_{j}),\, j=1,\dots,N_h,
\]
and unknown coefficient vector
\[
\mathbf{u}_{h}=[ x_{1},\dots ,x_{N_h}]^T\in \R^{N_h}.
\]
Then algebraic system \eqref{AlgebraicSystem} can be written  of matrix form, which consists of finding $\mathbf{u}_{h}\in\R^{N_{h}}$ such that
\begin{equation}\label{matrixproblem}
\mathbf{A}\mathbf{u}_{h}=\mathbf{b}.
\end{equation}
We consider a specific condition for the matrix $\mathbf{A}$ that is the invertibility. First we introduce positive definiteness of the $\mathbf{A}$ in the following lemma.

\begin{lemma}\label{lemmadefpositive}{\bf (Positive definiteness)} 
Let $V_{h}$, $\dim(V_{h})=N_{h}<\infty$ be a Hilbert space and ${\cal A}(\cdot,\cdot):V_h\times V_h\to\R$ a bilinear $V_h$-elliptic form . Then the stiffness matrix $\mathbf{A}$ of the  problem \eqref{matrixproblem} is positive definite.
\end{lemma}

\begin{proof}
We show that $\mathbf{u}_{h}^{T}\mathbf{A}\mathbf{u}_{h}>0$ for all $\mathbf{u}_{h}\neq 0$. Then we take an arbitrary $\hat{\mathbf{u}}_{h}=[\hat{x}_1,\dots,\hat{x}_{N_h}]$ and define the vector
\[
\hat{v}=\sum_{j=1}^{N_h}\hat{x}_jv_j, 
\]
where $\left\{v_1,\dots,v_{N_h}\right\}$ is some basis in $V_h$. By the $V_h$-ellipticity of the form ${\cal A}(\cdot,\cdot)$ it is
\begin{eqnarray*}
\hat{\mathbf{u}}_{h}^{T}\mathbf{A}\hat{\mathbf{u}}_{h} & = &\sum_{j=1}^{N_h}\sum_{i=1}^{N_h}\hat{x}_ja_{ji}\hat{x}_i\\
& = &{\cal A}\left(\sum_{j=1}^{N_h}\hat{x}_jv_j,\sum_{i=1}^{N_h}\hat{x}_iv_i\right)\\
& = & {\cal A}(\hat{v},\hat{v})\geq \beta\|\hat{v}\|_V^2>0,
\end{eqnarray*}
which was to be shown.
\end{proof}
\begin{corollary}{\bf (Invertibility $\mathbf{A}$)}
The stiffness matrix $\mathbf{A}$ of the discrete problem
\eqref{matrixproblem} is nonsingular.
\end{corollary}
\begin{proof}
See \cite{Solin-book2005}.
\end{proof}
Finally we conclude that the matrix system \eqref{matrixproblem} has an unique solution $\mathbf{u}_{h}$ that defines  a unique solution $u_{h}\in V_{h}$ of \eqref{gral Galerkin}.

\section{Orthogonality of the Error (Cea's Lemma)}

In this section present a brief introduction of the analysis of  the error of the solution. The error in particular has the next orthogonality property.

\begin{lemma}{\bf (Orthogonality of the error)}
Let $u\in V$ the exact solution of the continuous problem \eqref{Continuousproblem} and $u_{h}$ the discrete solution of the problem \eqref{gral Galerkin}. Then, the error $e_{h}=u-u_{h}$ satisfy
\begin{equation}\label{ortho-error}
{\cal A}\left(u-u_{h},v_{h}\right)=0,\,\mbox{ for all  }v_{h}\in V_{h}.
\end{equation}
\end{lemma}

\begin{proof}
By subtracting formulas \eqref{Continuousproblem} and \eqref{gral Galerkin} and using the fact $V_h\subset V$, we get
\begin{eqnarray*}
{\cal A}(u,v)-{\cal A}(u_h,v) & = & {\cal F}(v)-{\cal F}(v), \quad\mbox{for all }v\in V_h\\
{\cal A}(u,v)-{\cal A}(u_h,v) & = & 0,\quad\mbox{for all }v\in V_h,
\end{eqnarray*}
by the $V$-ellipticity of the bilinear form ${\cal A}$ results
\[
{\cal A}(u-u_h,v)=0,\quad\mbox{for all }v\in V_h,
\]
which was to be shown.
\end{proof}

For more details about of the error estimates, see e.g. \cite{babuvska1971error,MR2373954,Solin-book2005}. We show the important result about the error estimate.

\begin{theorem}\label{(LemmaCea)}{\bf (C\'ea's lemma)}
Let $V$ be a Hilbert space ${\cal A}(\cdot,\cdot):V\times V\to\R$ a bilinear $V$-elliptic form  and  ${\cal F}\in V'$. Let $u\in V$ be the solution to the problem \eqref{Continuousproblem}. Also let $V_{h}$ be the subspace of $V$ and $u_{h}\in V_{h}$ the solution of the Galerkin approximation \eqref{gral Galerkin}. Let $\alpha,\beta$ be  constants of the continuity and $V$-ellipticity
of the bilinear form ${\cal A}(\cdot,\cdot)$. then
\[
\left\Vert u-u_{h}\right\Vert _{V}\leq\frac{\alpha}{\beta}\inf_{v_{h}\in V_{h}}\left\Vert u-v\right\Vert _{V}.
\]
\end{theorem}

\begin{proof}
Using the orthogonality given in \eqref{ortho-error}, we obtain
\begin{eqnarray*}
{\cal A}(u-u_{h},u-u_{h}) & = & {\cal A}(u-u_{h},u-v_{h})-{\cal A}(u-u_{h},u_{h}-v_{h})\\
& = & {\cal A}(u-u_h,u-v_h)-{\cal A}(u,u_h-v_h)+{\cal A}(u_h,u_h-v_h)\\
& = & {\cal A}(u-u_h,u-v_h)-{\cal F}(u_h-v_h)+{\cal F}(u_h-v_h)\\
& = & {\cal A}(u-u_{h},u-v_{h}),
\end{eqnarray*}
for any $v_{h}\in V_{h}$. By the ellipticity of the bilinear form ${\cal A}(\cdot,\cdot)$ we get
\begin{equation}\label{eq:1lemmaCea}
\beta\left\Vert u-u_{h}\right\Vert_{V}^{2}\leq {\cal A}(u-u_{h},u-u_{h}).
\end{equation}
The boundedness of ${\cal A}(\cdot,\cdot)$ yields
\begin{equation}\label{eq:2lemmaCea}
{\cal A}(u-u_{h},u-u_{h})\leq\alpha\left\Vert u-u_{h}\right\Vert _{V}\left\Vert u-v_{h}\right\Vert _{V},\,\mbox{ for all }v_h\in V_{h}.
\end{equation}
using conditions \eqref{eq:1lemmaCea} and \eqref{eq:2lemmaCea} we obtain
\[
\left\Vert u-u_{h}\right\Vert _{V}\leq\frac{\alpha}{\beta}\left\Vert u-v_{h}\right\Vert _{V},\mbox{ for all }v_h\in V_{h},
\]
which was what we wanted to show.
\end{proof}

\section{Convergence of Galerkin Approximation}

The convergence of the Galerkin  formulation constitutes the main result of this section. The proof of the following result uses ideas commonly used in the finite element methods and is a immediate consequence of the Theorem \ref{(LemmaCea)}. Then we have the next result.

\begin{theorem}
Let $V$ be a Hilbert space and $V_{1}\subset V_{2}\subset\cdots\subset V$ a subsequence of its finite dimensional subspaces such that
\begin{equation}\label{eq:sequencesVh}
\overline{\bigcup_{j=1}^{\infty}V_{h_j}}=V.
\end{equation}
Let ${\cal A}(\cdot,\cdot):V\times V$ be a $V$-elliptic bounded bilinear form and ${\cal F}\in V'$. Then
\[
\lim_{h\to\infty}\left\Vert u-u_{h}\right\Vert _{V}=0,
\]
i.e., the Galerkin formulation converges to the problem \eqref{Continuousproblem}.
\end{theorem}

\begin{proof}
Let $\{h_{j}\}$ such that $h_{j+1}>h_j$ with $h_j\to 0$ as $j\to \infty$. Given the exact solution $u\in V$ of the problem \eqref{Continuousproblem}, from the equation \eqref{eq:sequencesVh} is possible to find a sequence  $\left\{ v_{j}\right\} _{j=1}^{N_{h}}$ such that $v_{j}\in V_{h_j}$ for all  $j=1,\dots$ and 
\begin{equation}\label{eq:lim vh}
\lim_{j\to\infty}\left\Vert u-v_{h_j}\right\Vert _{V}=0.
\end{equation}
From the Lemma \ref{Lemma  uniqueness} yields the existence and uniqueness of the solution $u_{h_j}\in V_{h_j}$ of the discrete problem \eqref{gral Galerkin} for all $h_j$. Then by the C\'ea's lemma we have
\[
\left\Vert u-u_{h_j}\right\Vert _{V}\leq\frac{\alpha}{\beta}\inf_{v\in V_{h_j}}\left\Vert u-v\right\Vert _{V}\leq\frac{\alpha}{\beta}\left\Vert u-v_{h_j}\right\Vert _{V},\mbox{ for all }j=1,2,\dots.
\]
From the equation \eqref{eq:lim vh} we conclude that
\[
\lim_{h\to\infty}\left\Vert u-u_{h}\right\Vert _{V}=0,
\]
that is what we wanted to show.
\end{proof}
For more details we refer to e.g., \cite{Galvis-book2011} and references there in.

\section{Finite Element Spaces}

In this section we consider piecewise linear functions as  finite element spaces. For this we introduce the next definition.

\subsection{Triangulation}

As above let $D\subset \R^d$ for $d=1,2,3$ be a polygonal domain. A triangulation (or mesh) ${\cal T}^{h}$, is a (finite) partition of $D$ in disjoint subsets of $D$ are called   elements, that is ${\cal T}^{h}=\left\{ K_{j}\right\} _{i=j}^{N_{h}}$ with
\[
\bigcup_{j=1}^{N_{h}}\overline{K}_{j}=\overline{D},\quad K_{i}\cap K_{j}=\varnothing,\,\mbox{ for }i\neq j,\,\mbox{ and }h=\max_{1\leq j\leq N_{h}}\diame(K_{j}).
\]
We say that a triangulation is geometrically conforming if the intersection of the closures of two different elements $\overline{K}_{i}\cap\overline{K}_{j}$ with $i\neq j$ is either a common side or a vertex  to the two elements, see Figure \ref{FigureC.1}.\\

\begin{figure}[htb]
\begin{centering}
\subfloat[Conforming]{\begin{centering}
\includegraphics[scale=0.3]{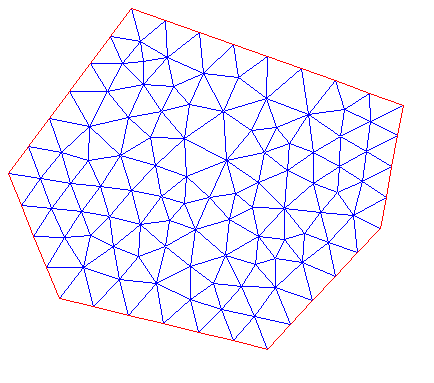}
\par\end{centering}

}\subfloat[Non-conforming]{\begin{centering}
\includegraphics[scale=0.3]{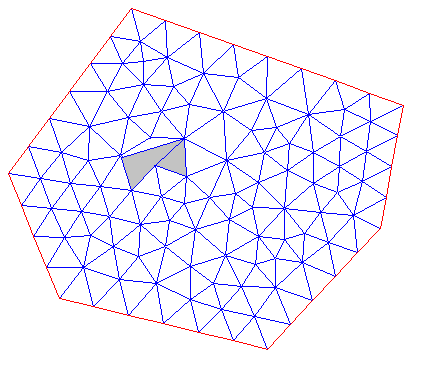}
\par\end{centering}

}
\par\end{centering}

\caption{Example of geometrically conforming or non-conforming triangulation.} \label{FigureC.1}
\end{figure}
For the next definition we refer \cite{Galvis-book2009,Galvis-book2011,toselli2005domain}. The reference triangle $\hat{K}$ in $\R^2$ is a triangle with vertices $(0,0),\,(0,1)$ and  $(1,0)$ or tetrahedrons $(0,0,0),\,(0,0,1)$ and $(0,1,0)$ in $\R^3$. In this part the triangulation is formed  by triangles that are images of the affine mapping from $\hat{K}$ onto an element $K$, i.e., for any $K$ of the triangulation there exists a mapping
$f_{K}:\hat{K}\to K$ is defined by
\[
f_{K}\left(\hat{x}\right)=\left[\begin{array}{cc}
b_{11}^{K} & b_{12}^{K}\\
b_{21}^{K} & b_{22}^{K}
\end{array}\right]\left[\begin{array}{c}
\hat{x}_{1}\\
\hat{x}_{2}
\end{array}\right]+\left[\begin{array}{c}
c_{1}^{K}\\
c_{2}^{K}
\end{array}\right].
\]
For each element $K\in{\cal T}^{h}$ defines the aspect ratio
\[
\rho\left(K\right)=\frac{\diame\left(K\right)}{r_{K}},
\]
where $r_{K}$ is the radius of the largest circle contained in $K$.\\

A family of triangulations $\left\{ {\cal T}^{h}\right\} _{h>0}$ it says of regular aspect if there exists a independent constant $C>0$  of $h$ such that $\rho(K)\leq C$ for all element $K\in{\cal T}^{h}$. This family is quasi-uniform if there exists a independent  constant $C$ $h$ such that $\diame\left(K\right)\geq Ch$ for all element $K\in{\cal T}^{h}$.

Given  a triangulation ${\cal T}^{h}$, let $N_{h}^{v}$ be the number of vertices of the triangulation. The triangulation vertices $\left\{ x_{i}\right\} $ are divided in boundary $\left\{ x_{i}\in\partial D\right\}$ and interior $\left\{ x_{i}\in D\right\}$. Furthermore, we define the space of linear continuous functions by parts associated to the triangulation ${\cal T}^{h}$
\[
\mathbb{P}^{1}\left({\cal T}^{h}\right)=\left\{ \begin{array}{ll}
 & \mbox{is the two variables polynomial}\\
v\in C(D), & \mbox{of total grade \ensuremath{1}for all element \ensuremath{K}}\\
 & \mbox{of the triangulation \ensuremath{{\cal T}^{h}}}
\end{array}\right\} 
\]
where $v|_{K}$ denotes the restriction of the function $v$ to the element $K$. We get the next result

\begin{lemma} 
A function $u:\Omega\to\R$ belongs to the space $H^{1}(D)$ if and only if the restriction of $u$ to every $K\in{\cal T}^{h}$ belongs to $H^{1}(K)$ and for each common face (or edge in two dimension)  $\overline{f}=\overline{K}_{1}\cap\overline{K}_{2}$
we have
\[
u|_{k_{1}}=u|_{K_{2}},\quad\mbox{ in }f.
\]
\end{lemma}
The finite element spaces of piecewise linear continuous functions are therefore contained in $H^{1}(D)$. For more details of the triangulations and finite element spaces see, \cite{Galvis-book2011,toselli2005domain}.

\begin{lemma}
The set of piecewise linear function $\mathbb{P}^1\left({\cal T}^h\right)$ is subset of $H^1(D)$. The gradient of the linear function is a piecewise constant vector function in ${\cal T}^h$.
\end{lemma}

\subsection{Interpolation Operator}

In this part we introduce the operator ${\cal I}^h$ which is denoted by
\[
{\cal I}^{h}:C(D)\to V^h.
\]
This operator associates each continuous function $f$ to a new function ${\cal I}^hf\in V^h$. In addition, the operator is useful tool in the Finite Element Method.\\

${\cal I}^h$ is defined as follows: let $v_i$, $i=1,2\dots,N^h$ be the set of vertices of the triangulation ${\cal T}^h$. Then, given a continuous function $f$, ${\cal I}^hf\in V^h$ and ${\cal T}^hf(v_i)=f(v_i)$ for $i=1,2,\dots,N^h$. Note that ${\cal I}^hf$ is a continuous function and is well defined.\\

According Cea's lemma, see Theorem \ref{(LemmaCea)}, allows us to limit the discretization error by $\alpha/\beta \inf_{v_h\in V_h}\|u-v\|_{V}$. We recall the result given in \cite{Galvis-book2011} to obtain an estimative of the error in function of $h$, i.e., the discretization error is ${\cal O}(h)^1$. We consider $v={\cal I}^hu$ in Cea's lemma and use the fact
\[
\inf_{v_h \in V_h}\|u-v\|_V\leq \|{\cal I}^hu-u\|_V.
\]
So we have the following result given in \cite{Galvis-book2011} for the case one or two dimensional.
\begin{lemma}
Let $D=(a,b)$ or $D\subset \R^2$ be a polygonal set. Given a family of triangulations quasi-uniform ${\cal T}^h$ of $D$, let $V_0^1$ be the piecewise linear function associated to ${\cal T}^h$. We have
\[
\|{\cal T}^hu-u\|_{V}\leq ch\|u\|_{H^2(D)},
\]
where $\|\cdot\|_{H^{2}(D)}$ is defined by
\[
\|f\|_{H^2(D)}=\left(\int_Df(x)^2+|\nabla f(x)|^2+\sum_{i,j=1}^{2}\left(\partial_{ij}f\right)^2\right)^{1/2}.
\]
\end{lemma}



\chapter{\textsc{MatLab} Codes} \label{AppendixD}

\minitoc

In this Appendix we present the code in \textsc{MatLab} using PDE toolbox with which we develop the examples in this manuscript.We note that this code (as well as others codes used), is based on a code developed by the authors of \cite{calo2014asymptotic}.\\

\lstset{basicstyle=\footnotesize\ttfamily,breaklines=true}
\lstset{framextopmargin=20pt,numbers=left, numberstyle=\tiny, stepnumber=1, numbersep=-9pt}

\begin{lstlisting}[frame=single]
  	% loading mesh and parameter information
	load newexample05.mat

	ibackground=1;

	NI=30;
	eta=1000;
	
	ref=3;
	for refaux=1:ref
    	[p,e,t]=refinemesh(g,p,e,t) ;

	end

	M=t(1:3,:)';
	vx=p(1,:);
	vy=p(2,:);
	%%%%%%%%%%%%%%%%%%%%%%%%%%%%%%%%%%%%%%
	%%%%%% matrix assemble and interior index
	[Kd,Fd,B,ud]=assempde(b,p,e,t,c,a,f);
	Ifreeg=B'*(1:size(B,1))';
	for i=1:length(Ifreeg)
   		Ifreeginv(Ifreeg(i))=i;
	end
	I=1:length(Fd);

	%%%%%%%%%%%%%%%%%%%%%%%%%%%%%%%%%%%%%%%%%%%%%%%%%5
	%%% extracting background subdomain info
	com=pdesdp(p,e,t);
	imax=max(t(4,:));
	[idb,cdb]=pdesdp(p,e,t,ibackground);

	%%%%%%%%%%%%%%%%%%%%%%%%%%%%%%%%%%%%%%5
	%%%% extracting inclusions meshes
	ido=[]; itinc=[];intido=[];
	basis=0;
	for i=1:imax
    	if i~=ibackground
        	basis=basis+1;
        	[idaux,cdo]=pdesdp(p,e,t,i);
        	ido=[ido,[idaux,cdo]];
        	intido=[intido,ido];
        	it=pdesdt(t,i);
       	itinc=[itinc,it];
    	end
	end
	ido=Ifreeginv(ido);
	intido=Ifreeginv(intido);
	Ifree=setdiff(I,ido);

	%%%%%%%%%%%%%%%%%%%%%%%%%%%%%%%%%%%%%%%%5
	%%%%% Neumann matrix and mass matrix for inclusions
	itbacground=pdesdt(t,ibackground);
	[Kaux,Faux,Baux,aux]=assempde(b,p,e,t(:,itinc),c,a,f);
	[Kaux2,Faux2,Baux2,aux2]=assempde(b,p,e,t(:,itbacground),c,a,f*0);
	[Massaux,Faux,Baux,aux]=assempde(b,p,e,t(:,itinc),0,1,f);
	mulaux=sum(Massaux,2);

	%%%%%%%%%%%%%%%%%%%%%%%%%%%%%%%%%%%%%%%%%%%%
	%%%%% harmonic characteristic functions, coarse matrix	and
	%%%%% lagrange multipliers for Neumann problems
	R=sparse(length(Fd),imax-1);
	basis=0;
	Mul=sparse(length(Fd),imax-1);
	for i1=1:imax
   		 if i1~=ibackground
        	basis=basis+1;
        	[id,cd]=pdesdp(p,e,t,i1);
        	id=Ifreeginv(id);
        	cd=Ifreeginv(cd);
        	u=Fd*0;
        	u(Ifree)=Kd(Ifree,Ifree)\(-Kd(Ifree,[id,cd])*([id,cd]'*0+1));
        	u([id,cd])=1;
        	R(:,basis)=u;
        	cvmul=mulaux*0;
        	cvmul([id,cd])=mulaux([id,cd]);
        	Mul(:,basis)=cvmul;
    	end
	end
	K0=R'*Kd*R;
	%%%%%%%%%%%%%%%%%%%%%%%%%%%%%%%%55
	%%%% computation of u00
	u00=Fd*0;
	u00(Ifree)=Kd(Ifree,Ifree)\(Fd(Ifree));

	%%%%%%%%%%%%%%%%%%%%%%%%%%%%%%%%%%
	%%%% computation of uc=ucoarse
	Fc=R'*(Fd-Kd*u00);
	uc=K0\Fc;
	u0=R*uc+u00;
	%%%%%%%%%%%%%%%%%%%%%%%%%%%%%%%%%%%%%%%%%%%%%%%%%%%%%%%%%5
	%%%%%%%%%%%%% computation of u1%%%%%%%%%%%%%%%%%%%%%%%5
	%%%%%%%%%%%%%%%%%%%%%%%%%%%%%%%%%%%%%%%%%%%%%%%%%%%%%%%%%5
	%%%%%%%%%%%%%%%%%%%%%%%%%%%%%%%%%%%%%%%%%%%%%%%%%%%%%%%%%5

	%%%%%%%%%%
	%%% Neumann problem in inclusions
	u1=Fd*0;
	Ku0=[Fd-Kd*u0;sparse(imax-1,1)];
	KauxMul=[Kaux,Mul;Mul',sparse(imax-1,imax-1)];
	u1aux=[u1;sparse(imax-1,1)];
	idofmul=[ido,length(Fd)+(1:(imax-1))];
	u1aux(idofmul)=KauxMul(idofmul,idofmul)\(Ku0(idofmul));
	u1=u1aux(1:length(Fd));

	%%%%%%%%%%%%%%%%%%%%%%%%%%%%%%%%%%
	%%%% Dirichlet problem on background

	u1tilde=Fd*0;
	u1tilde(Ifree)=Kd(Ifree,Ifree)\(-Kd(Ifree,:)*(u1));
	u1tilde(ido)=u1tilde(ido)+u1(ido);

	%%%%%%%%%%%%%%%%%%%%%%%%%%%%%%
	%%%% correction of constants for next compatibility condition
	%%%% with a coarse problem
	F1c=R'*(-Kd*u1tilde);
	u1c=K0\F1c;
	u1=u1tilde+R*u1c;

	Sol=sparse(length(Fd),NI);
	Sol(:,1)=u0;
	Sol(:,2)=u1;

	%%%%%%%%%%%%%%%%%%%%%%%%%%%%%%%%%%%%%%%%%%%%%%%%%%%
	%%%%%%%%%%%%% computation of the ui, i\geq 1

	uold=u1;
	for k=1:(NI-2)
	    %%%%%%%%%%%%%%%%%%%%%%
	%     uold(intido)=0;
	    %%% Neumann problems on inclusions
	    unew1=Fd*0;
	    unew1aux=[unew1;sparse(imax-1,1)];
	    Kuold=[Kaux2*uold;sparse(imax-1,1)];
	%    Kuold(intido)=0;
	    unew1aux(idofmul)=KauxMul(idofmul,idofmul)\(-Kuold(idofmul));
    	unew1=unew1aux(1:length(Fd));
    
	    %%%%%%%%%%%%%%%%%%%%%%%%%%%%%
	    %%%% dirichlet problem on background
	    unewtilde=Fd*0;    
	    unewtilde(Ifree)=Kd(Ifree,Ifree)\(-Kd(Ifree,ido)*(unew1(ido)));    
	    unewtilde(ido)=unewtilde(ido)+unew1(ido);
    
	    %%%%%%%%%%%%%%%%%%%%%%%%%%%%%%
	    %%%% correction of constants for next compatibility condition
	    %%%% with a coarse problem
	    Fnewc=-R'*(Kd*unewtilde);
	    unewc=K0\Fnewc;    
	    unew=unewtilde+R*unewc;
	    Sol(:,k+2)=unew;
	    uold=unew;
    
	end
	coef=1.00*(t(4,:)==ibackground)+eta*(t(4,:)~=ibackground);

	[Keta,F,B,ud]=assempde(b,p,e,t,coef,a,f);

	[KL,FL]=assempde(b,p,e,t,1,a,f);

	uetaaux=Keta\F;
	ueta=B*uetaaux+ud;

	Sol=B*Sol;
	Sol(:,1)=Sol(:,1)+ud;
	su=0;
	for k=1:NI
    	su=su+(1/eta)^(k-1)*Sol(:,k);
    	error(k)= sqrt((ueta-su)'*KL*(ueta-su));
 	         % % %  figure(1)
	    % % %
	 %   trisurf(M,vx,vy,ueta-su)
	 %   title(k)
	 %    view(2)
	 %    shading interp
	 %    colorbar
	 %    pause(0.2)
	    % %       hold off
	    % %       pause
	end

\end{lstlisting}


\bibliographystyle{amsplain}
\bibliography{MasterBib}
\end{document}